\documentclass[a4paper]{article}

\usepackage{amssymb}
\usepackage{amsmath}
\usepackage{mathtools}
\usepackage[amsmath,thmmarks,hyperref]{ntheorem}
\usepackage{enumerate}	
\usepackage{color}

\usepackage{booktabs}

\usepackage{geometry}
\usepackage[final,notcite,notref]{showkeys}	%
\usepackage[bookmarksnumbered=true,hidelinks]{hyperref}
\usepackage{caption}

\theoremnumbering{arabic}
\theoremstyle{break}
\RequirePackage{latexsym}
\theorembodyfont{\itshape}
\theoremsymbol{}
\theoremheaderfont{\normalfont\bfseries}
\theoremseparator{}
	\newtheorem{theorem}{Theorem}[section]
	\newtheorem{lemma}[theorem]{Lemma}
	
	\newtheorem{corollary}[theorem]{Corollary}

\theorembodyfont{\upshape}
\theoremsymbol{\ensuremath{\clubsuit}}
	\newtheorem{remark}[theorem]{Remark}

\theorembodyfont{\upshape}
\theoremsymbol{}
	\newtheorem{case}{Case}

	\newtheorem{definition}[theorem]{Definition}
	\newtheorem{assumption}{Assumption}

	\usepackage{chngcntr}
	\newcounter{parentnumber}
	
	\makeatletter 
	\newenvironment{subtheorem}[1]
	{%
		\counterwithin*{#1}{parentnumber}
		\def\subtheoremcounter{#1}%
		\refstepcounter{#1}%
		\protected@edef\theparentnumber{\csname the#1\endcsname}%
		\setcounter{parentnumber}{\value{#1}}%
		\setcounter{#1}{0}%
		\expandafter\def\csname the#1\endcsname{\theparentnumber\alph{#1}}%
		\ignorespaces
	}{%
		\setcounter{\subtheoremcounter}{\value{parentnumber}}%
		\counterwithout*{\subtheoremcounter}{parentnumber} % kick it from the reset list
		\ignorespacesafterend
	}
	\makeatother
	
\theoremstyle{nonumberplain}
\theoremheaderfont{\normalfont\bfseries}
\theoremseparator{:}
\theorembodyfont{\normalfont}
\theoremsymbol{\ensuremath{\square}}
\RequirePackage{amssymb}
        \newtheorem{proof}{Proof}
\numberwithin{equation}{section}

\newcommand{\RR}{\mathbb{R}}
\newcommand{\NN}{\mathbb{N}}
\newcommand{\ZZ}{\mathbb{Z}}
\renewcommand{\d}{\mathrm{d}}
\usepackage{color}

\newcommand{\KAT}{\mathcal{T}}
\newcommand{\KAM}{\mathcal{M}}
\newcommand{\bignorm}[1]{\left\| {#1} \right\|}
\newcommand{\norm}[1]{\| {#1} \|}

\newcommand{\vtd}[2]{{\mathbf{VTD}_{#2}^{#1}}}

\newcommand{\CII}{\mathfrak{C}_0}

\newcommand{\CIIpointB}{\widetilde{\mathfrak{C}}_{0,1}}
\newcommand{\CIIpointtB}{\widetilde{\mathfrak{C}}_{0,2}}
\newcommand{\Csup}{\mathfrak{C}_1}
\newcommand{\Cpoint}{\mathfrak{C}_{1,1}}
\newcommand{\Cpointt}{\mathfrak{C}_{1,2}}
\newcommand{\CpointB}{\widetilde{\mathfrak{C}}_{1,1}}
\newcommand{\CpointtB}{\widetilde{\mathfrak{C}}_{1,2}}
\newcommand{\Cdiff}{\mathfrak{C}_2}

\newcommand{\Co}[1]{C{#1}}

\newcommand{\rIf}{r_\If}
\newcommand{\rIfII}{r_{\If}^{\mathscr{I}}}
\newcommand{\rIfIIi}[1]{r_{\If,{#1}}^{\mathscr{I}}}
\newcommand{\rIfInti}[1]{r_{\If,{#1}}^{\int}}
\newcommand{\rIfIIVar}{r_{\mathrm{var}}^{\mathscr{I}\!, \If}}
\newcommand{\rExII}{r_\mathrm{ex}^{\mathscr{I}}}
\newcommand{\rExIf}{r_\mathrm{ex}^\If}

\newcommand{\kIf}{k_\If}
\newcommand{\kII}{k_{\mathscr{I}}}
\newcommand{\kJJ}{k_{\mathcal{J}}}
\newcommand{\KIf}{K^\If}
\newcommand{\KTIf}{\widetilde{K}^\If}
\newcommand{\KII}{K^{\mathscr{I}}}
\newcommand{\KTII}{\widetilde{K}^{\mathscr{I}}}

\DeclareMathOperator{\df}{def}
\DeclareMathOperator{\rf}{rem}

\newcommand{\Id}{\mathrm{Id}}
\newcommand{\If}{\mathcal{I}}
\newcommand{\Ifhat}{\widehat{\mathcal{I}}}
\newcommand{\Ifapp}{\mathcal{I}^\mathrm{app}}

\usepackage{relsize}
\usepackage[scr=rsfso,scrscaled=1.1]{mathalfa}
\newcommand{\II}{
	{\mathchoice
	{%displaystyle
		\raisebox{-0.3mm}{$\displaystyle \mathlarger{{{\mathscr{I}}}}_{\!n}$}
	}
	{%textstyle
		\raisebox{-0.3mm}{$\mathlarger{{{\mathscr{I}}}}_{\!n}$}
	}
	{%scriptstyle
		\mathscr{I}_n
	}
	{%scriptscriptstyle
		\mathscr{I}_n
	}}
}

\newcommand{\IIn}[1]{
	\II\!\left[#1\right]
}
\newcommand{\IIhat}{
	{\mathchoice
	{%displaystyle
		\raisebox{-0.3mm}{$\displaystyle \mathlarger{{{\widehat{\mathscr{I}}}}}$}
	}
	{%textstyle
		\raisebox{-0.3mm}{$\mathlarger{{{\widehat{\mathscr{I}}}}}$}
	}
	{%scriptstyle
		\widehat{\mathscr{I}}
	}
	{%scriptscriptstyle
		\widehat{\mathscr{I}}
	}}
}

\newcommand{\JJhat}[1][\If]{\widehat{\mathcal{J}}^{\mathscr{I}\!, {#1}}}
\newcommand{\JJn}[1][\If]{\mathcal{J}_n^{\mathscr{I}\!, {#1}}}
\newcommand{\JJ}[1][\If]{\mathcal{J}^{\mathscr{I}\!, {#1}}}
\newcommand{\PPhat}[1][\If]{\widehat{\mathcal{P}}^{\mathscr{I}\!, {#1}}}
\newcommand{\PPn}[1][\If]{\mathcal{P}_n^{\mathscr{I}\!, {#1}}}

\title{Unified Analysis for Variational Time Discretizations of Higher Order and
Higher Regularity Applied to Non-stiff ODEs}
\author{Simon Becher, Gunar Matthies \footnote{Institute of Numerical Mathematics,
Technical University of Dresden, Dresden 01062, Germany.
\mbox{e-mail:} Simon.Becher@tu-dresden.de, Gunar.Matthies@tu-dresden.de}}
\date{\today}

\begin{document}

	\maketitle
	
	\begin{abstract}
		We present a unified analysis for a family of variational time discretization methods,
		including discontinuous Galerkin methods and continuous Galerkin--Petrov methods,
		applied to non-stiff initial value problems. Besides the well-definedness of the
		methods, the global error and superconvergence properties are analyzed under
		rather weak abstract assumptions which also allow considerations of a wide
		variety of quadrature formulas. Numerical
		experiments illustrate and support the theoretical results.
	\end{abstract}
	
	%65L05 (Initial value problems), \\
	%65L20 (Stability and convergence of numerical methods),\\
	%65L60 (Finite elements, Rayleigh-Ritz, Galerkin and collocation methods)\\
	

%%%%%%%%%%%%%%%%%%%%%%%%%%%%%%%%%%%%%%%%%%%%%%%%%%%%%%%%%%%%%%%%%%%%%%%%%%%%%%%
%%%%%%%%%%%%%%%%%%%%%%%%%%%%%%%%%%%%%%%%%%%%%%%%%%%%%%%%%%%%%%%%%%%%%%%%%%%%%%%
  \section{Introduction}
	
	We study variational time discretization methods applied to non-stiff
	initial value problems of the form
	\begin{equation}
		\label{initValueProb}
		u'(t) = f\big(t,u(t)\big),\qquad u(t_0) = u_0 \in \RR^d,
	\end{equation}
	where $f$ is supposed to be sufficiently smooth and to satisfy
	a Lipschitz condition on the second variable with constant $L > 0$.
	Here, an initial value $u_0$ is given at $t=t_0$. The system of ordinary
	differential equations (ODEs) is considered on a time interval
	$I=(t_0,t_0+T]$ with arbitrary but fixed length $T > 0$. 
	
	Probably the most popular variational discretization schemes
	for solving~\eqref{initValueProb} numerically are discontinuous
	Galerkin (dG) and continuous Galerkin--Petrov (cGP) methods.
	Both have been studied in the literature for many decades.
	However, an important part for the design of a then fully computable discretization,
	namely the use of quadrature methods for approximate integration, is often
	only marginally considered or the assumptions 
	on the quadrature formulas are quite restrictive and, thus,
	allow a small variability of quadrature formulas only.
	So, for example, the number of quadrature points is supposed to equal
	the dimension of the local ansatz space of the discrete solution,
	cf.~\cite{DD86,DHT81}, or the number of independent variational
	conditions, cf.~\cite{Hul72a,Hul72b}. Often also the investigations are
	restricted to special Gauss, Gauss--Radau, or Gauss--Lobatto formulas.
	A quite general setting is considered in~\cite{ES02} where so-called internal and
	external quadrature formulas are studied. External quadrature means that integrals
	over products of $f$ with test functions are approximated by quadrature
	rules. If $f$ is replaced by a polynomial approximation before (numerical)
	integration then internal quadrature formulas appear. Since in~\cite{ES02} 
	the dynamical behavior of the schemes is analyzed, it is assumed that at
	least for Dahlquist's test problem $u' = \lambda u$ all integrals are integrated
	exactly which again is quite restrictive.
	
	In this paper we try to keep the requirements on the approximation
	operators involved in the definition of the discrete method as low as possible.
	Our scope is to figure out how the approximation properties of these
	(external and internal) operators affect the error behavior
	of the numerical solution. Thereby our studies are not restricted to the
	dG or cGP method but cover the whole family of variational
	discretization methods recently introduced in~\cite{BMW19,BM21}.
	The used assumptions are quite general and abstract in order
	enable the study of a wide variety of methods. Especially, we allow that
	the right-hand side is approximated by a composition of various
	interpolation operators (interpolation cascade) or that the quadrature formulas
	also use derivative values. Therefore all variational methods considered
	in~\cite{BM21} and their convergence results can be verified by a
	rigorous and unified error analysis in the non-stiff case now.
	
	The paper is structured as follows. In Section~\ref{sec:formulation}
	the variational time discretization methods are formulated. The
	well-definedness of the method formulation is studied in
	Section~\ref{sec:existence} where the existence of unique discrete
	solutions is proven under appropriate conditions. Section~\ref{sec:error}
	is devoted to the error analysis. Here, at first, the global error is bounded
	by a sum of rather abstract approximation error terms. Then, in order
	to give a better insight into the error behavior, the convergence orders of these
	approximation errors are estimated in terms of basic quantities
	describing the approximation properties of the involved operators.
	Thereafter, superconvergence results in the time mesh nodes are derived in
	Section~\ref{sec:superconvergence}. Finally, in Section~\ref{sec:numerics}
	we use numerical experiments to illustrate the convergence behavior of
	the variational time discretization methods and show that the proven estimates
	are sharp. 
	
	Notation: In this paper $C$ denotes a generic constant independent of the
	time mesh interval length. To describe the vector-valued case ($d > 1$) in
	an easy way, let $(\cdot,\cdot)$ be the standard inner product and $\|\cdot\|$
	the Euclidean norm on $\RR^d$, $d \in \NN$. Besides let $e_j$, $1 \leq j \leq d$,
	be the $j$th standard unit vector in $\RR^d$.
	
	For an arbitrary interval $J$ and positive integers $q$, the spaces of continuous
	and $k$ times continuously differentiable $\RR^q$-valued functions
	on $J$ are written as $C(J,\RR^q)$ and $C^k(J,\RR^q)$, respectively.
	Furthermore, the space of square-integrable $\RR^q$-valued functions shall be
	denoted by $L^2(J,\RR^q)$ or, for convenience, sometimes also by $C^{-1}(J,\RR^q)$.
	For non-negative integers $s$, we write $P_s(J, \RR^q)$ for the space of $\RR^q$-valued
	polynomials of degree less than or equal to $s$. For $q=1$, we sometimes omit
	$\RR^q$. Further notation will be introduced later at the beginning of the
	sections where it is needed.

%%%%%%%%%%%%%%%%%%%%%%%%%%%%%%%%%%%%%%%%%%%%%%%%%%%%%%%%%%%%%%%%%%%%%%%%%%%%%%%
%%%%%%%%%%%%%%%%%%%%%%%%%%%%%%%%%%%%%%%%%%%%%%%%%%%%%%%%%%%%%%%%%%%%%%%%%%%%%%%
	\section{Formulation of the methods}
	\label{sec:formulation}
	
	The interval $I$ is decomposed by
	\begin{gather*}
		t_0 < t_1 < \dots < t_{N-1} < t_N = t_0+T
	\end{gather*}
	into $N$ disjoint subintervals $I_n := (t_{n-1},t_n]$, where $n = 1,\ldots, N$.
	Furthermore, we set
	\begin{gather*}
		\tau_n := t_n - t_{n-1}, \qquad
		\tau := \max_{1\le n\le N} \tau_n.
	\end{gather*}
	For convenience and to simplify the notation in some proofs, we assume $\tau_n \leq 1$
	which is not really a restriction since we are anyway
	interested in the asymptotic error behavior for $\tau \to 0$.
	For any piecewise continuous functions $v$ we define by
	\begin{gather*}
		v(t_n^+) := \lim_{t\to t_n+0} v(t),\qquad
		v(t_n^-) := \lim_{t\to t_n-0} v(t),\qquad
		[v]_{n} := v(t_n^+) - v(t_n^-)
	\end{gather*}
        the one-sided limits and the jump of $v$ at $t_n$.
	
	We now present a slight generalization of the variational time discretization
	methods $\vtd{r}{k}$ investigated in~\cite{BM21}. Let $r, k \in \ZZ$, $0 \leq k \leq r$.
	Then the local version ($I_n$-problem) of the numerical method reads as follows:
	
	Given $U(t_{n-1}^-)\in\RR^d$, find $U\in P_r(I_n,\RR^d)$ such that
	\begin{subequations}
	\label{eq:locProbGen}
	\begin{align}
		U(t_{n-1}^+)
			& = U(t_{n-1}^-),
			&& \text{if } k \geq 1, 
			\label{eq:locProbGenI} \\
		U^{(i+1)}(t_n^-)
			& = \frac{\d^i}{\d t^i} \Big(f\big(t,U(t)\big)\Big)\Big|_{t=t_n^-},
			&& \text{if } k \geq 2, \, i = 0, \ldots, \left\lfloor\tfrac{k}{2}\right\rfloor - 1,
			\label{eq:locProbGenE} \\
		U^{(i+1)}(t_{n-1}^+)
			& = \frac{\d^i}{\d t^i} \Big(f\big(t,U(t)\big)\Big)\Big|_{t=t_{n-1}^+},
			&& \text{if } k \geq 3, \, i = 0, \ldots, \left\lfloor\tfrac{k-1}{2}\right\rfloor - 1,
			\label{eq:locProbGenA}
	\end{align}
	and
	\begin{equation}
		\label{eq:locProbGenVar}
		\IIn{ \Big( U',\varphi\Big) } + \delta_{0,k}\Big(\big[U\big]_{n-1},\varphi(t_{n-1}^+)\Big) =
		\IIn{ \Big( \If_n f\big(\cdot ,U(\cdot)\big), \varphi \Big) }
		\qquad\forall \varphi\in P_{r-k}(I_n,\RR^d)
	\end{equation}
	\end{subequations}
	where $U(t_0^-) = u_0$ and $\delta_{i,j}$ is the Kronecker symbol.
	
	Here, $\II$ denotes an integrator that typically represents either the
	integral over $I_n$ or the application of a quadrature formula for approximate
	integration. Details will be described later on. Moreover, $\If_n$ could be used
	to model some projection/interpolation of $f$ or the usage of some special
	(internal) quadrature rules even if $\II$ is just the integral.

%%%%%%%%%%%%%%%%%%%%%%%%%%%%%%%%%%%%%%%%%%%%%%%%%%%%%%%%%%%%%%%%%%%%%%%%%%%%%%%
%%%%%%%%%%%%%%%%%%%%%%%%%%%%%%%%%%%%%%%%%%%%%%%%%%%%%%%%%%%%%%%%%%%%%%%%%%%%%%%
	\section{Existence and Uniqueness}
	\label{sec:existence}
	
	First of all, we agree that both $\II$ and $\If_n$ are local versions
	(obtained by transformation) of appropriate linear operators $\IIhat$ and $\Ifhat$
	given on the closure of the reference interval $(-1,1]$. However, $\II$ is an
	approximation of the integral operator while $\If_n$ approximates the identity operator.
	Thus, the transformations scale quite differently. More precisely, let
	\begin{gather}
		\label{def:timeTransfNew}
		T_n:(-1,1]\to I_n,\quad \hat{t}\mapsto \frac{t_n + t_{n-1}}{2}+\frac{\tau_n}{2} \hat{t},
	\end{gather}
	denote the affine transformation that maps the reference interval $(-1,1]$
	on an arbitrary mesh interval $I_n = (t_{n-1},t_n]$. Furthermore, let $\kII$
	and $\kIf$ be the smallest non-negative integers such that $\IIhat$ and $\Ifhat$
	are well-defined for functions in $C^{\kII}([-1,1])$ and $C^{\kIf}([-1,1])$,
	respectively. Then we have for all $\varphi \in C^{\kII}(\overline{I}_n,\RR^d)$
	and for all $\psi \in C^{\kIf}(\overline{I}_n,\RR^d)$ that
	\begin{gather*}
		\IIn{\varphi} = \IIhat\!\left[\varphi \circ T_n\right] \left(T_n\right)' 
			= \frac{\tau_n}{2} \IIhat\!\left[\varphi \circ T_n\right] \qquad \text{and} \qquad
		\If_n \psi = \big(\Ifhat (\psi \circ T_n)\big) \circ T_n^{-1}.
	\end{gather*}
	Of course, these operators act component-wise when applied to vector-valued functions.
	Moreover, we suppose for all non-negative integers $l$ that 
	$\Ifhat \hat{v} \in C^{l}([-1,1])$ holds for all $\hat{v} \in C^{\max\{\kIf,l\}}([-1,1])$,
	i.e., $\Ifhat \hat{v}$ is at least as smooth as $\hat{v}$.
	
	The study of existence and uniqueness of solutions to~\eqref{eq:locProbGen}
	as well as the later error analysis is strongly connected with the following
	operator.
	
	Let $\JJhat : C^{\kJJ+1}([-1,1]) \to P_r([-1,1])$ where
	$\kJJ := \max\left\{\left\lfloor \frac{k}{2} \right\rfloor-1,\,\kII,\,\kIf\right\}$
	be defined by
	\begin{subequations}
	\label{def:JJhat}
	\begin{align}
		\big(\JJhat \hat{v}\big)^{(i)}(-1^+)
			& = \hat{v}^{(i)}(-1^+), \qquad
			&& \text{if } k \geq 1, \, i = 0,\ldots,\left\lfloor \tfrac{k-1}{2} \right\rfloor\!, \\
		\big(\JJhat \hat{v}\big)^{(i)}(+1^-)
			& = \hat{v}^{(i)}(+1^-),
			&& \text{if } k \geq 2, \, i = 1,\ldots,\left\lfloor \tfrac{k}{2} \right\rfloor\!, \\
		& \hspace{-9em} \IIhat\Big[\big(\JJhat \hat{v}\big)' \widehat{\varphi}\Big] + \delta_{0,k} \JJhat \hat{v}(-1^+) \widehat{\varphi}(-1^+) \nonumber \\
		& \hspace{-7.5em} = \IIhat\Big[\big(\Ifhat (\hat{v}')\big) \widehat{\varphi}\Big] + \delta_{0,k} \hat{v}(-1^+) \widehat{\varphi}(-1^+)
			&& \forall \widehat{\varphi} \in P_{r-k}([-1,1]).
	\end{align}
	\end{subequations}
	
	\begin{assumption}
		\label{assIIhat}
		As before let $r,k \in \ZZ$, $0 \leq k \leq r$, be the parameters of the method.
		The integrator $\IIhat$ is such that
		$\widehat{\psi}\in P_{r-\max\{1,k\}}([-1,1])$ and
		\begin{gather*}
			\IIhat\Big[\big(1-\hat{t}\,\big)^{\left\lfloor \frac{k}{2} \right\rfloor}
				\big(1+\hat{t}\,\big)^{\left|\left\lfloor \frac{k-1}{2} \right\rfloor\right|}
				\widehat{\psi} \widehat{\varphi}\Big]
			= 0 \qquad \forall \widehat{\varphi} \in P_{r-\max\{1,k\}}([-1,1])
		\end{gather*}
		imply $\widehat{\psi} \equiv 0$. Note that the absolute value in the
		exponent is needed only for $k=0$.
	\end{assumption}
	
	\begin{lemma}
		\label{le:JJhatWellDef}
		Let $r,k \in \ZZ$, $0 \leq k \leq r$. Suppose that Assumption~\ref{assIIhat} holds. Then $\JJhat$
		given by~\eqref{def:JJhat} is well-defined. If $\Ifhat$ preserves polynomials
		up to degree $\tilde{r}$ then $\JJhat$ preserves polynomials up to degree $\min\{\tilde{r}+1,r\}$.
	\end{lemma}
	\begin{proof}
		We have to study existence and uniqueness of the approximation. 
		We need to determine $r+1$ coefficients of a polynomial in
		$P_r([-1,1])$. To this we have
		\begin{gather*}
			\left(\left\lfloor \tfrac{k-1}{2}\right\rfloor + 1\right) + \left\lfloor \tfrac{k}{2} \right\rfloor + (r-k + 1)
				= \tfrac{k}{2} + \tfrac{k-1}{2} - \tfrac{1}{2} + 2 + r - k = r + 1
		\end{gather*}
		linear conditions. So, it remains to prove that for $\hat{v} \equiv 0$ we obtain $\JJhat \hat{v} \equiv 0$
		as uniquely defined approximation. The conditions with $\hat{v} \equiv 0$ read as follows
		\begin{subequations}
		\begin{align}
			\big(\JJhat \hat{v}\big)^{(i)}(-1^+)
				& = 0,
				&& \text{if } k \geq 1, \, i=0,\ldots,\left\lfloor \tfrac{k-1}{2} \right\rfloor\!, 
				\label{genhelp1} \\
			\big(\JJhat \hat{v}\big)^{(i)}(+1^-)
				& = 0,
				&& \text{if } k \geq 2, \, i=1,\ldots,\left\lfloor \tfrac{k}{2} \right\rfloor\!, 
				\label{genhelp2} \\
			\IIhat\Big[\big(\JJhat \hat{v}\big)' \widehat{\varphi}\Big] + \delta_{0,k}\JJhat \hat{v}(-1^+) \widehat{\varphi}(-1^+)
				& = 0
				&& \forall \widehat{\varphi} \in P_{r-k}([-1,1]).
				\label{genhelp3}
		\end{align}
		\end{subequations}
		
		Now, there are two cases. For $k=0$ testing in~\eqref{genhelp3} with
		$\widehat{\varphi} = (1+\hat{t}\,) \widetilde{\varphi}$, $\widetilde{\varphi} \in P_{r-1}([-1,1])$,
		we obtain $\IIhat\Big[\big(\JJhat \hat{v}\big)' (1+\hat{t}\,) \widetilde{\varphi}\Big] = 0$.
		So, by Assumption~\ref{assIIhat} we get $\big(\JJhat \hat{v}\big)' \equiv 0$. 
		Using this and choosing in~\eqref{genhelp3} the test function $\widehat{\varphi} \equiv 1$,
		it furthermore follows $\JJhat \hat{v}(-1^+) = 0$ which then yields
		$\JJhat \hat{v} \equiv 0$.
		
		Otherwise, for $k \geq 1$ we see from~\eqref{genhelp1}, \eqref{genhelp2} that
		$\big(\JJhat \hat{v}\big)'(\hat{t}\,) = \big(1-\hat{t}\,\big)^{\left\lfloor \frac{k}{2} \right\rfloor}
		\big(1+\hat{t}\,\big)^{\left\lfloor \frac{k-1}{2} \right\rfloor} \widehat{\psi}(\hat{t}\,)$
		with a polynomial $\widehat{\psi} \in P_{r-k}([-1,1])$. Because of Assumption~\ref{assIIhat}
		the variational condition~\eqref{genhelp3} implies that $\widehat{\psi} \equiv 0$.
		Thus, $\big(\JJhat \hat{v}\big)' \equiv 0$. Additionally using~\eqref{genhelp1}
		for $i=0$, it again follows that $\JJhat \hat{v} \equiv 0$.
		
		Hence, either way $\JJhat \hat{v} \equiv 0$ is the unique approximation of $\hat{v} \equiv 0$
		and so in general the approximation is always uniquely determined.
		
		Using the uniqueness of the approximation, the second statement can be easily verified.
	\end{proof}
	
	\begin{lemma}
		\label{le:normdefNew}
		Let $r\in\NN$ and $k\in\ZZ$ such that $0\leq k \leq r$.
		Suppose that Assumption~\ref{assIIhat} holds. Then the mapping
		\begin{gather}
			\label{def:normdef}
			\widehat{\psi} \mapsto \! \left(
				\sum_{i=0}^{\lfloor{\frac{k-1}{2}}\rfloor-1} \left\| \widehat{\psi}^{(i)}(-1^+) \right\|
				+ \sum_{i=0}^{\lfloor{\frac{k}{2}}\rfloor-1} \left\| \widehat{\psi}^{(i)}(+1^-)\right\|
				+ \sum_{i=\delta_{0,k}}^{r-k} \left\| \IIhat\Big[\widehat{\psi}(\,\hat{t}\,) \, (1+\hat{t}\,)^{i} \Big] \right\|
				\right)
		\end{gather}
		defines a norm on $P_{r-1}([-1,1],\RR^d)$.
	\end{lemma}
	\begin{proof}
		The absolute homogeneity and the triangle inequality follow from the linearity
		of the integrator and of the derivatives along with the respective properties
		of the norm $\| \cdot \|$. Finally, to show the positive definiteness, we can
		adapt the arguments that are used in the proof of Lemma~\ref{le:JJhatWellDef}
		to show $\big(\JJhat[\Id] \hat{v}\big)' \equiv 0$ if $\hat{v} \equiv 0$.
		Note that precisely those $r$ conditions of~\eqref{def:JJhat} that uniquely
		define $\big(\JJhat[\Id] \hat{v}\big)'$ also appear on the right-hand side of the
		mapping~\eqref{def:normdef}.
	\end{proof}
	
	\begin{lemma}
		\label{le:normdefDG}
		Let $r \in \NN_0$. Suppose that Assumption~\ref{assIIhat} holds.
		Then the mapping
		\begin{gather}
			\label{def:normdefDG}
			\widehat{\psi} \mapsto 
				\sum_{i=0}^{r} \left\| \IIhat\Big[\widehat{\psi}'(\,\hat{t}\,) \, (1+\hat{t}\,)^{i} \Big]
					+ \delta_{0,i} \widehat{\psi}(-1^+)\right\|
		\end{gather}
		defines a norm on $P_{r}([-1,1],\RR^d)$. % where the integrator is applied component-wise if $d > 1$.
	\end{lemma}
	\begin{proof}
		As in the proof of Lemma~\ref{le:normdefNew}, the absolute homogeneity
		and the triangle inequality are clear. Moreover, by an adaption of the
		$k=0$ case of the proof of Lemma~\ref{le:JJhatWellDef} we get the
		positive definiteness.
	\end{proof}
	
	Recall from~\eqref{def:timeTransfNew} the affine transformation $T_n$
	mapping the reference interval $(-1,1]$ on an arbitrary mesh interval $I_n = (t_{n-1},t_n]$.
	Then, of course, any $\varphi \in P_{r}(I_n,\RR^d)$ is associated with a
	polynomial $\widehat{\varphi}\in P_{r}\big((-1,1],\RR^d\big)$ by
	\begin{gather*}
		\widehat{\varphi}(\,\hat{t}\,) := \big(\varphi \circ T_n \big)(\,\hat{t}\,)
			= \varphi\big(T_n(\,\hat{t}\,)\big)
		\qquad\forall \hat{t}\in(-1,1].
	\end{gather*}
	The chain rule yields
	\begin{gather*}
		\widehat{\varphi}'(\,\hat{t}\,) = \varphi'(t) T_n'(\,\hat{t}\,) = \frac{\tau_n}{2} \varphi'(t)
		\quad \text{with} \quad t = T_n(\,\hat{t}\,) = \frac{t_n + t_{n-1}}{2}+\frac{\tau_n}{2} \hat{t},
	\end{gather*}
	where $\widehat{\varphi}$ and $T_n$ are differentiated with respect to $\hat{t}$
	and $\varphi$ is differentiated with respect to $t$. The inverse of $T_n$
        and its derivative are given by
	\begin{gather*}
		T_n^{-1}(t) = \frac{2t-(t_n+t_{n-1})}{\tau_n}, \qquad
		\left(T_n^{-1}\right)'(t) = \frac{2}{\tau_n}
	\end{gather*}
	for all $t \in I_n$.
	
	\begin{lemma}
		\label{le:normIneq}
		Let $r,k \in \ZZ$, $0\leq k \leq r$. Suppose that Assumption~\ref{assIIhat} holds. 
		Then there is a positive constant $\kappa_{r,k}$, independent of $\tau_n$, such that
		\begin{align*}
			& \int_{I_n} \big\|\varphi'(t)\big\| \, \d t \leq \tau_n \sup_{t \in I_n} \big\|\varphi'(t)\big\| \\
			& \quad \leq \kappa_{r,k} \Bigg(
					\sum_{i=1}^{\lfloor{\frac{k-1}{2}}\rfloor} \!\!\!\left(\tfrac{\tau_n}{2}\right)^{i} \big\| \varphi^{(i)}(t_{n-1}^+)\big\|
					+ \sum_{i=1}^{\lfloor{\frac{k}{2}}\rfloor} \!\left(\tfrac{\tau_n}{2}\right)^{i} \big\| \varphi^{(i)}(t_n^-)\big\|
				+ \sum_{i=\delta_{0,k}}^{r-k} \left\| \IIn{\varphi'(t) \, (1+T_n^{-1}(t))^i } \right\|\!
				\Bigg)
		\end{align*}
		and
		\begin{gather*}
			\sup_{t \in I_n} \big\|\varphi(t)\big\|
			\leq \min\left\{\left\| \varphi(t_{n-1}^+) \right\|, \left\| \varphi(t_{n}^-) \right\| \right\}
				+ \int_{I_n} \big\|\varphi'(t)\big\| \, \d t
		\end{gather*}
		for all $\varphi \in P_{r}(I_n,\RR^d)$.
	\end{lemma}
	\begin{proof}
		Since the statement is obviously true for $r=0$, we can assume $r\ge 1$ in the following.
		By the fundamental theorem of calculus we have for $t \in I_n$ that
		\begin{gather*}
			\varphi(t) = \varphi(t_{n-1}^+) + \int_{t_{n-1}}^t \varphi'(s) \, \d s
				= \varphi(t_{n}^-) - \int_{t}^{t_n} \varphi'(s) \, \d s.
		\end{gather*}
		Therefore, it follows that
		\begin{subequations}
		\begin{align*}
			\sup_{t\in I_n} \big\|\varphi(t)\big\|
				& \leq \big\|\varphi(t_{n-1}^+)\big\| + \sup_{t \in I_n} \bigg\|\int_{t_{n-1}}^t \varphi'(s) \, \d s \,\bigg\|
				\leq \big\|\varphi(t_{n-1}^+)\big\| + \int_{I_n} \big\| \varphi'(s) \big\| \, \d s
			\shortintertext{and analogously}
			\sup_{t\in I_n} \big\|\varphi(t)\big\|
				& \leq \big\|\varphi(t_{n}^-)\big\| + \int_{I_n} \big\| \varphi'(s) \big\| \, \d s
		\end{align*}
		\end{subequations}
		which already gives the second statement.
		
		On the finite dimensional function space $P_{r-1}\big((-1,1],\RR^d\big)$, both
		\begin{align*}
			\widehat{\psi} \mapsto \sup_{\hat{t} \in (-1,1]} \big\|\widehat{\psi}(\,\hat{t}\,)\big\|
			\qquad\text{and}\qquad
			\eqref{def:normdef}
		\end{align*}
		define equivalent norms. Hence, for all $\widehat{\psi} \in P_{r-1}\big((-1,1],\RR^d\big)$ we have
		\begin{align*}
			\sup_{\hat{t} \in (-1,1]} \big\|\widehat{\psi}(\,\hat{t}\,)\big\|
			\leq \frac{\kappa_{r,k}}{2} \left(
				\sum_{i=0}^{\lfloor{\frac{k-1}{2}}\rfloor-1} \left\| \widehat{\psi}^{(i)}(-1^+) \right\|
				+ \sum_{i=0}^{\lfloor{\frac{k}{2}}\rfloor-1} \left\| \widehat{\psi}^{(i)}(+1^-)\right\|
				+ \sum_{i=\delta_{0,k}}^{r-k} \left\| \IIhat\Big[\widehat{\psi}(\,\hat{t}\,) \, (1+\hat{t}\,)^{i} \Big] \right\|
				\right)
		\end{align*}
		where $\kappa_{r,k}$ is a positive constant that is independent of $\tau_n$.
		
		So, using the transformation $T_n$ given in~\eqref{def:timeTransfNew}, the last estimate
		together with integral transformations (which analogously also hold for the integrator)
		and the chain rule imply that
		\begin{align*}
			& \int_{I_n} \big\|\varphi'(t)\big\| \, \d t
				\leq \tau_n \sup_{t \in I_n} \big\| \varphi'(t) \big\|
				= \tau_n \sup_{\hat{t} \in (-1,1]} \big\| \tfrac{2}{\tau_n} \widehat{\varphi}'(\,\hat{t}\,) \big\|
				= 2 \sup_{\hat{t} \in (-1,1]} \big\| \widehat{\varphi}'(\,\hat{t}\,) \big\| \\
			& \quad \leq {\kappa}_{r,k} \!\left(
				\sum_{i=1}^{\lfloor{\frac{k-1}{2}}\rfloor} \left\| \widehat{\varphi}^{(i)}(-1^+) \right\|
				+ \sum_{i=1}^{\lfloor{\frac{k}{2}}\rfloor} \left\| \widehat{\varphi}^{(i)}(+1^-)\right\|
				+ \sum_{i=\delta_{0,k}}^{r-k} \left\| \IIhat\Big[\widehat{\varphi}'(\,\hat{t}\,) \, (1+\hat{t}\,)^{i} \Big] \right\|
				\right) \\
%			& \quad = {\kappa}_{r,k} \!\left(
%				\sum_{i=1}^{\lfloor{\frac{k-1}{2}}\rfloor} \!\left(\tfrac{\tau_n}{2}\right)^{i} \left\| \varphi^{(i)}(t_{n-1}^+)\right\|
%				+ \sum_{i=1}^{\lfloor{\frac{k}{2}}\rfloor} \!\left(\tfrac{\tau_n}{2}\right)^{i} \left\| \varphi^{(i)}(t_n^-)\right\|
%				+ \sum_{i=\delta_{0,k}}^{r-k} \left\| \IIn{\varphi'(t) \tfrac{\tau_n}{2} \, (1+T_n^{-1}(t))^i (T_n^{-1})' } \right\|
%				\right) \\
			& \quad = {\kappa}_{r,k} \!\left(
				\sum_{i=1}^{\lfloor{\frac{k-1}{2}}\rfloor} \!\left(\tfrac{\tau_n}{2}\right)^{i} \left\| \varphi^{(i)}(t_{n-1}^+)\right\|
				+ \sum_{i=1}^{\lfloor{\frac{k}{2}}\rfloor} \!\left(\tfrac{\tau_n}{2}\right)^{i} \left\| \varphi^{(i)}(t_n^-)\right\|
				+ \sum_{i=\delta_{0,k}}^{r-k} \left\| \IIn{\varphi'(t) \, (1+T_n^{-1}(t))^i } \right\|
				\right)
		\end{align*}
		for all $\varphi \in P_{r}\big(I_n,\RR^d\big)$. This completes the proof.
	\end{proof}
	
	\begin{lemma}
		\label{le:normIneqDG}
		Let $r \in \NN_0$. Suppose that Assumption~\ref{assIIhat} holds. 
		Then there is a positive constant $\kappa_{r}$, independent of $\tau_n$, such that
		\begin{gather*}
			\sup_{t \in I_n} \big\|\varphi(t)\big\|
			\leq \kappa_{r} \sum_{i=0}^{r} \left\| \IIn{\varphi'(t) \, (1+T_n^{-1}(t))^i } + \delta_{0,i} \varphi(t_{n-1}^+)\right\|
		\end{gather*}
		for all $\varphi \in P_{r}(I_n,\RR^d)$.
	\end{lemma}
	\begin{proof}
		Similar to the proof of Lemma~\ref{le:normIneq} the statement follows
		from Lemma~\ref{le:normdefDG} using the norm equivalence on the finite
		dimensional space $P_{r}\big((-1,1],\RR^d\big)$ along with transformations
		via $T_n$ given in~\eqref{def:timeTransfNew} between the reference interval $(-1,1]$
		and the actual interval $I_n$.
	\end{proof}
	
	In the following, we assume that the inverse inequalities
	\begin{gather}
		\label{ieq:invIneq}
		\sup_{t \in I_n}\big\| V^{(l)}(t) \big\| \leq C_\mathrm{inv} \left(\tfrac{\tau_n}{2}\right)^{-l} \sup_{t \in I_n} \big\| V(t) \big\|
	        \qquad \forall V\in P_r(I_n,\RR^d),\: l\in\NN,
	\end{gather}
	hold true where $C_\mathrm{inv} > 0$ is independent of $\tau_n$ but
	may depend on $l$ and $r$. The proof is standard and uses transformations to
	together with norm equivalences on finite dimensional spaces on the reference interval.
	
	For the proof of the existence and uniqueness of solutions of~\eqref{eq:locProbGen}
	we need some more assumptions.
	
	\begin{assumption}
		\label{assSupII}
		The reference integrator $\IIhat$ satisfies
		\begin{equation*}
			\left\| \IIhat\!\left[\widehat{\varphi}\right] \right\|
				\leq \CII \sum_{j=0}^{\kII} \sup_{\hat{t} \in [-1,1]} \big\| \widehat{\varphi}^{(j)}(\,\hat{t}\,)\big\|
			\qquad \forall \widehat{\varphi} \in C^{\kII}([-1,1],\RR^d)
		\end{equation*}
		where as before $\kII \geq 0$ is the smallest integer such that $\IIhat$
		is well-defined on $C^{\kII}([-1,1])$. This means by transformation that
		\begin{equation*}
			\left\| \IIn{ \varphi }\right\|
				\leq \CII \frac{\tau_n}{2} \sum_{j=0}^{\kII} \left(\tfrac{\tau_n}{2}\right)^j
					\sup_{t \in I_n} \big\| \varphi^{(j)}(t)\big\|
			\qquad \forall \varphi \in C^{\kII}(\overline{I}_n,\RR^d)
		\end{equation*}
		holds for the local integrators $\II$, $1 \leq n \leq N$.
	\end{assumption}
	
	\begin{assumption}
		\label{assSup}
		Let $0 \leq l \leq \kII$. The reference approximation operator
		$\Ifhat$ satisfies
		\begin{equation*}
			\sup_{\hat{t} \in [-1,1]} \big\| (\Ifhat \widehat{\varphi} )^{(l)}(\,\hat{t}\,)\big\|
				\leq \Csup \sum_{j=0}^{\max\{\kIf,l\}}
					\sup_{\hat{t} \in [-1,1]} \big\| \widehat{\varphi}^{(j)}(\,\hat{t}\,)\big\|
			\qquad \forall \widehat{\varphi} \in C^{\max\{\kIf,l\}}([-1,1],\RR^d)
		\end{equation*}
		where as before $\kIf \geq 0$ is the smallest integer such that $\Ifhat$
		is well-defined on $C^{\kIf}([-1,1])$.
		This means by transformation that
		\begin{equation*}
			\sup_{t \in I_n} \big\| (\If_n \varphi )^{(l)}(t)\big\|
				\leq \Csup \sum_{j=0}^{\max\{\kIf,l\}} \left(\tfrac{\tau_n}{2}\right)^{j-l}
					\sup_{t \in I_n} \big\| \varphi^{(j)}(t)\big\|
			\qquad \forall \varphi \in C^{\max\{\kIf,l\}}(\overline{I}_n,\RR^d)
		\end{equation*}
		holds for the local approximation operators $\If_n$, $1 \leq n \leq N$.
	\end{assumption}
	
	\begin{assumption}
		\label{assDiff}
		For $0 \leq i \leq \kJJ = \max\left\{\left\lfloor \tfrac{k}{2} \right\rfloor -1, \kII, \kIf \right\}$,
		it holds for sufficiently smooth functions $v, w$ that
		\begin{gather*}
			\left\| \frac{\d^i}{\d t^i} \Big(f\big(t,v(t)\big)-f\big(t,w(t)\big)\Big)\Big|_{t=s}\right\|
				\leq \Cdiff \, \sum_{l=0}^i \big\| (v-w)^{(l)}(s)\big\|
			\qquad \text{for } \mathrm{a.e.}\, s \in \overline{I} = [t_0,t_0+T].
		\end{gather*}
		Here $\Cdiff$ depends on $\kJJ$ and $f$.
	\end{assumption}
	
	\begin{remark}
		Sufficient conditions for Assumption~\ref{assDiff} would be, e.g.,
		\begin{enumerate}[(i)]
			\item for $\kJJ = 0$: $f$ satisfies a Lipschitz condition on
				the second variable with constant $L>0$,
			\item for $\kJJ \geq 1$: $f$ is affine linear in $u$, i.e., $f(t,u(t)) = A(t)u(t)+b(t)$, and
				$\left\| A(\cdot) \right\|_{C^{\kJJ}} < \infty$. Then the inequality follows from
				Leibniz' rule for the $i$th derivative.
			\item In the literature, see e.g.~\cite[p.~74]{Hen62}, there also
				appear conditions of the form
				\begin{gather*}
					\sup_{t \in I, \, y \in \RR^d} \left\| \frac{\partial}{\partial y} f^{(i)}(t,y) \right\| < \infty,
					\qquad 0 \leq i \leq \kJJ,
				\end{gather*}
				where $f^{(i)}$ denotes the $i$th total derivative of $f$ with respect to $t$ in the sense
				of~\cite[p.~65]{Hen62}. These conditions may be weaker in some cases.
		\end{enumerate}
		Since in general the constant $\Cdiff$ is somewhat connected to the Lipschitz
		constant and, thus, to the stiffness of the ode system, the dependence
		on this constant shall be studied very thoroughly in the analysis.
	\end{remark}
	
	Now, we shall investigate the solvability of the local problem~\eqref{eq:locProbGen}.
	
	\begin{theorem}[Existence and uniqueness]
		\label{th:solvableInt_genNew}
		Let $r,k \in \ZZ$, $0 \leq k \leq r$. We suppose that Assumptions~\ref{assIIhat},
		\ref{assSupII}, \ref{assSup}, and~\ref{assDiff} hold. Then there is a constant
		$\gamma_{r,k}>0$ multiplicatively depending on $\Cdiff^{-1}$ but independent of $n$
		%%% depending on $r$, $k$, $\kIf$, $\Cdiff$, $\Csup$, and $C_\mathrm{inv}$
		such that problem~\eqref{eq:locProbGen} has a unique solution for all $1 \leq n \leq N$
		when $\tau_n \le \gamma_{r,k}$.
	\end{theorem}
	\begin{proof}
		Since the single local problems can be solved one after another,
		it suffices to prove that~\eqref{eq:locProbGen} determines a unique
		solution for a given initial value $U_{n-1}^-$.
		
		In order to show this, we use the auxiliary mapping $g : P_r(I_n,\RR^d)\to
		P_r(I_n,\RR^d)$ given by
		\begin{subequations}
		\label{g_gen}
		\begin{align}
			g(U)(t_{n-1}^+)
				& = U_{n-1}^-, 
				&& \text{if } k \geq 1, \label{gA_gen} \\
			g(U)^{(i+1)}(t_n^-)
				& = \frac{\d^i}{\d t^i} \Big(f\big(t,U(t)\big)\Big)\Big|_{t=t_n^-},
				&& \text{if } k \geq 2, \: i = 0, \ldots, \left\lfloor\tfrac{k}{2}\right\rfloor - 1, \label{gDerE_gen}\\
			g(U)^{(i+1)}(t_{n-1}^+)
				& = \frac{\d^i}{\d t^i} \Big(f\big(t,U(t)\big)\Big)\Big|_{t=t_{n-1}^+},
				&& \text{if } k \geq 3, \: i = 0, \ldots, \left\lfloor\tfrac{k-1}{2}\right\rfloor - 1, \label{gDerA_gen}
		\end{align}
		and
		\begin{multline}
			\label{gVar_gen}
			\IIn{ \Big( g(U)'(t),\varphi(t)\Big) } + \delta_{0,k} \Big(g(U)(t_{n-1}^+),\varphi(t_{n-1}^+)\Big) \\
			= \IIn{ \Big( \If_n f\big(t,U(t)\big), \varphi(t)\Big) } + \delta_{0,k} \Big(U_{n-1}^-,\varphi(t_{n-1}^+)\Big)
			\qquad\forall \varphi\in P_{r-k}(I_n,\RR^d).
		\end{multline}
		\end{subequations}
		Since all these conditions are linear, $g(U)$ is uniquely determined
		by~\eqref{g_gen} if and only if we have the unique solvability of the correspondent
		homogeneous problem, i.e.,
		find $V \in P_r(I_n,\RR^d)$ such that
		\begin{subequations}
		\label{homogProb_gen}
		\begin{align}
			V(t_{n-1}^+)
				& = 0, 
				&& \text{if } k \geq 1, \label{homgA_gen}\\
			V^{(i+1)}(t_n^-)
				& = 0,
				&& \text{if } k \geq 2, \: i = 0, \ldots, \left\lfloor\tfrac{k}{2}\right\rfloor - 1, \label{homgDerE_gen}\\
			V^{(i+1)}(t_{n-1}^+)
				& = 0,
				&& \text{if } k \geq 3, \: i = 0, \ldots, \left\lfloor\tfrac{k-1}{2}\right\rfloor - 1, \label{homgDerA_gen}
		\end{align}
		and
		\begin{equation}
			\label{homgVar_gen}
			\IIn{ \Big( V'(t),\varphi(t)\Big) } + \delta_{0,k} \Big(V(t_{n-1}^+),\varphi(t_{n-1}^+)\Big) = 0
			\qquad\forall \varphi\in P_{r-k}(I_n,\RR^d).
		\end{equation}
		\end{subequations}
		Analogously to the proof of Lemma~\ref{le:JJhatWellDef}, we obtain that the conditions
		of~\eqref{homogProb_gen} allow the trivial solution $V\equiv 0$ only. So the homogeneous
		problem is uniquely solvable and, thus, the mapping $g$ is well-defined by~\eqref{g_gen}.
				
		It can be easily seen that for given $U_{n-1}^-$ every fixed point of $g$ is a solution
		of the local problem~\eqref{eq:locProbGen} and vice versa. To get the existence of
		a unique fixed point, Banach's fixed point theorem shall be applied. Therefore, we
		will prove for $\tau_n \le \gamma_{r,k}$ with $\gamma_{r,k}>0$ to be defined
		that $g$ is on $\left\{U \in P_r(I_n,\RR^d) : U(t_{n-1}^+) = U_{n-1}^- \text{ if } k \geq 1\right\}$
		a contraction with respect to the supremum norm.
		
		So, let $V,W \in P_r(I_n,\RR^d)$ and additionally assume $V(t_{n-1}^+) = W(t_{n-1}^+) = U_{n-1}^-$ if $k \geq 1$.
		Since $g(V)- g(W) \in P_{r}(I_n,\RR^d)$, Lemma~\ref{le:normIneq} (if $k \geq 1$)
		or Lemma~\ref{le:normIneqDG} (if $k=0$), respectively, yields
		\begin{align}
			& \sup_{t \in I_n} \big\| g(V)-g(W) \big\| \label{ghelp1_gen} \\[-3ex]
			& \quad \leq \bar{\kappa}_{r,k} \Bigg(
				\overbrace{\sum_{i=1}^{\lfloor{\frac{k-1}{2}}\rfloor} \!\left(\tfrac{\tau_n}{2}\right)^{i}
					\left\| \big(g(V)- g(W)\big)^{(i)}(t_{n-1}^+)\right\|}^{\mathrm{(I)}}
				+ \overbrace{\sum_{i=1}^{\lfloor{\frac{k}{2}}\rfloor} \!\left(\tfrac{\tau_n}{2}\right)^{i}
					\left\| \big(g(V)- g(W)\big)^{(i)}(t_n^-)\right\|}^{\mathrm{(II)}} \nonumber \\
			& \hspace{4em} + \underbrace{\sum_{i=0}^{r-k} \left\| \IIn{\big(g(V)- g(W)\big)'(t) \, (1+T_n^{-1}(t))^i } 
				+ \delta_{0,k}\delta_{0,i} \big(g(V)-g(W)\big)(t_{n-1}^+) \right\|}_{\mathrm{(III)}}
				\Bigg) \nonumber
		\end{align}
		where% $\bar{\kappa}_{r,k} := \delta_{0,k} \kappa_r + (1-\delta_{0,k}) \kappa_{r,k}$.
		\begin{equation*}
			\bar{\kappa}_{r,k} := \begin{cases}\kappa_r, & k=0,\\ \kappa_{r,k}, & k>0.\end{cases}
		\end{equation*}
		Here, for $k \geq 1$ we also used that $\big(g(V)- g(W)\big)(t_{n-1}^+) = \big(V- W\big)(t_{n-1}^+) = 0$
		by~\eqref{gA_gen}. Involving~\eqref{gDerE_gen} and~\eqref{gDerA_gen}
		the sums $\mathrm{(I)}$ and $\mathrm{(II)}$ can be rewritten as
		\begin{align*}
			\mathrm{(I)}+ \mathrm{(II)}
				& = \frac{\tau_n}{2} \!
					\left(\sum_{i=0}^{\lfloor{\frac{k-1}{2}}\rfloor-1} \!\left(\tfrac{\tau_n}{2}\right)^{i}
					\left\| \frac{\d^i}{\d t^i} \Big(f\big(t,V(t)\big)-f\big(t,W(t)\big)\Big)\Big|_{t=t_{n-1}^+}\right\| \right. \\
				& \hspace{6em} \left.+ \sum_{i=0}^{\lfloor{\frac{k}{2}}\rfloor-1} \!\left(\tfrac{\tau_n}{2}\right)^{i}
					\left\| \frac{\d^i}{\d t^i} \Big(f\big(t,V(t)\big)-f\big(t,W(t)\big)\Big)\Big|_{t=t_n^-}\right\| \right) \!.
		\end{align*}
		Thus by Assumption~\ref{assDiff} we conclude
		\begin{align*}
			\mathrm{(I)}+ \mathrm{(II)}
				& \leq \Cdiff \frac{\tau_n}{2} \!
					\left(\sum_{i=0}^{\lfloor{\frac{k}{2}}\rfloor-1} \!\left(\tfrac{\tau_n}{2}\right)^{i}
					\sum_{l=0}^i \left\| (V-W)^{(l)}(t_n^-)\right\|
				  + \! \sum_{i=0}^{\lfloor{\frac{k-1}{2}}\rfloor-1} \!\left(\tfrac{\tau_n}{2}\right)^{i}
					\sum_{l=0}^i \left\| (V-W)^{(l)}(t_{n-1}^+)\right\| \right) \\
				& \leq \Cdiff \left\lfloor{\frac{k}{2}}\right\rfloor \frac{\tau_n}{2} \!
					\left(\sum_{l=0}^{\lfloor{\frac{k}{2}}\rfloor-1} \!\left(\tfrac{\tau_n}{2}\right)^{l}
					\left\| (V-W)^{(l)}(t_n^-)\right\|
				  + \sum_{l=0}^{\lfloor{\frac{k-1}{2}}\rfloor-1} \!\left(\tfrac{\tau_n}{2}\right)^{l}
					\left\| (V-W)^{(l)}(t_{n-1}^+)\right\| \right)\!.
		\end{align*}
		The appearing sums can be bounded by the inverse inequalities~\eqref{ieq:invIneq}
		in the following way
		\begin{align*}
			& \sum_{l=0}^{\lfloor{\frac{k}{2}}\rfloor-1} \!\left(\tfrac{\tau_n}{2}\right)^{l} \big\|(V-W)^{(l)}(t_n^-)\big\|
				\leq \sum_{l=0}^{\lfloor{\frac{k}{2}}\rfloor-1} \!\left(\tfrac{\tau_n}{2}\right)^{l} \sup_{t \in I_n}\big\|(V-W)^{(l)}(t)\big\| \\
			& \qquad \leq
					\sum_{l=0}^{\lfloor{\frac{k}{2}}\rfloor-1} \!\left(\tfrac{\tau_n}{2}\right)^{l} C_{\mathrm{inv}} \left(\tfrac{\tau_n}{2}\right)^{-l}
					\sup_{t \in I_n}\big\|(V-W)(t)\big\|
				\leq C_{\mathrm{inv}} \left\lfloor{\tfrac{k}{2}}\right\rfloor \sup_{t \in I_n}\big\|(V-W)(t)\big\|.
		\end{align*}
		The argumentation for the second sum is analogue. Altogether, we obtain
		\begin{gather}
			\label{ghelp2(i)(ii)_gen}
			\mathrm{(I)}+ \mathrm{(II)}
				\leq \Cdiff C_\mathrm{inv} \left\lfloor\tfrac{k}{2}\right\rfloor (k-1) \,
				\frac{\tau_n}{2} \sup_{t\in I_n} \big\|(V-W)(t)\big\|.
		\end{gather}
		
		We now consider the third sum of~\eqref{ghelp1_gen}. Exploiting~\eqref{gVar_gen}
		we gain that
		\begin{gather*}
			\mathrm{(III)}
				= \sum_{i=0}^{r-k}
				\left\| \IIn{\big(\If_n f\big(t,V(t)\big)-\If_n f\big(t,W(t)\big)\big) \, (1+T_n^{-1}(t))^i } \right\|\!
		\end{gather*}
		where we used test functions of the form $\varphi(t) = (1+T_n^{-1}(t))^i e_j$, $j = 1,\ldots, d$,
		in order to derive the component-wise identity needed here. Applying the estimate 
		of Assumption~\ref{assSupII} it follows
		\begin{align*}
			& \left\| \IIn{\big(\If_n f\big(t,V(t)\big)-\If_n f\big(t,W(t)\big)\big) \, (1+T_n^{-1}(t))^i } \right\| \\
			& \quad \leq \CII \frac{\tau_n}{2} \sum_{j=0}^{\kII} \left(\tfrac{\tau_n}{2}\right)^j
				\sup_{t \in I_n} \left\| \frac{\d^j}{\d t^j} \Big(
					\big(\If_n f\big(t,V(t)\big)-\If_n f\big(t,W(t)\big)\big) \, (1+T_n^{-1}(t))^i \Big) \right\| \\
			& \quad \leq \CII \frac{\tau_n}{2} \sum_{j=0}^{\kII} \sum_{l=0}^j \tbinom{j}{l} \left(\tfrac{\tau_n}{2}\right)^j
				\sup_{t \in I_n} \left\| \frac{\d^{l}}{\d t^{l}} \Big(\If_n f\big(t,V(t)\big)-\If_n f\big(t,W(t)\big)\Big) \right\|
				\, \sup_{t \in I_n} \left\| \big((1+T_n^{-1}(t))^i\big)^{(j-l)} \right\|
		\end{align*}
		where Leibniz' rule for $j$th derivative was exploited for the last inequality.
		Obviously, $(1+T_n^{-1}(t))^i$ is a polynomial of degree $i$ in $t$ and one easily shows
		\begin{gather*}
			\big((1+T_n^{-1}(t))^i\big)^{(j-l)} =
			\begin{cases}
				\big(\tfrac{2}{\tau_n}\big)^{j-l} \tfrac{i!}{(i-(j-l))!} (1+T_n^{-1}(t))^{i-(j-l)},
					& 0 \leq j-l \leq i,\\
				0,	& i < j-l.
			\end{cases}
		\end{gather*}
		Moreover, by construction $1+T_n^{-1}(t) \in (0,2]$ for all $t \in I_n$. Hence, we get
		\begin{align*}
			& \left\| \IIn{\big(\If_n f\big(t,V(t)\big)-\If_n f\big(t,W(t)\big)\big) \, (1+T_n^{-1}(t))^i } \right\| \\
			& \quad \leq \CII \frac{\tau_n}{2} \sum_{j=0}^{\kII} \sum_{l=\max\{0,j-i\}}^{j} \tbinom{j}{l} \tfrac{i!}{(i-(j-l))!} 2^{i-(j-l)}
				\left(\tfrac{\tau_n}{2}\right)^{l}
				\sup_{t \in I_n} \left\| \frac{\d^{l}}{\d t^{l}} \Big(\If_n f\big(t,V(t)\big)-\If_n f\big(t,W(t)\big)\Big) \right\| \\
			& \quad = \CII \frac{\tau_n}{2} \sum_{l=0}^{\kII} \left(\tfrac{\tau_n}{2}\right)^{l}
				\sup_{t \in I_n} \left\| \frac{\d^{l}}{\d t^{l}} \Big(\If_n f\big(t,V(t)\big)-\If_n f\big(t,W(t)\big)\Big) \right\|
					\sum_{j=l}^{\min\{l+i,\kII\}} \tbinom{j}{l} \tfrac{i!}{(i-(j-l))!} 2^{i-(j-l)}.
		\end{align*}
		Now, involving Assumption~\ref{assSup} we conclude further that
		\begin{align*}
			& \left\| \IIn{\big(\If_n f\big(t,V(t)\big)-\If_n f\big(t,W(t)\big)\big) \, (1+T_n^{-1}(t))^i } \right\| \\
			& \quad \leq \CII \Csup \frac{\tau_n}{2} \sum_{l=0}^{\kII} \sum_{j=0}^{\max\{\kII,\kIf\}} \left(\tfrac{\tau_n}{2}\right)^{j}
				\sup_{t \in I_n} \left\| \frac{\d^{j}}{\d t^{j}} \Big(f\big(t,V(t)\big)-f\big(t,W(t)\big)\Big) \right\|
					\sum_{j=0}^{\min\{i,\kII-l\}} \tbinom{j+l}{l} \tfrac{i!}{(i-j)!} 2^{i-j} \\
			& \quad \leq \CII \Csup \frac{\tau_n}{2} \sum_{j=0}^{\max\{\kII,\kIf\}} \left(\tfrac{\tau_n}{2}\right)^{j}
				\sup_{t \in I_n} \left\| \frac{\d^{j}}{\d t^{j}} \Big(f\big(t,V(t)\big)-f\big(t,W(t)\big)\Big) \right\|
					\underbrace{\sum_{l=0}^{\kII} \sum_{j=0}^{\min\{i,\kII-l\}} \tbinom{j+l}{l} \tfrac{i!}{(i-j)!} 2^{i-j}}_{= C_{\Sigma,i}}.
		\end{align*}
		The backmost double sum $C_{\Sigma,i}$ can be simplified and estimated for example as follows
		\begin{gather}
			\label{def:CSigma}
			C_{\Sigma,i} = \sum_{l=0}^{\kII} \sum_{j=0}^{\min\{i,\kII-l\}} \tbinom{i}{j} \tfrac{(j+l)!}{l!} 2^{i-j}
				\leq \kII! \sum_{l=0}^{\kII} \sum_{j=0}^{i} \tbinom{i}{j} 2^{i-j}
				= \kII! \sum_{l=0}^{\kII} 3^i
				= (\kII+1)! \, 3^i.
		\end{gather}
		By Assumption~\ref{assDiff} we then obtain that
		\begin{align*}
			\mathrm{(III)}
				& \leq \CII \Csup \Cdiff \frac{\tau_n}{2}
					\sum_{j=0}^{\max\{\kII,\kIf\}} \left(\tfrac{\tau_n}{2}\right)^{j} \sum_{l=0}^j \sup_{t \in I_n} \big\| \big(V-W\big)^{(l)}(t) \big\|
					\sum_{i=0}^{r-k} C_{\Sigma,i} \\
				& = \CII \Csup \Cdiff \frac{\tau_n}{2}
					\sum_{l=0}^{\max\{\kII,\kIf\}} \left(\tfrac{\tau_n}{2}\right)^{l} \sup_{t \in I_n} \big\| \big(V-W\big)^{(l)}(t) \big\| 
					\sum_{j=l}^{\max\{\kII,\kIf\}} \left(\tfrac{\tau_n}{2}\right)^{j-l} \sum_{i=0}^{r-k} C_{\Sigma,i} \\
%				& = \CII \Csup \Cdiff \frac{\tau_n}{2}
%					\sum_{l=0}^{\max\{\kII,\kIf\}} \left(\tfrac{\tau_n}{2}\right)^{l} \sup_{t \in I_n} \big\| \big(V-W\big)^{(l)}(t) \big\|
%					\sum_{j=0}^{\max\{\kII,\kIf\}-l} \left(\tfrac{\tau_n}{2}\right)^{j} \sum_{i=0}^{r-k} C_{\Sigma,i} \\
%				& \leq \CII \Csup \Cdiff \frac{1-\left(\tfrac{\tau_n}{2}\right)^{\max\{\kII,\kIf\}+1}}{1-\left(\tfrac{\tau_n}{2}\right)}
%					(\kII+1)! \frac{\left(3^{r-k+1} -1\right)}{2} \frac{\tau_n}{2}
%					\sum_{l=0}^{\max\{\kII,\kIf\}} \left(\tfrac{\tau_n}{2}\right)^{l} \sup_{t \in I_n} \big\| \big(V-W\big)^{(l)}(t) \big\| \\
				& \leq \CII \Csup \Cdiff (\max\{\kII,\kIf\}+1) (\kII+1)! \frac{\left(3^{r-k+1} -1\right)}{2} \frac{\tau_n}{2}
					\sum_{l=0}^{\max\{\kII,\kIf\}} \left(\tfrac{\tau_n}{2}\right)^{l} \sup_{t \in I_n} \big\| \big(V-W\big)^{(l)}(t) \big\|.
		\end{align*}
		So, applying the inverse inequalities~\eqref{ieq:invIneq} we conclude similar to the estimation
		of~$\mathrm{(I)}$ and~$\mathrm{(II)}$ that
		\begin{gather}
			\label{ghelp2(iii)_gen}
			\mathrm{(III)}
				\leq \CII \Csup \Cdiff C_\mathrm{inv} (\max\{\kII,\kIf\}+1)^2 (\kII+1)! \frac{\left(3^{r-k+1} -1\right)}{2} \frac{\tau_n}{2}
					\sup_{t \in I_n} \big\| \big(V-W\big)(t) \big\|.
		\end{gather}
		
		Now, combining~\eqref{ghelp1_gen}, \eqref{ghelp2(i)(ii)_gen}, and~\eqref{ghelp2(iii)_gen},
		we get
		\begin{gather*}
			\sup_{t \in I_n} \big\| g(V)-g(W) \big\|
				\leq \bar{\kappa}_{r,k} \Cdiff
					\left(\Co{\big(C_\mathrm{inv},k\big)} + \Co{\big(\CII,\Csup,C_\mathrm{inv},\kIf,\kII,r,k\big)} \right)
					\frac{\tau_n}{2} \sup_{t\in I_n} \big\|(V-W)(t)\big\|.
		\end{gather*}
		Hence, for sufficiently small $\tau_n \leq \gamma_{r,k}$ where $\gamma_{r,k}$ multiplicatively
		depends on $\Cdiff^{-1}$ but is independent of $n$, the mapping $g$ is on
		$\left\{U \in P_r(I_n,\RR^d) : U(t_{n-1}^+) = U_{n-1}^- \text{ if } k \geq 1\right\}$
		a contraction with respect to the supremum norm.
		By Banach's fixed point theorem we have the existence of a unique
		fixed point which is just the unique solution of the local
		problem~\eqref{eq:locProbGen}.
	\end{proof}
	
%%%%%%%%%%%%%%%%%%%%%%%%%%%%%%%%%%%%%%%%%%%%%%%%%%%%%%%%%%%%%%%%%%%%%%%%%%%%%%%
%%%%%%%%%%%%%%%%%%%%%%%%%%%%%%%%%%%%%%%%%%%%%%%%%%%%%%%%%%%%%%%%%%%%%%%%%%%%%%%
	\section{Error analysis}
	\label{sec:error}
	
	In order to prove an error estimate we shall reuse the operator $\JJhat$
	introduced in the previous section. It is used to define a local approximation
	on the interval $I_n$.
	\begin{lemma}[Approximation operator]
		\label{le:approxJJn}
		Let $r,k\in\ZZ$ with $0\le k\le r$. Moreover, suppose that
		Assumption~\ref{assIIhat} holds. Then, the operator
		$\JJn:C^{\kJJ+1}(\overline{I}_n,\RR^d) \to P_r(I_n,\mathbb{R}^d)$ given by
		\begin{subequations}
		\label{def:JJn}
		\begin{align}
			(\JJn v)^{(i)}(t_{n-1}^+)
				& = v^{(i)}(t_{n-1}^+),
				&& \text{if } k \geq 1, \, i=0,\ldots,\left\lfloor \tfrac{k-1}{2} \right\rfloor\!,
			\label{def:JJn_A} \\
			(\JJn v)^{(i)}(t_{n}^-)
				& = v^{(i)}(t_{n}^-),
				&& \text{if } k \geq 2, \, i=1,\ldots,\left\lfloor \tfrac{k}{2} \right\rfloor\!,
			\label{def:JJn_E} \\
				& \hspace{-9em} \IIn{ \big((\JJn v)'(t), \varphi(t)\big)} + \delta_{0,k} \JJn v(t_{n-1}^+) \varphi(t_{n-1}^+) && \nonumber \\
				& \hspace{-7.5em} = \IIn{ \big(\If_n (v')(t), \varphi(t)\big)} + \delta_{0,k} v(t_{n-1}^+) \varphi(t_{n-1}^+)
				&& \forall \varphi \in P_{r-k}(I_n,\RR^d)
			\label{def:JJn_Var}
		\end{align}
		\end{subequations}
		is well-defined.
	\end{lemma}
	\begin{proof}
		The existence and uniqueness of the approximation is a direct consequence of
		Lemma~\ref{le:JJhatWellDef}.
	\end{proof}
	
	For convenience, we define for $v \in C^{\kJJ+1}(\overline{I},\RR^d)$
	a global approximation operator $\JJ$ by combining the local
	approximations on the mesh intervals, i.e.,
	$(\JJ v) \big|_{I_n} = \JJn \big(v|_{I_n}\big) \in P_{r}(I_n,\RR^d)$
	for $n = 1,\ldots,N$ and setting $\JJ v(t_0^-) = v(t_0^-) = v(t_0)$.
	Note that even for $k \geq 1$ in general $\JJ v$ is not globally continuous.
	
	\begin{lemma}
		\label{le:stabJJn}
		Let $r,k \in \ZZ$, $0 \leq k \leq r$, and suppose that Assumptions~\ref{assIIhat},
		\ref{assSupII}, and~\ref{assSup} hold. Then for $1 \leq n \leq N$ and
		$0 \leq l \leq r$ we have for all
		$v \in C^{\kJJ+1}(\overline{I}_n,\RR^d)$
		\begin{equation*}
			\sup_{t \in I_n} \big\| (\JJn v )^{(l)}(t)\big\|
				\leq C \sum_{j=0}^{\kJJ+1}
					\left(\tfrac{\tau_n}{2}\right)^{j-l} \sup_{t \in I_n} \big\| v^{(j)}(t)\big\|.
		\end{equation*}
	\end{lemma}
	\begin{proof}
		To derive the estimate, an inverse inequality, which yields the
		factor $\left(\frac{\tau_n}{2}\right)^{-l}$, is applied first.
		The remaining term, $\sup_{t \in I_n} \big\| \JJn v (t)\big\|$,
		can be bounded on the reference interval independent of
		$\tau_n$. The factors $\left(\frac{\tau_n}{2}\right)^{j}$ then
		result from transformation.
	\end{proof}
	
	For the argument used to prove the error estimate we need additional
	assumptions. In detail, compared to Theorem~\ref{th:solvableInt_genNew}
	we replace Assumption~\ref{assSup} by Assumption~\ref{assPoint_a} or~\ref{assPoint_b}
	since derivatives can be handled given in certain points, but not their supremum.
	Furthermore, we exploit in the error analysis an auxiliary interpolation
	operator $\Ifapp$ defined below, see Definition~\ref{def:auxIfapp}, which
	amongst others is based on these assumptions.
	
	For brevity the Assumptions~\ref{assPoint_a} and~\ref{assPoint_b} below shall
	be stated directly for the local operators $\II$ and $\If_n$. However,
	appropriate properties of $\IIhat$ and $\Ifhat$ guarantee these
	assumptions by transformation, cf. Assumptions~\ref{assSupII} and~\ref{assSup}.
		
	\begin{subtheorem}{assumption}
	\label{assPoint}
	\begin{assumption}
		\label{assPoint_a}
		For $1 \leq n \leq N$ and $0 \leq l \leq \kII$ it holds $\If_n \varphi \in C^{l}(\overline{I}_n, \RR^d)$
		and there are disjoint points $\hat{t}_m^{\,\If}$, $m = 0, \ldots,\KIf$, in the
		reference interval $[-1,1]$ such that 
		\begin{equation*}
			\left(\tfrac{\tau_n}{2}\right)^{l} \sup_{t \in I_n} \big\| (\If_n \varphi)^{(l)}(t)\big\|
				\leq \Cpoint \sum_{m=0}^{\KIf} \sum_{j=0}^{\KTIf_m}
					\left(\tfrac{\tau_n}{2}\right)^{j} \big\| \varphi^{(j)}(t_{n,m}^\If)\big\|
					+ \Cpointt \sup_{t \in I_n} \big\| \varphi(t)\big\|
			\qquad \forall \varphi \in C^{\kIf}(\overline{I}_n,\RR^d)
		\end{equation*}
		where $t_{n,m}^\If := \frac{t_n+t_{n-1}}{2} + \frac{\tau_n}{2} \hat{t}_m^{\,\If}$.
		Note that then typically $\kIf = \max\{\KTIf_m : m = 0,\ldots,\KIf\}$.
	\end{assumption}
	
	\begin{assumption}
		\label{assPoint_b}
		For $1 \leq n \leq N$, there are disjoint points
		$\hat{t}_m^{\,\mathscr{I}}$, $m = 0, \ldots,\KII$, in the reference interval
		$[-1,1]$ such that 
		\begin{equation*}
			\left\| \IIn{ \varphi }\right\|
				\leq \CIIpointB \frac{\tau_n}{2} \sum_{m=0}^{\KII} \sum_{j=0}^{\KTII_m}
					\left(\tfrac{\tau_n}{2}\right)^{j} \big\| \varphi^{(j)}(t_{n,m}^{\mathscr{I}})\big\|
					+ \CIIpointtB \frac{\tau_n}{2} \sup_{t \in I_n} \big\| \varphi(t)\big\|
			\qquad \forall \varphi \in C^{\kII}(\overline{I}_n,\RR^d)
		\end{equation*}
		where $t_{n,m}^{\mathscr{I}} := \frac{t_n+t_{n-1}}{2} + \frac{\tau_n}{2} \hat{t}_m^{\,\mathscr{I}}$.
		Note that then typically $\kII = \max\{\KTII_m : m = 0,\ldots,\KII\}$.
		
		Moreover, for $1 \leq n \leq N$ assume that there are disjoint points
		$\hat{t}_m^{\,\If}$, $m = 0, \ldots,\KIf$, in the reference interval $[-1,1]$ such that 
		\begin{multline*}
			\sum_{m=0}^{\KII} \sum_{l=0}^{\KTII_m}
					\left(\tfrac{\tau_n}{2}\right)^{l} \big\| (\If_n \varphi)^{(l)}(t_{n,m}^{\mathscr{I}})\big\|
				+ \sup_{t \in I_n} \big\| \If_n \varphi(t)\big\| \\
			\leq \CpointB \sum_{m=0}^{\KIf} \sum_{j=0}^{\KTIf_m}
					\left(\tfrac{\tau_n}{2}\right)^{j} \big\| \varphi^{(j)}(t_{n,m}^\If)\big\|
				+ \CpointtB \sup_{t \in I_n} \big\| \varphi(t)\big\|
			\qquad \forall \varphi \in C^{\max\{\kII,\,\kIf\}}(\overline{I}_n,\RR^d)
		\end{multline*}
		where $t_{n,m}^\If := \frac{t_n+t_{n-1}}{2} + \frac{\tau_n}{2} \hat{t}_m^{\,\If}$.
	\end{assumption}
	\end{subtheorem}
	
	\begin{remark}
		Assumption~\ref{assPoint_a} is satisfied if $\If_n$ is a polynomial
		approximation operator whose defining degrees of freedom only use
		derivatives in certain points, as for example Hermite interpolation operators.
		Together with Assumption~\ref{assSupII} the term $\left\| \IIn{ \If_n \varphi }\right\|$
		can be estimated by the supremum of $\left\| \varphi\right\|$ and
		certain point values of derivatives of $\varphi$.
		
		However, Assumption~\ref{assPoint_a} is not satisfied if $\If_n = \Id$
		and $\kII > 0$. In order to enable a similar estimate for
		$\left\| \IIn{ \If_n \varphi }\right\|$ also in this case, Assumption~\ref{assPoint_b}
		is formulated. Here the requirements on the integrator $\II$ are increased. Of
		course, the defining degrees of freedom for the integrator now should use
		derivatives in certain points only. In return, the requirements for $\If_n$
		can be weakened such that they are met for example also by $\If_n = \Id$.
	\end{remark}
	
	\begin{definition}[Auxiliary interpolation operator]
		\label{def:auxIfapp}
		For the error estimation we introduce a special Hermite interpolation operator $\Ifapp_n$.
		Concretely, the operator should satisfy the following conditions: $\Ifapp_n$
		preserves derivatives up to order $\left\lfloor{\frac{k}{2}}\right\rfloor -1$ in $t_n^-$ and up to
		order $\left\lfloor{\frac{k-1}{2}}\right\rfloor-1$ in $t_{n-1}^+$, i.e.,
		\begin{gather}
		\label{def:Ifapp1}
		\begin{aligned}
			(\Ifapp_n \varphi)^{(l)}(t_n^-)
				& = \varphi^{(l)}(t_n^-)
				&& \qquad \text{for } 0 \leq l \leq \left\lfloor{\tfrac{k}{2}}\right\rfloor-1, \\
			(\Ifapp_n \varphi)^{(l)}(t_{n-1}^+)
				& = \varphi^{(l)}(t_{n-1}^+)
				&& \qquad \text{for } 0 \leq l \leq \left\lfloor{\tfrac{k-1}{2}}\right\rfloor-1.
		\end{aligned}
		\end{gather}
		Moreover, we demand that
		\begin{subequations}
		\label{def:Ifapp2}
		\begin{alignat}{2}
			(\Ifapp_n \varphi)^{(l)}(t_{n,m}^\If)
				& = \varphi^{(l)}(t_{n,m}^\If)
				&& \quad \qquad \text{for } 0 \leq m \leq \KIf, \, 0 \leq l \leq \KTIf_m, \label{def:Ifapp2a}
		\intertext{%
		with $t_{n,m}^\If := \frac{t_n+t_{n-1}}{2} + \frac{\tau_n}{2} \hat{t}\,_m^\If$
		where the points $\hat{t}\,_m^\If$ are those of Assumption~\ref{assPoint_a}
		or~\ref{assPoint_b}, respectively. If~\eqref{def:Ifapp1} and~\eqref{def:Ifapp2a}
		provide $r^{\mathrm{app}}$ independent interpolation conditions and
		$r^{\mathrm{app}} < r+1$, then we choose $r+1-r^{\mathrm{app}}$ further
		points $\hat{t}\,_{m}^\If \in (-1,1) \setminus \{\hat{t}\,_j^\If : j = 0,\ldots, \KIf\}$,
		$m = \KIf + 1, \ldots, \KIf + r +1- r^{\mathrm{app}}$,
		and demand
		}
			(\Ifapp_n \varphi)(t_{n,m}^\If)
				& = \varphi(t_{n,m}^\If)
				&& \quad \qquad \text{for }\KIf + 1 \leq m \leq \KIf + r+1- r^{\mathrm{app}}
		\end{alignat}
		\end{subequations}
		where again $t_{n,m}^\If := \frac{t_n+t_{n-1}}{2} + \frac{\tau_n}{2} \hat{t}_m^{\,\If}$.
		We agree that $\Ifapp_n$ is applied component-wise to vector-valued functions. So, overall
		the conditions~\eqref{def:Ifapp1} and~\eqref{def:Ifapp2} uniquely define an Hermite-type
		interpolation operator of ansatz order $\max\{r^{\mathrm{app}}-1,r\}$.
	\end{definition}
	
	Now, we are able to prove an abstract error estimate. 
	
	\begin{theorem}
		\label{th:errorInt_genNew2}
		Let $r,k \in \ZZ$, $0\leq k \leq r$. We suppose that Assumptions~\ref{assIIhat},
		\ref{assSupII}, and~\ref{assDiff} hold. Moreover, let Assumption~\ref{assPoint_a}
		or~\ref{assPoint_b} be satisfied.
		Denote by $u$ and $U$ the solutions of~\eqref{initValueProb}
		and~\eqref{eq:locProbGen}, respectively.
		Then we have for $1 \leq n \leq N$, sufficiently
		small $\tau$, and $l = 0,1$ that
		\begin{align*}
			\sup_{t \in I_n} \big\|(u-U)^{(l)}(t)\big\|
				& \leq C \max_{1\leq \nu \leq n}
					\Bigg( \sup_{t\in I_\nu} \bignorm{\big(\Id-\Ifapp_\nu\big)u(t)}
					+ \sum_{j=0}^l \sup_{t \in I_\nu} \big\|\big(u - \JJ_\nu u\big)^{(j)}(t)\big\| \Bigg) \\
			& \qquad + C \max_{1\leq \nu \leq n-1} 	\tau_\nu^{-1} \big\|\big(u - \JJ_\nu u\big)(t_\nu^-)\big\|
		\end{align*}
		where the constants $C$ in general exponentially depend on the product of $T$ and $\Cdiff$.
	\end{theorem}
	\begin{proof}
		In order to shorten the notation of the proof, we set $E_n := \sup_{t\in I_n} \big\|u(t) - U(t)\big\|$.
		Using the approximation $\JJ u$, which has been introduced directly below
		Lemma~\ref{le:approxJJn}, we split the error in two parts
		\begin{gather*}
			\eta(t):=u(t)-\JJ u(t),\qquad
			\zeta(t):=\JJ u(t)-U(t).
		\end{gather*}
		Then the triangle inequality yields
		\begin{gather}
			\label{help:errorSplit_gen}
			E_n = \sup_{t \in I_n} \big\|u(t)-U(t)\big\|
				\leq \sup_{t \in I_n} \big\|\eta(t)\big\| + \sup_{t \in I_n} \big\|\zeta(t)\big\|.
		\end{gather}
		The first term on the right-hand side is the approximation error of the operator $\JJ$ and
		shall be analyzed separately later on.
		It remains to study the second term.
		
		Since $\zeta$ is continuously differentiable on
		each time interval $I_\nu$, $\nu=1,\ldots,N$, we have for $t \in I_n$
		\begin{align*}
			\zeta(t)
				& = \zeta(t_0^-) + \sum_{\nu = 1}^{n-1} \bigg(\int_{I_\nu} \zeta'(s)\,\d s + \big[\zeta\big]_{\nu-1}\bigg)
					+ \bigg(\int_{t_{n-1}}^t \zeta'(s)\,\d s + \big[\zeta\big]_{n-1}\bigg) \\
				& \leq \sum_{\nu=1}^n \bigg( \int_{I_\nu} |\zeta'(s)|\,\d s + \big[\zeta\big]_{\nu-1}\bigg).
		\end{align*}
		Here, the integrals and the modulus have to be applied component-wise. Furthermore,
		$\zeta(t_0^-)=0$ due to $U(t_0^-) = u_0 = u(t_0^-)$ and $\JJ u(t_0^-) = u(t_0^-)$, see~\eqref{eq:locProbGen}
		and directly below Lemma~\ref{le:approxJJn}. Hence, by the triangle inequality it follows
		\begin{align}
			\sup_{t \in I_n} \norm{\zeta(t)}
				\leq \bignorm{\sum_{\nu=1}^n \bigg( \int_{I_\nu} |\zeta'(s)|\,\d s + \big[\zeta\big]_{\nu-1}\bigg)}
				& \leq \sum_{\nu=1}^n \bigg( \bignorm{  \int_{I_\nu} |\zeta'(s)|\,\d s} + \big\|\big[\zeta\big]_{\nu-1}\big\| \bigg) \nonumber \\
				& \leq \sum_{\nu=1}^n \bigg( \int_{I_\nu} \big\|\zeta'(s)\big\|\,\d s + \big\|\big[\zeta\big]_{\nu-1}\big\| \bigg).
					\label{help:errorZeta_gen}
		\end{align}
		
		We start analyzing the integral term of~\eqref{help:errorZeta_gen}.
		For $1 \leq n \leq N$ and because of $\left. \zeta \right|_{I_n} \in P_{r}(I_n,\RR^d)$
		we can apply Lemma~\ref{le:normIneq} to get
		\begin{align}
			\label{help:errorDerZeta_gen}
			& \int_{I_n} \big\| \zeta'(t) \big\| \, \d t \leq \tau_n \sup_{t \in I_n} \big\| \zeta'(t) \big\| \\
			& \quad \leq \kappa_{r,k} \Bigg( 
						\underbrace{\sum_{i=1}^{\lfloor{\frac{k-1}{2}}\rfloor} \!\left(\tfrac{\tau_n}{2}\right)^{i}
							\big\| \zeta^{(i)}(t_{n-1}^+)\big\|
						+ \sum_{i=1}^{\lfloor{\frac{k}{2}}\rfloor} \!\left(\tfrac{\tau_n}{2}\right)^{i}
							\big\| \zeta^{(i)}(t_n^-)\big\|}_{\mathrm{(I)}}
					+ \underbrace{\sum_{i=\delta_{0,k}}^{r-k} \left\| \IIn{\zeta'(t) \, (1+T_n^{-1}(t))^i } \right\|}_{\mathrm{(II)}}
				\Bigg). \nonumber
		\end{align}
		The right-hand side terms are now studied separately.
		
		The summands of $\mathrm{(I)}$ are basically of the form
		$\left(\tfrac{\tau_n}{2}\right)^{i+1} \bignorm{\zeta^{(i+1)}(\tilde{t})}$ where
		$\tilde{t} \in \{t_{n-1}^+,t_n^-\}$, $i \geq 0$. The definitions of $\JJn u$
		and $U$, especially~\eqref{def:JJn_A}, \eqref{def:JJn_E}
		and~\eqref{eq:locProbGenE}, \eqref{eq:locProbGenA}, together with
		the $i$th derivative of the ode system~\eqref{initValueProb} yield for the
		appearing summands
		\begin{gather*}
			\zeta^{(i+1)}(\tilde{t})
				= (\JJn u-U)^{(i+1)}(\tilde{t})
				= (u-U)^{(i+1)}(\tilde{t})
				= \frac{\d^i}{\d t^i} \Big(f\big(t,u(t)\big)-f\big(t,U(t)\big)\Big)\Big|_{t=\tilde{t}}
		\end{gather*}
		for $0 \le i \le \tilde{k}$ where
		\begin{gather*}
			\tilde{k} =  \tilde{k}(\tilde{t}) =
			\begin{cases}
				\left\lfloor{\frac{k}{2}}\right\rfloor-1,		& \text{if } \tilde{t} = t_{n}^-, \\
				\left\lfloor{\frac{k-1}{2}}\right\rfloor-1,		& \text{if } \tilde{t} = t_{n-1}^+,
			\end{cases}
		\end{gather*}
                denotes the decremented upper limits of the sums in $\mathrm{(I)}$
				depending on $\tilde{t}$.
		By Assumption~\ref{assDiff} we obtain
		\begin{align*}
			\sum_{i=0}^{\tilde{k}} \left(\tfrac{\tau_n}{2}\right)^{i+1}
				\bignorm{\zeta^{(i+1)}(\tilde{t})}
			& = \frac{\tau_n}{2} \sum_{i=0}^{\tilde{k}} \left(\tfrac{\tau_n}{2}\right)^{i}
				\bignorm{\frac{\d^i}{\d t^i} \Big(f\big(t,u(t)\big)-f\big(t,U(t)\big)\Big)\Big|_{t=\tilde{t}}} \\
			& \leq \Cdiff \frac{\tau_n}{2} \sum_{i=0}^{\tilde{k}} \left(\tfrac{\tau_n}{2}\right)^{i}
				\sum_{l=0}^i \left\| (u-U)^{(l)}(\tilde{t})\right\| \\
			& \leq \Cdiff \max\{0,(\tilde{k}+1)\} \frac{\tau_n}{2}
				\sum_{l=0}^{\tilde{k}} \left(\tfrac{\tau_n}{2}\right)^{l} \left\| (u-U)^{(l)}(\tilde{t})\right\|.
		\end{align*}
		Exploiting the Hermite interpolation operator $\Ifapp_n$
		that especially preserves derivatives up to order $\left\lfloor{\frac{k}{2}}\right\rfloor -1$
		in $t_n^-$ and up to order $\left\lfloor{\frac{k-1}{2}}\right\rfloor-1$ in $t_{n-1}^+$,
		see~\eqref{def:Ifapp1}, as well as invoking the inverse inequality~\eqref{ieq:invIneq},
		we find for $0 \leq l \leq \tilde{k}$
		\begin{gather}
		\label{help:treatPointDer_gen}
		\begin{aligned}
			\big\|(u-U)^{(l)}(\tilde{t})\big\|
				& = \big\|(\Ifapp_n u-U)^{(l)}(\tilde{t})\big\|
					\leq \sup_{t\in I_n} \big\|(\Ifapp_n u-U)^{(l)}(t)\big\| \\
				& \leq C_{\mathrm{inv}} \left(\tfrac{\tau_n}{2}\right)^{-l} \sup_{t\in I_n} \big\|(\Ifapp_n u-U)(t)\big\| \\[-0.8em]
				& \leq C_{\mathrm{inv}} \left(\tfrac{\tau_n}{2}\right)^{-l}
					\bigg( \sup_{t\in I_n} \big\| (\Ifapp_n u-u)(t) \big\|
						+ \overbrace{\sup_{t\in I_n} \big\| (u-U)(t) \big\|}^{= E_n} \bigg).
		\end{aligned}
		\end{gather}
		Overall this implies
		\begin{gather}
			\label{help:error(I)_gen}
			\mathrm{(I)}
				\leq \Cdiff C_{\mathrm{inv}}
					\underbrace{\big( \max\{0,\lfloor \tfrac{k-1}{2} \rfloor\}^2 + \lfloor \tfrac{k}{2} \rfloor^2 \big)
						}_{\leq \lfloor \frac{k}{2} \rfloor (k-1)}
					\frac{\tau_n}{2} \bigg(\sup_{t\in I_n} \big\| (u-\Ifapp_n u)(t) \big\| + E_n\bigg).
		\end{gather}
		
		In order to bound $\mathrm{(II)}$ in~\eqref{help:errorDerZeta_gen} we first
		of all note that by definition of $\JJn u$, especially~\eqref{def:JJn_Var},
		we gain for all $\varphi\in P_{r-k}(I_n,\mathbb{R}^d)$ that
		\begin{multline*}
			\IIn{ \Big( (\JJn u)'(t), \varphi(t) \Big) } + \delta_{0,k} \Big(\JJn u(t_{n-1}^+), \varphi(t_{n-1}^+)\Big)\\
			\begin{aligned}
				& = \IIn{ \Big( \If_n u'(t), \varphi(t) \Big) } + \delta_{0,k} \Big(u(t_{n-1}^+), \varphi(t_{n-1}^+)\Big) \\
				& = \IIn{ \Big( \If_n f(t,u(t)),\varphi(t)\Big) } + \delta_{0,k} \Big(u(t_{n-1}^-), \varphi(t_{n-1}^+)\Big)
			\end{aligned}
		\end{multline*}
		where the initial value problem~\eqref{initValueProb} and the continuity
		of $u$ are used in the last identity. So, subtracting the variational
		equation of the local problem~\eqref{eq:locProbGenVar} we obtain for
		all $\varphi\in P_{r-k}(I_n,\RR^d)$
		\begin{multline}
		\label{eq:zeta_Var}
			\IIn{ \Big( \zeta'(t),\varphi(t)\Big) } + \delta_{0,k} \Big(\zeta(t_{n-1}^+), \varphi(t_{n-1}^+)\Big) \\
		\begin{aligned}
			& = \IIn{ \Big( (\JJn u)'(t)-U'(t),\varphi(t)\Big) } + \delta_{0,k} \Big(\big(\JJn u - U\big)(t_{n-1}^+), \varphi(t_{n-1}^+)\Big) \\
			& = \IIn{ \Big( \If_n f\big(t,u(t)\big) - \If_n f\big(t,U(t)\big), \varphi(t)\Big) }
				+ \delta_{0,k} \Big(\big(u - U\big)(t_{n-1}^-), \varphi(t_{n-1}^+)\Big).
		\end{aligned}
		\end{multline}
		Using test functions of the form $\varphi(t) = (1+T_n^{-1}(t))^i e_j$, $j = 1,\ldots, d$,
		we get by a component-wise derivation that
		\begin{gather}
		\label{help:errorIntegrator1_gen}
			\mathrm{(II)}
				= \!\sum_{i=\delta_{0,k}}^{r-k}\! \left\| \IIn{\zeta'(t) \, (1+T_n^{-1}(t))^i } \right\|
				= \!\sum_{i=\delta_{0,k}}^{r-k}\!
					\left\| \IIn{ \big( \If_n f\big(t,u(t)\big) - \If_n f\big(t,U(t)\big) \big) \, (1+T_n^{-1}(t))^i } \right\|\!.
		\end{gather}
		
		For the summands on the right-hand side, we consider two different cases. If
		Assumption~\ref{assPoint_a} holds, we first conclude by Assumption~\ref{assSupII}
		and Leibniz' rule for the $j$th derivative that
		\begin{align*}
			& \left\| \IIn{ \big( \If_n f\big(t,u(t)\big) - \If_n f\big(t,U(t)\big) \big) \, (1+T_n^{-1}(t))^i } \right\| \\
			& \quad \leq \CII \frac{\tau_n}{2} \sum_{j=0}^{\kII} \left(\tfrac{\tau_n}{2}\right)^j
				\sup_{t \in I_n} \left\| \frac{\d^j}{\d t^j}
					\Big( \big(\If_n f\big(t,u(t)\big)-\If_n f\big(t,U(t)\big)\big) \, (1+T_n^{-1}(t))^i \Big) \right\| \\
			& \quad \leq \CII \frac{\tau_n}{2} \sum_{j=0}^{\kII} \sum_{l=0}^j \tbinom{j}{l} \left(\tfrac{\tau_n}{2}\right)^j
				\sup_{t \in I_n} \left\| \frac{\d^{l}}{\d t^{l}} \Big(\If_n f\big(t,u(t)\big)-\If_n f\big(t,U(t)\big)\Big) \right\|
				\, \sup_{t \in I_n} \left\| \big((1+T_n^{-1}(t))^i\big)^{(j-l)} \right\| \\
			& \quad = \CII \frac{\tau_n}{2} \sum_{l=0}^{\kII} \left(\tfrac{\tau_n}{2}\right)^{l}
				\sup_{t \in I_n} \left\| \frac{\d^{l}}{\d t^{l}} \Big(\If_n f\big(t,u(t)\big)-\If_n f\big(t,U(t)\big)\Big) \right\|
					\sum_{j=0}^{\min\{i,\kII-l\}} \tbinom{j+l}{l} \tfrac{i!}{(i-j)!} 2^{i-j}
		\end{align*}
		where we refer for details regarding the derivation of the last identity
		to the proof of Theorem~\ref{th:solvableInt_genNew}. Then involving the
		estimate of Assumption~\ref{assPoint_a}, we obtain that
		\begin{multline*}
			\left\| \IIn{ \big( \If_n f\big(t,u(t)\big) - \If_n f\big(t,U(t)\big) \big) \, (1+T_n^{-1}(t))^i } \right\| \\
			\begin{aligned}
			\leq \CII C_{\Sigma,i} \frac{\tau_n}{2}
				\Bigg( & \Cpoint \sum_{m=0}^{\KIf} \sum_{j=0}^{\KTIf_m} \left(\tfrac{\tau_n}{2}\right)^{j}
					\left\| \frac{\d^j}{\d t^j} \Big(f\big(t,u(t)\big)-f\big(t,U(t)\big)\Big)\Big|_{t=t_{n,m}^\If}\right\| \\
				& \qquad + \Cpointt \sup_{t \in I_n} \big\| f\big(t,u(t)\big)-f\big(t,U(t)\big)\big\| \Bigg),
			\end{aligned}
		\end{multline*}
		with $C_{\Sigma,i}$ from~\eqref{def:CSigma}. A quite similar argumentation also
		works when only Assumption~\ref{assPoint_b} holds. In this case we gain
		\begin{align*}
			& \left\| \IIn{ \big( \If_n f\big(t,u(t)\big) - \If_n f\big(t,U(t)\big) \big) \, (1+T_n^{-1}(t))^i } \right\| \\
			& \quad \leq \CIIpointB \frac{\tau_n}{2}
				\sum_{m=0}^{\KII} \sum_{j=0}^{\KTII_m} \left(\tfrac{\tau_n}{2}\right)^{j}
					\left\| \frac{\d^j}{\d t^j} \Big(\big( \If_n f\big(t,u(t)\big) - \If_n f\big(t,U(t)\big) \big) \, (1+T_n^{-1}(t))^i\Big)
						\Big|_{t=t_{n,m}^{\mathscr{I}}}\right\| \\
			& \hspace{16em} + \CIIpointtB \frac{\tau_n}{2}
				\sup_{t \in I_n} \big\| \big( \If_n f\big(t,u(t)\big) - \If_n f\big(t,U(t)\big) \big) \, (1+T_n^{-1}(t))^i \big\| \\
			& \quad \leq \CIIpointB \frac{\tau_n}{2}
				\sum_{m=0}^{\KII} \sum_{l=0}^{\KTII_m} \left(\tfrac{\tau_n}{2}\right)^{l}
					\left\| \frac{\d^l}{\d t^l} \Big(\If_n f\big(t,u(t)\big) - \If_n f\big(t,U(t)\big) \Big)
						\Big|_{t=t_{n,m}^{\mathscr{I}}}\right\|
					\sum_{j=0}^{\min\{i,\KTII_m-l\}} \tbinom{j+l}{l} \tfrac{i!}{(i-j)!} 2^{i-j} \\
			& \hspace{16em} + \CIIpointtB \frac{\tau_n}{2}
				\sup_{t \in I_n} \big\| \big( \If_n f\big(t,u(t)\big) - \If_n f\big(t,U(t)\big) \big) \big\| \, 2^i \\
			& \quad \leq \max\{\CIIpointB \kII! 3^i,\CIIpointtB 2^i \} \frac{\tau_n}{2}
				\Bigg( \CpointB \sum_{m=0}^{\KIf} \sum_{j=0}^{\KTIf_m} \left(\tfrac{\tau_n}{2}\right)^{j}
					\left\| \frac{\d^j}{\d t^j} \Big(f\big(t,u(t)\big) - f\big(t,U(t)\big) \Big)
						\Big|_{t=t_{n,m}^{\If}}\right\| \\
			& \hspace{16em} + \CpointtB \sup_{t \in I_n} \big\| \big( f\big(t,u(t)\big) - f\big(t,U(t)\big) \big) \big\| \Bigg).
		\end{align*}
		So, either way applying Assumption~\ref{assDiff} gives
		\begin{align*}
			& \left\| \IIn{ \big( \If_n f\big(t,u(t)\big) - \If_n f\big(t,U(t)\big) \big) \, (1+T_n^{-1}(t))^i } \right\| \\
			& \quad \leq C \Cdiff \frac{\tau_n}{2}
				\Bigg( \sum_{m=0}^{\KIf} \sum_{j=0}^{\KTIf_m} \left(\tfrac{\tau_n}{2}\right)^{j}
					\sum_{l=0}^j \big\| (u-U)^{(l)}(t_{n,m}^{\If})\big\|
				+ \sup_{t \in I_n} \big\| (u-U)(t) \big\| \Bigg) \\
			& \quad \leq C \Cdiff \frac{\tau_n}{2}
				\Bigg( \max_{0 \leq m \leq \KIf}\{\KTIf_m+1\}\sum_{m=0}^{\KIf} \sum_{l=0}^{\KTIf_m} \left(\tfrac{\tau_n}{2}\right)^{l}
					\big\| (u-U)^{(l)}(t_{n,m}^{\If})\big\|
				+ \sup_{t \in I_n} \big\| (u-U)(t) \big\| \Bigg).
		\end{align*}
		
		The terms that include derivatives can be estimated similar
		to~\eqref{help:treatPointDer_gen} since by definition the interpolation
		operator $\Ifapp_n$ also preserves the particular derivatives at the
		particular points $t_{n,m}^\If$, see~\eqref{def:Ifapp2a}.
		Therefore, altogether this results in
		\begin{equation}
			\label{help:errorIntegrator3_gen}
			\left\| \IIn{ \big( \If_n f\big(t,u(t)\big) - \If_n f\big(t,U(t)\big) \big) \, (1+T_n^{-1}(t))^i } \right\|
				\leq C \Cdiff \frac{\tau_n}{2}
					\bigg( \sup_{t\in I_n} \bignorm{\big(\Id-\Ifapp_n\big)u(t)} + E_n \bigg)
		\end{equation}
		and, hence, together with~\eqref{help:errorIntegrator1_gen}
		we obtain
		\begin{equation}
			\label{help:error(II)_gen}
			\mathrm{(II)}
				\leq C \Cdiff \frac{\tau_n}{2} \bigg( \sup_{t\in I_n} \bignorm{\big(\Id-\Ifapp_n\big)u(t)} + E_n \bigg).
		\end{equation}
		
		So, combining~\eqref{help:errorDerZeta_gen} with~\eqref{help:error(I)_gen}
		and~\eqref{help:error(II)_gen}, we have already shown that
		\begin{equation}
			\label{ieq:zetaDerBound}
			\int_{I_n} \big\| \zeta'(t) \big\| \, \d t
				\leq \tau_n \sup_{t \in I_n} \big\| \zeta'(t) \big\|
				\leq C \Cdiff \frac{\tau_n}{2} \bigg( \sup_{t\in I_n} \bignorm{\big(\Id-\Ifapp_n\big)u(t)} + E_n \bigg)
		\end{equation}
		for all $1 \leq n \leq N$.
		
		Next, we analyze the jump term of~\eqref{help:errorZeta_gen}. First, we have
		a closer look at $\big[\zeta\big]_{n-1}$ for $1 \leq n \leq N$. There are two cases.
		If $k \geq 1$, the discrete solution $U$ is globally continuous due
		to~\eqref{eq:locProbGenI}. So,
		\begin{align*}
			\big[\zeta\big]_{n-1}
				= \big[\JJ u\big]_{n-1} - \big[U\big]_{n-1}
				= \JJ u(t_{n-1}^+) - \JJ u(t_{n-1}^-)
			  & = u(t_{n-1}^+) - \JJ u(t_{n-1}^-) \\
			  & = u(t_{n-1}^-) - \JJ u(t_{n-1}^-)
		\end{align*}
		where also~\eqref{def:JJn_A} and the continuity of $u$ was used.
		If $k=0$ we rewrite the jump term as follows
		\begin{align*}
			\big[\zeta\big]_{n-1}
				& = \IIn{\zeta'(t)} + \zeta(t_{n-1}^+) - \zeta(t_{n-1}^-) - \IIn{\zeta'(t)} \\
				& = \IIn{ \If_n f\big(t,u(t)\big) - \If_n f\big(t,U(t)\big) }
					+ (u - U)(t_{n-1}^-) - (\JJ u - U)(t_{n-1}^-) - \IIn{\zeta'(t)} \\
				& = u(t_{n-1}^-) - \JJ u(t_{n-1}^-)
					+ \IIn{ \If_n f\big(t,u(t)\big) - \If_n f\big(t,U(t)\big) } - \IIn{\zeta'(t)}
		\end{align*}
		where we exploited~\eqref{eq:zeta_Var}. The integrator terms on the right-hand side
		can be bounded using already known estimates. Indeed, \eqref{help:errorIntegrator3_gen}
		with $i=0$ yields
		\begin{gather*}
			\left\| \IIn{ \big( \If_n f\big(t,u(t)\big) - \If_n f\big(t,U(t)\big) \big) } \right\|
				\leq C \Cdiff \frac{\tau_n}{2}
					\bigg( \sup_{t\in I_n} \bignorm{\big(\Id-\Ifapp_n\big)u(t)} + E_n \bigg).
		\end{gather*}
		Furthermore, combining Assumption~\ref{assSupII}, the inverse inequality~\eqref{ieq:invIneq},
		and~\eqref{ieq:zetaDerBound} we gain
		\begin{align*}
			\big\|\IIn{\zeta'(t)}\big\|
				\leq \CII \frac{\tau_n}{2} \sum_{j=0}^{\kII} \left(\tfrac{\tau_n}{2}\right)^j
					\sup_{t \in I_n} \big\| \zeta^{(j+1)}(t)\big\|
			  & \leq \CII C_\mathrm{inv} \left(\kII +1 \right) \frac{\tau_n}{2} \sup_{t \in I_n} \big\| \zeta'(t)\big\| \\
			  & \leq C \Cdiff \frac{\tau_n}{2} \bigg( \sup_{t\in I_n} \bignorm{\big(\Id-\Ifapp_n\big)u(t)} + E_n \bigg).
		\end{align*}
		Hence, in both cases we have
		\begin{align}
			\big\|\big[\zeta\big]_{n-1}\big\|
				& \leq \big\|u(t_{n-1}^-) - \JJ u(t_{n-1}^-)\big\|
					+ \delta_{0,k} \Big( \big\|\IIn{ \If_n f\big(t,u(t)\big) - \If_n f\big(t,U(t)\big) }\big\|
						+ \big\|\IIn{\zeta'(t)}\big\| \Big) \nonumber \\
				& \leq \big\|u(t_{n-1}^-) - \JJ u(t_{n-1}^-)\big\|
					+ \delta_{0,k} C \Cdiff \frac{\tau_n}{2} \bigg( \sup_{t\in I_n} \bignorm{\big(\Id-\Ifapp_n\big)u(t)} + E_n \bigg).
					\label{ieq:zetaJumpBound}
		\end{align}
		
		Summarizing, we get from~\eqref{help:errorSplit_gen}, \eqref{help:errorZeta_gen},
		\eqref{ieq:zetaDerBound}, and~\eqref{ieq:zetaJumpBound} for $1 \leq n \leq N$
		\begin{align*}
			E_n
			  & \leq \sup_{t \in I_n} \big\|\eta(t)\big\| + \sup_{t \in I_n} \big\| \zeta(t) \big\|
				\leq \sup_{t \in I_n} \big\|\eta(t)\big\| +
					\sum_{\nu=1}^n \bigg( \int_{I_\nu} \big\|\zeta'(s)\big\|\,\d s + \big\|\big[\zeta\big]_{\nu-1}\big\| \bigg) \\
			  & \leq \sup_{t \in I_n} \big\|\eta(t)\big\|
			  	+ \underbrace{C \Cdiff}_{=\widetilde{C}}
			  		\sum_{\nu = 1}^n \frac{\tau_\nu}{2} \bigg( \sup_{t\in I_\nu} \bignorm{\big(\Id-\Ifapp_\nu\big)u(t)} + E_\nu \bigg)
			  	+ \sum_{\nu=1}^{n} \big\|u(t_{\nu-1}^-) - \JJ u(t_{\nu-1}^-)\big\|.
		\end{align*}		
		Note that $\widetilde{C}$ is independent of $T$ but especially depends multiplicatively
		on the Lipschitz constant of $f$ (hidden in $\Cdiff$). For $\tau_n$
		sufficiently small ($\widetilde{C} \tau_n/2 < 1$),
		the $E_n$ term of the right-hand side can be absorbed on the left. After
		that the application of a variant of the discrete Gronwall lemma,
		see~\cite[Lemma~1.4.2, p.~14]{QV08}, applied with
		\begin{align*}
			k_s & = \tfrac{\widetilde{C}\tau_s/2}{1-\widetilde{C}\tau/2} (1-\delta_{0,s}), 
				\qquad \phi_s = E_s,
				\qquad g_0 = 0, \\
			p_s & = \textstyle \frac{1}{1-\widetilde{C}\tau/2} \Big(
				\sup_{t \in (t_0,t_{s+1}]} \big\|\eta(t)\big\| - \sup_{t \in (t_0,t_s]} \big\|\eta(t)\big\| \\
				& \hspace{8em} \textstyle + \widetilde{C} \tau_{s+1} \sup_{t\in I_{s+1}} \bignorm{\big(\Id-\Ifapp_{s+1}u\big)u(t)}
				+ {\big\|u(t_{s}^-) - \JJ u(t_{s}^-)\big\|} \Big),
		\end{align*}
		yields
		\begin{align*}
			E_n & \leq \exp\left(\sum_{\nu=1}^{n-1}\frac{\widetilde{C}\tau_\nu/2}{1-\widetilde{C}\tau/2}\right)
					\frac{1}{1-\widetilde{C}\tau/2} \\
				& \qquad \Bigg\{\!\max_{1 \leq \nu \leq n} \bigg(\sup_{t \in I_\nu} \big\|\eta(t)\big\|\bigg)
					+ \sum_{\nu = 1}^n \bigg(\widetilde{C}\frac{\tau_\nu}{2} \sup_{t\in I_\nu} \bignorm{\big(\Id-\Ifapp_\nu\big)u(t)}
					+ \big\|u(t_{\nu-1}^-) - \JJ u(t_{\nu-1}^-)\big\| \bigg) \!\Bigg\} \\
				& \leq \exp\left(\frac{\widetilde{C}(t_{n-1}-t_0)/2}{1-\widetilde{C}\tau/2}\right)
					\frac{1}{1-\widetilde{C}\tau/2} % \left(1+\widetilde{C}(t_{n}-t_0)/2 + (t_{n-1}-t_0) \right)
					\Bigg\{ \widetilde{C}(t_{n}-t_0)/2 \max_{1 \leq \nu \leq n}
						\bigg( \sup_{t\in I_\nu} \bignorm{\big(\Id-\Ifapp_\nu\big)u(t)} \bigg) \\
				& \qquad + \max_{1 \leq \nu \leq n} \bigg( \sup_{t \in I_\nu} \big\|\big(\Id - \JJ_\nu\big)u(t)\big\| \bigg)
					+ (t_{n-1}-t_0) \max_{1 \leq \nu \leq n-1} \bigg( \tau_{\nu}^{-1} \big\|u(t_{\nu}^-) - \JJ_\nu u(t_{\nu}^-)\big\| \bigg) \! \Bigg\}
		\end{align*}
		where we also used that $u(t_{0}^-) = u_0 = \JJ u(t_{0}^-)$.
		This already is the wanted statement for $l = 0$. Especially note the
		exponential dependency on $T$
		and the Lipschitz constant of $f$ (hidden in $\widetilde{C}$, $\Cdiff$).
		
		Finally, we shall derive a bound for the first derivative of the error.
		Recalling the estimate for $\zeta'$ in~\eqref{ieq:zetaDerBound} it follows
		\begin{align*}
			\sup_{t \in I_n} \big\| (u-U)'(t)\big\|
				& \leq \sup_{t \in I_n} \big\| (u-\JJn u)'(t)\big\| + \sup_{t \in I_n} \big\| (\JJn u -U)'(t)\big\| \\
			\quad & \leq \sup_{t \in I_n} \big\| (u-\JJn u)'(t)\big\| + C \Cdiff
				\bigg( \sup_{t\in I_n} \bignorm{\big(\Id-\Ifapp_n\big)u(t)} + E_n \bigg).
		\end{align*}
		Using the already known estimate for $E_n$, the statement for $l=1$ follows.
	\end{proof}
	
	\begin{remark}
		Based on Theorem~\ref{th:errorInt_genNew2} we can also prove abstract estimates
		for higher order derivatives of the error between the solution $u$
		of~\eqref{initValueProb} and the discrete solution $U$ of~\eqref{eq:locProbGen}.
		Of course, we gain that
		\begin{multline*}
			\sup_{t \in I_n} \big\|(u-U)^{(l)}(t)\big\|
				\leq \sup_{t \in I_n} \big\|(u-\JJn u)^{(l)}(t)\big\|
					+ \sup_{t \in I_n} \big\|(\JJn u - U)^{(l)}(t)\big\| \\
                        \begin{aligned}
			& \quad \leq \sup_{t \in I_n} \big\|(u-\JJn u)^{(l)}(t)\big\|
					+ C_\mathrm{inv} \left(\tfrac{\tau_n}{2}\right)^{-l} \sup_{t \in I_n} \big\|(\JJn u - U)(t)\big\| \\
			& \quad \leq \sup_{t \in I_n} \big\|(u-\JJn u)^{(l)}(t)\big\|
					+ C_\mathrm{inv} \left(\tfrac{\tau_n}{2}\right)^{-l} \left(\sup_{t \in I_n} \big\|(\JJn u - u)(t)\big\|
						+ \sup_{t \in I_n} \big\|(u - U)(t)\big\|\right)\!
                        \end{aligned}
		\end{multline*}
		where an inverse inequality was used.
		However, since we only have a non-local error estimate for $\sup_{t \in I_n} \big\| (u-U)(t) \big\|$
		we cannot expect that the inverse of the local time step length can be compensated in general.
		So, usually we additionally need to assume that $\tau_\nu \leq \tau_{\nu+1}$
		for all $\nu$ or alternatively that the mesh is quasi uniform
		($\tau/\tau_\nu \leq C$ for all $\nu$) to obtain a proper estimate.
	\end{remark}
	
	\begin{remark}
		In the proof of Theorem~\ref{th:errorInt_genNew2} stiffness
		of the problem would be critical at several points.
		
		Indeed, for large Lipschitz constants (hidden in $\Cdiff$ and so
		in $\widetilde{C}$) the needed inequality $\widetilde{C} \tau_n/2 < 1$
		would force very small time step lengths. For semidiscretizations in
		space of time-space problems, where the Lipschitz constant is typically
		proportional to $h^{-2}$ with $h$ denoting the spatial mesh parameter,
		this would cause upper bounds on the time step length with respect to $h$
		similar to CFL conditions.
		
		Moreover, since the error constant $C$ exponentially depends on $\Cdiff$,
		this constant would be excessively large for stiff problems such that
		the error estimate would be useless in this case.
	\end{remark}
	
	Of course, Theorem~\ref{th:errorInt_genNew2} provides an abstract bound for the error
	of the variational time discretization method. However, the order of convergence still
	is not clear. Since $\Ifapp$ is an Hermite-type interpolator of polynomial ansatz order
	larger than or equal to $r$ its approximation order is at least $r+1$. It remains
	to prove suitable bounds on the error of the approximation operator $\JJ$.
	
	\begin{definition}[Approximation orders of $\II$ and $\If_n$]
	\label{def:approxOrders}
		Let $\rExII$, $\rExIf$, $\rIf$, and $\rIfIIi{i} \in \NN_0 \cup \{-1,\infty\}$
		denote the largest numbers such that
		\begin{align*}
			\int_{I_n} \varphi(t) \, \d t & = \IIn{\varphi}
				&& \forall \varphi \in P_{\rExII}(I_n), & \qquad \qquad
			\int_{I_n} \varphi(t) \,\d t & = \int_{I_n} \If_n \varphi(t) \, \d t
				&& \forall \varphi \in P_{\rExIf}(I_n),
		\end{align*}
		\begin{align*}
			\varphi &= \If_n \varphi
				&& \forall \varphi \in P_{\rIf}(I_n), & \qquad \qquad
			\IIn{\varphi \psi_i} &= \IIn{(\If_n\varphi) \psi_i}
				&& \forall \varphi \in P_{\rIfIIi{i}}(I_n), \psi_i \in P_{i}(I_n).
		\end{align*}
		Here, $P_{-1}(I_n)$ is interpreted as $\{0\}$, in which case the respective operator 
		does not provide the corresponding approximation property. For convenience, set 
		$\rIfII := \rIfIIi{r-k}$.
		Note that $\rExIf \geq \rIfInti{i} \geq \rIf$
		and $\rIfIIi{i} \geq \rIf$ hold by definition.
	\end{definition}

	Using standard techniques the above quantities can be connected with certain
	approximation estimates. For example, let $\check{r} \in \NN_0$ then
	together with Assumption~\ref{assSupII} we find that for arbitrary
	$\varphi \in C^{\max\{\kII,\min\{\check{r},\rExII+1\}\}}(\overline{I}_n,\RR^d)$
	\begin{equation}
		\label{ieq:IntErrorIIn}
		\left\| \int_{I_n} \varphi(t) \,\d t - \IIn{\varphi} \right\|
			\leq C \frac{\tau_n}{2} \sum_{j=\min\{\check{r},\rExII+1\}}^{\max\{\kII,\min\{\check{r},\rExII+1\}\}}
				\left(\tfrac{\tau_n}{2}\right)^j \sup_{t \in I_n} \big\| \varphi^{(j)}(t)\big\|.
	\end{equation}
	Furthermore, together with Assumption~\ref{assSup} it follows for arbitrary
	$\varphi \in C^{\max\{\kIf,\min\{\check{r},\rExIf+1\}\}}(\overline{I}_n,\RR^d)$
	\begin{equation}
		\label{ieq:IntErrorIfn}
		\left\| \int_{I_n} \varphi(t) - \If_n \varphi(t) \,\d t \right\|
			\leq C \frac{\tau_n}{2} \sum_{j=\min\{\check{r},\rExIf+1\}}^{\max\{\kIf,\min\{\check{r},\rExIf+1\}\}}
				\left(\tfrac{\tau_n}{2}\right)^j \sup_{t \in I_n} \big\| \varphi^{(j)}(t)\big\|.
	\end{equation}
	
	\begin{lemma}
		\label{le:errorJJn}
		Let $r,k \in \ZZ$, $0 \leq k \leq r$, and suppose that Assumptions~\ref{assIIhat},
		\ref{assSupII}, and~\ref{assSup} hold. Furthermore, let $l, \check{r} \in \NN_0$
		and define
		\begin{gather*}
			j_{\mathrm{min},\check{r}} := \min\{\check{r}, r+1, \rIfII+2 \},
			\qquad
			j_{\mathrm{max},\check{r}} := \max\{\kJJ + 1, l, j_{\mathrm{min},\check{r}}\}.
		\end{gather*}
		If $v \in C^{j_{\mathrm{max},\check{r}}}(\overline{I}_n, \RR^d)$ then
		the error estimate
		\begin{gather}
		\label{ieq:estErrorJJn}
			\sup_{t\in I_n} \big\|\big(v - \JJn v\big)^{(l)}(t)\big\|
				\leq C \sum_{j=j_{\mathrm{min},\check{r}}}^{j_{\mathrm{max},\check{r}}}
					\left(\tfrac{\tau_n}{2}\right)^{j-l} \sup_{t\in I_n} \big\|v^{(j)}(t)\big\|
				=: \KAT^{\check{r},l}_{\eqref{ieq:estErrorJJn}}\big[v\big]
		\end{gather}
		holds with a constant $C$ independent of $\tau_n$.
	\end{lemma}
	\begin{proof}
		The error estimate follows from standard approximation theory since $\JJn$ preserves
		under the given assumptions polynomials up to degree $\min\{r,\rIfII+1\}$.
		Here also note the stability estimate for $\JJn$, cf.~Lemma~\ref{le:stabJJn},
		which motivates the upper summation bound $j_{\mathrm{max},\check{r}}$.
	\end{proof}
	
	As we shall see below, compared to the pointwise estimate of Lemma~\ref{le:errorJJn},
	the estimate for the approximation error of $\JJ$ in the mesh points $t_n^-$ can
	be even improved in some cases. However, to this end we need some further knowledge
	on the approximation property connected with the quantity $\rIfIIi{i}$.
	Besides, the respective result presented in the following lemma
	will be also used later in the superconvergence analysis.
	
	\begin{lemma}
		\label{le:auxLemma}
		Let $r,k \in \ZZ$, $0 \leq k \leq r$. Suppose that the Assumptions~\ref{assSupII}
		and~\ref{assSup} hold. Moreover, assume that $\psi_i \in P_i(I_n,\RR)$, $i \in \NN_0$,
		satisfies $\sup_{t \in I_n} |\psi_i^{(j)}(t)| \leq C \tau_n^{-j}$ for all $j\in\NN_0$.
		Let $\check{r} \in \NN_0$ and define
		\begin{gather*}
			j_{\mathrm{min},i,\check{r}}^* := \min\{\check{r},\rIfIIi{i}+1\},
			\qquad
			j_{\mathrm{max},i,\check{r}}^* := \max\{\kII,\kIf,j_{\mathrm{min},i,\check{r}}^*\}.
		\end{gather*}
		Then, we have for $\varphi \in C^{j_{\mathrm{max},i,\check{r}}^*}(\overline{I}_n,\RR^d)$ the bound
		\begin{align*}
			\left\| \IIn{\big(\varphi - \If_n \varphi \big)(t) \psi_i(t) } \right\|
			& \leq C \frac{\tau_n}{2} 
				\sum_{j=j_{\mathrm{min},i,\check{r}}^*}^{j_{\mathrm{max},i,\check{r}}^*}
					\left(\tfrac{\tau_n}{2}\right)^{j} \sup_{t\in I_n} \big\|\varphi^{(j)}(t)\big\|
		\end{align*}
		with a constant $C$ independent of $\tau_n$.
	\end{lemma}
	\begin{proof}
		Let $\widetilde{\varphi} \in P_{\min\{\check{r}-1,\rIfIIi{i}\}}(I_n,\RR^d)$
		be arbitrarily chosen. We start rewriting the left-hand side of the
		wanted inequality as follows
		\begin{equation*}
			\IIn{\big(\varphi - \If_n \varphi \big)(t) \psi_i(t) }
			= \IIn{\big(\varphi - \widetilde{\varphi}\big)(t) \psi_i(t) }
				+ \IIn{\big(\widetilde{\varphi} - \If_n \widetilde{\varphi} \big)(t) \psi_i(t) }
				+ \IIn{\If_n\big(\widetilde{\varphi} - \varphi \big)(t) \psi_i(t) }\!.
		\end{equation*}
		Since the second summand on the right-hand side vanishes by
		definition of $\rIfIIi{i}$, we find
		\begin{gather*}
			\left\| \IIn{\big(\varphi - \If_n \varphi \big)(t) \psi_i(t) } \right\|
			\leq \left\| \IIn{\big(\varphi - \widetilde{\varphi}\big)(t) \psi_i(t) } \right\|
				+ \left\| \IIn{\If_n\big(\varphi - \widetilde{\varphi} \big)(t) \psi_i(t) } \right\|\!.
		\end{gather*}
		Exploiting Assumption~\ref{assSupII}, the Leibniz rule for the $j$th derivative,
		and the given bound for $\psi_i^{(j)}$, we gain
		\begin{align*}
			& \left\| \IIn{\big(\varphi - \If_n \varphi \big)(t) \psi_i(t) } \right\| \\
			& \quad \leq \CII \frac{\tau_n}{2} \sum_{j=0}^{\kII} \left(\tfrac{\tau_n}{2}\right)^j
					\left(\sup_{t \in I_n} \big\| \big((\varphi - \widetilde{\varphi})\psi_i\big)^{(j)}(t)\big\|
						+ \sup_{t \in I_n} \big\| \big(\If_n (\varphi - \widetilde{\varphi} ) \psi_i\big)^{(j)}(t)\big\| \right) \\
			& \quad \leq \CII \frac{\tau_n}{2} \sum_{j=0}^{\kII} \left(\tfrac{\tau_n}{2}\right)^j
				\sum_{l=0}^{j} \tbinom{j}{l}
				\left(\sup_{t \in I_n} \big\| \big(\varphi - \widetilde{\varphi}\big)^{(l)}(t) \big\|
					+ \sup_{t \in I_n} \big\| \big(\If_n(\varphi - \widetilde{\varphi})\big)^{(l)}(t) \big\| \right)
				\sup_{t \in I_n} \big\| \psi_i^{(j-l)}(t) \big\| \\
			& \quad \leq C \frac{\tau_n}{2} \sum_{j=0}^{\kII}
				\sum_{l=0}^{j} \tbinom{j}{l} \left(\tfrac{\tau_n}{2}\right)^l
				\left(\sup_{t \in I_n} \big\| \big(\varphi - \widetilde{\varphi}\big)^{(l)}(t) \big\|
					+ \sup_{t \in I_n} \big\| \big(\If_n(\varphi - \widetilde{\varphi})\big)^{(l)}(t) \big\| \right) \\
			& \quad \leq C \frac{\tau_n}{2} \sum_{l=0}^{\kII}
				\left(\tfrac{\tau_n}{2}\right)^l \left(\sup_{t \in I_n} \big\| \big(\varphi
				- \widetilde{\varphi}\big)^{(l)}(t) \big\|
					+ \sup_{t \in I_n} \big\| \big(\If_n(\varphi - \widetilde{\varphi})\big)^{(l)}(t) \big\| \right)\!.
		\end{align*}
		Further, using Assumption~\ref{assSup} which yields
		\begin{gather*}
			\sup_{t \in I_n} \big\| \big(\If_n(\varphi - \widetilde{\varphi})\big)^{(l)}(t) \big\|
				\leq C \sum_{j=0}^{\max\{\kIf,l\}} \left(\tfrac{\tau_n}{2}\right)^{j-l}
					\sup_{t \in I_n} \big\| \big(\varphi - \widetilde{\varphi}\big)^{(j)}(t)\big\|,
		\end{gather*}
		we conclude that
		\begin{multline*}
			\left\| \IIn{\big(\varphi - \If_n \varphi \big)(t) \psi_i(t) } \right\| \\
                        \begin{aligned}
			& \quad \leq C \frac{\tau_n}{2} \sum_{l=0}^{\kII}
				\left(\left(\tfrac{\tau_n}{2}\right)^l \sup_{t \in I_n} \big\| \big(\varphi
				- \widetilde{\varphi}\big)^{(l)}(t) \big\|
					+ C \sum_{j=0}^{\max\{\kIf,l\}} \left(\tfrac{\tau_n}{2}\right)^{j}
					\sup_{t \in I_n} \big\| \big(\varphi - \widetilde{\varphi}\big)^{(j)}(t)\big\| \right) \\
			& \quad \leq C \frac{\tau_n}{2} \sum_{j=0}^{\max\{\kII,\kIf\}} \left(\tfrac{\tau_n}{2}\right)^{j}
					\sup_{t \in I_n} \big\| \big(\varphi - \widetilde{\varphi}\big)^{(j)}(t)\big\|
                        \end{aligned}
		\end{multline*}
		for all $\widetilde{\varphi} \in P_{\min\{\check{r}-1,\rIfIIi{i}\}}(I_n,\RR^d)$.
		
		Choosing $\widetilde{\varphi}$ as the Taylor polynomial
		of $\varphi$ at $(t_{n-1}+t_n)/2$ of degree $\min\{\check{r}-1,\rIfIIi{i}\}$,
		we have that
		\begin{gather*}
			\sup_{t \in I_n} \big\| \big(\varphi - \widetilde{\varphi}\big)^{(j)}(t)\big\|
				\leq C \left(\tfrac{\tau_n}{2}\right)^{\max\{j,\min\{\check{r},\rIfIIi{i}+1\}\}-j}
					\sup_{t \in I_n} \big\| \varphi^{(\max\{j,\min\{\check{r},\rIfIIi{i}+1\}\})}(t)\big\|.
		\end{gather*}
		So, overall it follows
		\begin{multline*}
			\left\| \IIn{\big(\varphi - \If_n \varphi \big)(t) \psi_i(t) } \right\| \\
                        \begin{aligned}
			& \quad \leq C \frac{\tau_n}{2} \sum_{j=0}^{\max\{\kII,\kIf\}}
					\left(\tfrac{\tau_n}{2}\right)^{\max\{j,\min\{\check{r},\rIfIIi{i}+1\}\}}
					\sup_{t \in I_n} \big\| \varphi^{(\max\{j,\min\{\check{r},\rIfIIi{i}+1\}\})}(t)\big\|\\
			& \quad \leq C \frac{\tau_n}{2} \sum_{j=\min\{\check{r},\rIfIIi{i}+1\}}^{\max\{\kII,\kIf,\min\{\check{r},\rIfIIi{i}+1\}\}} 
					\left(\tfrac{\tau_n}{2}\right)^{j} \sup_{t \in I_n} \big\| \varphi^{(j)}(t) \big\|
                        \end{aligned}
		\end{multline*}
		which completes the proof.
	\end{proof}
	
	Now, we can state and prove the improved estimate for the approximation error of $\JJ$ in the mesh points $t_n^-$.
	
	\begin{lemma}
		\label{le:errorJJnPoint}
		Let $r,k \in \ZZ$, $0 \leq k \leq r$. Suppose that the Assumptions~\ref{assIIhat},
		\ref{assSupII}, and~\ref{assSup} hold.
		Moreover, assume that $\max\{\rExII,\rIfII+1\} \geq r-1$.
		Let $\check{r} \in \NN_0$ and define
		\begin{gather*}
			j_{\mathrm{min},\check{r}}^{\diamond}
				:= \min\{\check{r},\,\max\{\rExII+1,\min\{r,\rIfII+1\}\}+1,\,\rIfIIi{0}+2\},
			\qquad 
			j_{\mathrm{max},\check{r}}^{\diamond}
				:= \max\{\kJJ+1, j_{\mathrm{min},\check{r}}^{\diamond}\}.
		\end{gather*}
		Then, provided $v \in C^{j_{\mathrm{max},\check{r}}^{\diamond}}(\overline{I}_n,\RR^d)$, the estimate
		\begin{gather}
		\label{ieq:estErrorJJnPoint}
			\big\|\big(v - \JJn v\big)(t_{n}^-)\big\|
			\leq C \sum_{j=j_{\mathrm{min},\check{r}}^{\diamond}}^{j_{\mathrm{max},\check{r}}^{\diamond}}
					\left(\tfrac{\tau_n}{2}\right)^{j} \sup_{t\in I_n} \big\|v^{(j)}(t)\big\|
			=: \KAT_{\eqref{ieq:estErrorJJnPoint}}\big[v\big]
		\end{gather}
		holds for $1 \leq n \leq N$ where the constant $C$ is independent of $\tau_n$.
	\end{lemma}
	\begin{proof}
		We start rewriting the term to be estimated. From  the fundamental theorem of
		calculus and exploiting the definition~\eqref{def:JJn} of $\JJn$ we gain
		\begin{align*}
			& \big(v - \JJn v\big)(t_{n}^-)
				= \big(v - \JJn v\big)(t_{n-1}^+) + \int_{I_n} \big(v - \JJn v\big)'(t) \, \d t \\
			& \quad = \underbrace{\delta_{0,k} \big(v - \JJn v\big)(t_{n-1}^+) + \IIn{\If_n v' - (\JJn v)'}}_{=0}
				- \IIn{\If_n v' - (\JJn v)'} \\[-1em]
			& \hspace{28em} + \int_{I_n} \big(v - \JJn v\big)'(t) \, \d t \\
			& \quad = \left(\int_{I_n} v'(t)\,\d t - \IIn{ v' }\right)
				+ \IIn{v'-\If_n v'}
				+ \left(\IIn{(\JJn v)'} - \int_{I_n} (\JJn v)'(t) \,\d t\right).
		\end{align*}
		The first difference on the right-hand side can be estimated by~\eqref{ieq:IntErrorIIn}.
		We obtain
		\begin{gather*}
			\left\| \int_{I_n} v'(t)\,\d t - \IIn{ v' } \right\|
				\leq C \frac{\tau_n}{2} 
					\sum_{j=\min\{\check{r}-1,\rExII+1\}}^{\max\{\kII,\min\{\check{r}-1,\rExII+1\}\}}
					\left(\tfrac{\tau_n}{2}\right)^j \sup_{t \in I_n} \big\| v^{(j+1)}(t)\big\|.
		\end{gather*}
		In order to bound the second difference, we apply Lemma~\ref{le:auxLemma} with
		$i = 0$, $\psi_0 \equiv 1$. It follows
		\begin{align*}
			\left\| \IIn{v' - \If_n v' } \right\|
			& \leq C \frac{\tau_n}{2} 
				\sum_{j=\min\{\check{r}-1,\rIfIIi{0}+1\}}^{\max\{\kII,\kIf,\min\{\check{r}-1,\rIfIIi{0}+1\}\}}
					\left(\tfrac{\tau_n}{2}\right)^{j} \sup_{t\in I_n} \big\|v^{(j+1)}(t)\big\|.
		\end{align*}
		Finally, if $\rExII \geq r-1$ the third difference vanishes since then
		$(\JJn v)' \in P_{r-1}(I_n,\RR^d)$ is integrated exactly. If $\rIfII +1 \geq r-1$,
		i.e., $\JJn$ preserves polynomials up to degree $r-1$, we obtain from
		combining~\eqref{ieq:IntErrorIIn} and~\eqref{ieq:estErrorJJn} that
		\begin{align*}
			& \left\| \int_{I_n} (\JJn v)'(t) \,\d t - \IIn{(\JJn v)'}\right\| \\
			& \quad \leq C \frac{\tau_n}{2} \sum_{j=\min\{\check{r}-1,\rExII+1\}}^{\min\{r-1,\max\{\kII, \min\{\check{r}-1,\rExII+1\}\}\}}
					\left(\tfrac{\tau_n}{2}\right)^j
					\sup_{t \in I_n} \big\| (\JJn v)^{(j+1)}(t)\big\| \\
			& \quad \leq C \sum_{j=\min\{\check{r},\rExII+2\}}^{\min\{r,\max\{\kII+1, \min\{\check{r},\rExII+2\}\}\}}
					\left(\tfrac{\tau_n}{2}\right)^{j}
					\Big(\sup_{t \in I_n} \big\| v^{(j)}(t)\big\|
						+ \underbrace{\sup_{t \in I_n} \big\| (v-\JJn v)^{(j)}(t)\big\|}_{\leq \KAT^{j,j}_{\eqref{ieq:estErrorJJn}}[v]}\Big)\\[-1ex]
			& \quad \leq C \sum_{j=\min\{\check{r},\rExII+2\}}^{\min\{r,\max\{\kII+1, \min\{\check{r},\rExII+2\}\}\}}
					\Bigg(\sum_{l=j}^{\max\{\kJJ + 1, j\}}
						\left(\tfrac{\tau_n}{2}\right)^{l} \sup_{t\in I_n} \big\|v^{(l)}(t)\big\| \Bigg)\\
			& \quad \leq C \sum_{l=\min\{\check{r},\rExII+2\}}^{
						\max\{\kJJ + 1, \min\{r,\check{r},\rExII+2\}\}}
						\left(\tfrac{\tau_n}{2}\right)^{l} \sup_{t\in I_n} \big\|v^{(l)}(t)\big\|.
		\end{align*}
		So, overall we have shown that
		\begin{gather*}
			\big\|\big(v - \JJn v\big)(t_{n}^-)\big\|
			\leq C \sum_{j=\min\left\{\check{r},\, \rExII+2,\,\rIfIIi{0}+2\right\}}^{
				\max\left\{\kJJ + 1,\, \min\left\{\check{r},\,\rExII+2,\,\rIfIIi{0}+2\right\}\right\}}
					\left(\tfrac{\tau_n}{2}\right)^{j} \sup_{t\in I_n} \big\|v^{(j)}(t)\big\|.
		\end{gather*}
		Therefore, recalling~\eqref{ieq:estErrorJJn} which in some cases may
		provide a better estimate, we conclude the wanted statement. Here also note
		that always $\rIfIIi{0} \geq \rIfIIi{r-k} = \rIfII \geq \min\{r-1,\rIfII\}$.
	\end{proof}
	
	Finally, summarizing the above results, we now want to list the proven convergence orders.
	
	\begin{corollary}
	\label{Cor:ErrEst}
	Let $r,k \in \ZZ$, $0\leq k \leq r$, and $l \in \{0,1\}$. Suppose that
	Assumptions~\ref{assIIhat}, \ref{assSupII}, \ref{assSup}, and~\ref{assDiff}
	hold. Moreover, let Assumption~\ref{assPoint_a} or~\ref{assPoint_b} be
	satisfied. Denoting by $u$ and $U$ the solutions of~\eqref{initValueProb}
	and~\eqref{eq:locProbGen}, respectively, we have for $1 \leq n \leq N$
	\begin{gather}
		\label{Eq:ErrEst1}
		\sup_{t \in I_n} \big\| (u-U)^{(l)}(t) \big\|
		\leq C \tau^{\min\{r,\rIfII+1\}},
	\end{gather}
	with $\rIfII$ as defined in Definition~\ref{def:approxOrders}.
	If we furthermore suppose that $\max\{\rExII,\rIfII+1\} \geq r-1$,
	then we even have
	\begin{gather}
		\label{Eq:ErrEst2}
		\sup_{t \in I_n} \big\| (u-U)(t) \big\|
		%\leq C \tau^{\min\{r+1,\,\rIfII+2,\,\max\{\rExII+1,
		%\min\{r,\rIfII+1\}\},\,\rIfIIi{0}+1\}}
		\leq C \tau^{\min\{r+1,\,\rIfII+2,
		\,\rIfIIi{0}+1,
		\,\max\{\rExII+1, \min\{r,\rIfII+1\}\}
		\}}
	\end{gather}
	as improved $L^\infty$ estimate.
	\end{corollary}
	
	If $\max\{\rExII,\rIfII+1\} \geq r-1$ is satisfied, we obtain formally
	\begin{gather*}
		\sup_{t \in I_n} \big\| (u-U)'(t) \big\|
		%\leq C \tau^{\min\{r,\,\rIfII+1,\,\max\{\rExII+1,\min\{r,\rIfII+1\}\},
		%\,\rIfIIi{0}+1\}}.
		\leq C \tau^{\min\{r,\,\rIfII+1,\,
		\rIfIIi{0}+1,\,
		\max\{\rExII+1,\min\{r,\rIfII+1\}\}
		\}}
	\end{gather*}
	for the $W^{1,\infty}$ seminorm. However, this gives the same convergence
	order as~\eqref{Eq:ErrEst1}.
	
	Since the quantity $\rIfII = \rIfIIi{r-k}$ used in the lemmas and the corollary
	above is quite abstract, we want to provide some lower bound for $\rIfIIi{i}$
	based on the more familiar quantities $\rIf$, $\rExIf$, and $\rExII$.
	However, for this result, we need to formulate some further assumptions.
	
	\begin{assumption}
		\label{assErrorStar}
		The operator $\If_n$ is a projection onto the space of polynomials of maximal degree $\rIf < \infty$,
		i.e., $\If_n : C^{\kIf}(\overline{I}_n) \to P_{\rIf}(\overline{I}_n)$ and
		$\If_n \varphi = \varphi$ for all $\varphi \in P_{\rIf}(\overline{I}_n)$,
		or $\If_n$ is the identity characterized by setting $\rIf = \infty$.
	\end{assumption}
	
	\begin{assumption}
		\label{assPreserveIfn}
		Given $\ell\in\NN_0$, the identity
		\begin{gather}
			\label{eq:PreserveIfn}
			\int_{I_n} \If_n \big((\varphi-\If_n \varphi) \psi\big)(t)\,\d t = 0
			\qquad \forall \psi \in P_{\ell}(I_n)
		\end{gather}
		holds for all $\varphi \in C^{k_{\If}}(\overline{I}_n)$.
		Note that Assumption~\ref{assErrorStar} implies~\eqref{eq:PreserveIfn} for $\ell=0$.
		If $\If_n$ is either the identity or an Hermite-type interpolation operator then
		condition~\eqref{eq:PreserveIfn} holds for any $\ell\in\NN_0$.
	\end{assumption}
	
	We now give the lower bounds for $\rIfIIi{i}$.
	
	\begin{lemma}
		Let $r,k \in \ZZ$, $0 \leq k \leq r$, and $i \in \NN_0$. Then, it always holds $\rIfIIi{i} \geq \rIf$.
		Supposing that Assumption~\ref{assErrorStar} is fulfilled, we even get
		\begin{gather*}
			\rIfIIi{i} \geq \max\{\rIf,\min\{\rExII-i,\rIfInti{i}\}\}.
		\end{gather*}
		If $\If_n$ additionally satisfies Assumption~\ref{assPreserveIfn} (with parameter
		$\ell=i$), we simply have
		\begin{gather*}
			\rIfIIi{i} \geq \max\{\rIf,\min\{\rExII,\rExIf\}-i\}
		\end{gather*}
		since then $\rIfInti{i} \geq \max\{\rIf,\rExIf-i\}$.
		
		Furthermore, under the weaker assumption that $\If = \If^1 \circ \ldots \circ \If^l$
		is a composition of several operators $\If^j$, $1 \leq j \leq l$, that all by
		themselves satisfy the Assumptions~\ref{assErrorStar} and~\ref{assPreserveIfn}
		(with parameter $\ell=i$), we still find
		\begin{gather*}
			\rIfInti{i} \geq \min_{j\in \KAM_i \cup \{l\}}\big\{\max\{r_{\If^j},r_{\mathrm{ex}}^{\If^j}-i\}\big\}
		\end{gather*}
		where $\KAM_i := \big\{ j \in \NN \,:\, 1 \leq j \leq l-1, \,
			\max\{r_{\If^j},r_{\mathrm{ex}}^{\If^j}-i\} < \min_{j+1\leq m \leq l}\{r_{\If^m}\}\big\}$.
	\end{lemma}
	\begin{proof}
		Since by definition $\varphi = \If_n \varphi$ for $\varphi \in P_{\rIf}(I_n)$,
		it also holds $\IIn{\varphi \psi}=\IIn{(\If_n \varphi) \psi}$ for
		$\varphi \in P_{\rIf}(I_n)$ and sufficiently smooth $\psi$.
		This always implies $\rIfIIi{i} \geq \rIf$.
		
		Now, assume that $\min\{\rExII-i,\rIfInti{i}\} > \rIf$ and let
		$\varphi \in P_{\min\{\rExII-i,\rIfInti{i}\}}(I_n)$. Then Assumption~\ref{assErrorStar}
		yields that $\If_n \varphi \in P_{\rIf}(I_n) \subset P_{\min\{\rExII-i,\rIfInti{i}\}}(I_n)$.
		Therefore, since $\II$ is exact for polynomials up to degree $\rExII$ 
		and recalling the definition of $\rIfInti{i}$, it follows
		\begin{gather*}
			\IIn{(\varphi-\If_n \varphi)\psi_i}
				= \int_{I_n} \big(\varphi - \If_n \varphi \big)\psi_i \,\d t = 0
			\qquad \text{for all $\psi_i \in P_{i}(I_n)$.}
		\end{gather*}
		Hence, $\rIfIIi{i} \geq \min\{\rExII-i,\rIfInti{i}\}$, if $\min\{\rExII-i,\rIfInti{i}\} > \rIf$.
		So, altogether under Assumption~\ref{assErrorStar} we have shown that
		$\rIfIIi{i} \geq \max\{\rIf,\min\{\rExII-i,\rIfInti{i}\}\}$.
		
		Finally, we assume that $\If = \If^1 \circ \ldots \circ \If^l$ where each
		$\If^j$ satisfies the Assumptions~\ref{assErrorStar} and~\ref{assPreserveIfn}.
		Let $\varphi \in P_{\check{r}}(I_n)$ with $\check{r} \geq 0$ to be specified
		later and $\psi_i \in P_i(I_n)$, then
		\begin{align*}
			& \int_{I_n} \big(\varphi - \If_n^1 \circ \ldots \circ \If_n^l \varphi \big) \psi_i \,\d t \\
			& \quad = \int_{I_n} \big((\Id - \If_n^l ) \varphi \big) \psi_i \,\d t
				+ \sum_{j=1}^{l-1} \int_{I_n} \big((\Id - \If_n^j)(\If_n^{j+1} \circ \ldots \circ \If_n^l \varphi)\big) \psi_i \,\d t \\
			& \quad = \int_{I_n} \big((\Id - \If_n^l ) \varphi \big) \psi_i \,\d t
				+ \sum_{j \in \KAM_\infty} \int_{I_n} \big((\Id - \If_n^j)(\If_n^{j+1} \circ \ldots \circ \If_n^l \varphi)\big) \psi_i \,\d t \\
			& \quad = \int_{I_n} (\Id-\If_n^l)\big(\big((\Id - \If_n^l ) \varphi \big) \psi_i\big) \,\d t
				+ \sum_{j \in \KAM_\infty}
					\int_{I_n} (\Id - \If_n^j)\big(\big((\Id - \If_n^j)(\If_n^{j+1} \circ \ldots \circ \If_n^l \varphi)\big) \psi_i\big) \,\d t \\
			& \quad \qquad + \int_{I_n} \If_n^l \big(\big((\Id - \If_n^l ) \varphi \big) \psi_i\big) \,\d t
				+ \sum_{j \in \KAM_\infty}
					\int_{I_n} \If_n^j \big(\big((\Id - \If_n^j)(\If_n^{j+1} \circ \ldots \circ \If_n^l \varphi)\big) \psi_i\big) \,\d t.
		\end{align*}
		Because of~\eqref{eq:PreserveIfn}, both terms in the last line vanish.
		Here note that actually~\eqref{eq:PreserveIfn} is only needed for $\If_n^j$
		with $j \in \KAM_\infty \cup \{l\}$ and
		$\varphi \in P_{\min\{\check{r},\min\{r_{\If^m}\,:\,j+1\leq m \leq l\}\}}(I_n)$.
		Moreover, since
		\begin{gather*}
			\big((\Id - \If_n^j)(\If_n^{j+1} \circ \ldots \circ \If_n^l \varphi)\big) \psi_i
				\in P_{\min\{\check{r},\min\{r_{\If^m}\,:\,j+1\leq m \leq l\}\}+i}(I_n)
		\end{gather*}
		also those summands of the penultimate line with
		$\max\{r_{\If^j},r_{\mathrm{ex}}^{\If^j}-i\} \geq \min_{j+1\leq m \leq l}\{r_{\If^m}\}$
		vanish. Hence, we obtain
		\begin{align*}
			& \int_{I_n} \big(\varphi - \If_n^1 \circ \ldots \circ \If_n^l \varphi \big) \psi_i \,\d t \\
			& \quad = \int_{I_n} (\Id-\If_n^l)\big(\big((\Id - \If_n^l ) \varphi \big) \psi_i\big) \,\d t
				+ \sum_{j \in \KAM_i}
					\int_{I_n} (\Id - \If_n^j)\big(\big((\Id - \If_n^j)(\If_n^{j+1} \circ \ldots \circ \If_n^l \varphi)\big) \psi_i\big) \,\d t.
		\end{align*}
		From this identity, we easily find that
		$\rIfInti{i} \geq \min_{j\in \KAM_i \cup \{l\}}\big\{\max\{r_{\If^j},r_{\mathrm{ex}}^{\If^j}-i\}\big\}$.
		Note that here Assumption~\ref{assErrorStar} is exploited to guarantee
		that for a polynomial $\varphi$ the degree of $\If_n^j \varphi$ is never
		greater than that of $\varphi$.
	\end{proof}

%%%%%%%%%%%%%%%%%%%%%%%%%%%%%%%%%%%%%%%%%%%%%%%%%%%%%%%%%%%%%%%%%%%%%%%%%%%%%%%
%%%%%%%%%%%%%%%%%%%%%%%%%%%%%%%%%%%%%%%%%%%%%%%%%%%%%%%%%%%%%%%%%%%%%%%%%%%%%%%
	\section{Superconvergence analysis}
	\label{sec:superconvergence}
	
	In order to prove superconvergence in time mesh points, we shall exploit a
	special representation of the discrete problem~\eqref{eq:locProbGen}.
	However, to this end we will define the approximation operator
	$\PPhat : C^{\kJJ}([-1,1]) \to P_{r-1}([-1,1])$ by
	\begin{subequations}
	\label{def:PPhat}
	\begin{align}
		\big(\PPhat \hat{v}\big)^{(i)}(-1^+)
			& = \hat{v}^{(i)}(-1^+),
			&& \text{if } k \geq 3, \, i = 0,\ldots,\left\lfloor \tfrac{k-1}{2} \right\rfloor - 1, \\
		\big(\PPhat \hat{v}\big)^{(i)}(+1^-)
			& = \hat{v}^{(i)}(+1^-),
			&& \text{if } k \geq 2, \,i = 0,\ldots,\left\lfloor \tfrac{k}{2} \right\rfloor -1, \\
		\IIhat\Big[\big(\PPhat \hat{v}\big) \widehat{\varphi} \Big]
			& = \IIhat\Big[\big(\Ifhat \hat{v} \big) \widehat{\varphi} \Big]
			&& \forall \widehat{\varphi} \in P_{r-k}([-1,1]) \text{ with } \delta_{0,k}\widehat{\varphi}(-1^+) = 0.
	\end{align}
	\end{subequations}
	Note that $\PPhat$ is connected to $\JJhat$ in such a way that
	$\PPhat (\hat{v}') = (\JJhat \hat{v})'$ holds true.
	Analogously to the respective lemmas for $\JJhat$ (Lemmas~\ref{le:JJhatWellDef},
	\ref{le:approxJJn}, and~\ref{le:errorJJn}), we also get the following
	properties and estimates for the operator $\PPhat$.
	
	\begin{lemma}
		\label{le:PPhatWellDef}
		Let $r\in\NN$ and $k\in\ZZ$ such that $0\leq k \leq r$. Suppose that
		Assumption~\ref{assIIhat} holds. Then $\PPhat$ given by~\eqref{def:PPhat}
		is well-defined. If $\Ifhat$ preserves polynomials up to degree $\tilde{r}$
		then $\PPhat$ preserves polynomials up to degree $\min\{\tilde{r},r-1\}$.
	\end{lemma}
	
	\begin{lemma}[Approximation operator]
		\label{le:approxPPn}
		Let $r\in\NN$ and $k\in\ZZ$ such that $0\leq k \leq r$. Moreover, suppose
		that Assumption~\ref{assIIhat} holds. Then the operator
		$\PPn:C^{\kJJ}(\overline{I}_n,\RR^d) \to P_{r-1}(I_n,\mathbb{R}^d)$
		given by
		\begin{subequations}
		\label{eq:apprPP}
		\begin{align}
			(\PPn v)^{(i)}(t_{n-1}^+)
				& = v^{(i)}(t_{n-1}^+),
				&& \text{if } k \geq 3,\, i=0,\ldots,\left\lfloor \tfrac{k-1}{2} \right\rfloor -1,
			\label{eq:apprPP_A} \\
			(\PPn v)^{(i)}(t_{n}^-)
				& = v^{(i)}(t_{n}^-),
				&& \text{if } k \geq 2,\, i=0,\ldots,\left\lfloor \tfrac{k}{2} \right\rfloor-1,
			\label{eq:apprPP_E} \\
			\IIn{ \big(\PPn v(t), \varphi(t)\big)}
				& = \IIn{ \big(\If_n v(t), \varphi(t)\big)}
				&& \forall \varphi \in P_{r-k}(I_n,\RR^d) \text{ with } \delta_{0,k} \varphi(t_{n-1}^+) = 0,
			\label{eq:apprPP_Var}
		\end{align}
		\end{subequations}
		is well-defined.
		
		Furthermore, let Assumptions~\ref{assSupII} and~\ref{assSup} hold.
		For $l, \check{r} \in \NN_0$, we define
		\begin{gather*}
			j_{\mathrm{min},\check{r}}^{\bullet}
				:= \min\{\check{r},r,\rIfII+1\},
			\qquad
			j_{\mathrm{max},l,\check{r}}^{\bullet}
				:= \max\{\kJJ, l, j_{\mathrm{min},\check{r}}^{\bullet}\}.
		\end{gather*}
		Then, we have for $v \in C^{j_{\mathrm{max},l,\check{r}}^{\bullet}}(\overline{I}_n,\RR^d)$
		the error estimates
		\begin{gather}
			\label{ieq:estErrorPPn}
			\sup_{t\in I_n} \big\|\big(v - \PPn v\big)^{(l)}(t)\big\|
				\leq C \sum_{j=j_{\mathrm{min},\check{r}}^{\bullet}}^{j_{\mathrm{max},l,\check{r}}^{\bullet}}
					\left(\tfrac{\tau_n}{2}\right)^{j-l} \sup_{t\in I_n} \big\|v^{(j)}(t)\big\|
				=: \KAT_{\eqref{ieq:estErrorPPn}}^{\check{r},l}\big[v\big]
		\end{gather}
		with a constant $C$ independent of $\tau_n$.
	\end{lemma}
	
	In order to cover also the setting $r=k=0$, we agree that in this case
	the operators $\PPhat$ and $\PPn$ always return the zero function.
	
	\bigskip
	
	Inspecting the definition of the discrete solution $U$ of problem~\eqref{eq:locProbGen}
	and~\eqref{eq:apprPP}, we find that
	\begin{gather}
		\label{eq:derUisApproxOfF}
		U'\big|_{I_n} = \PPn f(\cdot,U(\cdot)).
	\end{gather}
	Therefore, since for $k \geq 1$ furthermore $U(t_{n-1}^+) = U_{n-1}^-$ with $U_{n-1}^-$ given, we also
	see that the discretization method fits into the unified framework of~\cite{AMN11} for $k \geq 1$.
	
	\begin{theorem}[Superconvergence estimate]
		\label{th:superconvergence}
		Let $r,k \in \ZZ$, $0 \leq k \leq r$. Suppose that the Assumptions~\ref{assIIhat},
		\ref{assSupII}, and~\ref{assSup} hold. Moreover, denote by $u$ and $U$
		the solutions of~\eqref{initValueProb} and~\eqref{eq:locProbGen}, respectively.
		Suppose that (for $\tau$ sufficiently small) the global error $\sup_{t \in I} \big\|(u-U)(t)\big\|$,
		as well as $U$ and all of its derivatives, can be bounded independent of the mesh parameter.
		Then we have
		\begin{equation}
			\label{Eq:ErrSupTheo}
			\bignorm{(u-U)(t_N^-)}
			\leq C(f,u) \left(\sup_{t \in I} \big\|(u-U)(t)\big\|^2
				+ \tau^{\min\{2r-k+1,\rIfIIVar+1,\max\{\rExII+1,\min\{r,\rIfII+1\}\}\}} \right)\!
		\end{equation}
		where $\rIfIIVar := \min_{0 \leq i \leq r-k}\{\rIfIIi{i}+i\}$.
	\end{theorem}
	\begin{proof}
		In order to prove the wanted statement, we firstly derive an estimate for the
		error $e(t) = (u-U)(t)$ at $t_n^-$ provided that we already have a suitable bound
		at $t_{n-1}^-$. To this end, we adapt some basic ideas known from the superconvergence proof
		for collocation methods, see e.g.~\cite[Theorem~II.1.5, p.~28]{HLW04}. So, we
		consider the local discrete solution $U\big|_{I_n}$ as solution $\widetilde{U}$ of the
		perturbed initial value problem
		\begin{gather*}
			\widetilde{U}'(t) = f(t,\widetilde{U}(t)) + \df(t)  \quad \text{in } I_n,
			\qquad \widetilde{U}(t_{n-1}) = U(t_{n-1}^+),
		\end{gather*}
		where $\df(t) := U'(t) - f(t,U(t))$ denotes the defect. Since $u$
		solves~\eqref{initValueProb}, we find for all $t \in I_n$
		\begin{align*}
			U'(t) - u'(t)
			& = f(t,U(t)) - f(t,u(t)) + \df(t) \\
			& = \int_{0}^1 \frac{\partial}{\partial v} f\big(t,u(t)+s(U(t)-u(t))\big) \big(U(t)-u(t)\big) \,\d s + \df(t) \\
			& = \frac{\partial}{\partial v} f\big(t,u(t)\big) \big(U(t)-u(t)\big) + \df(t) \\
			& \qquad + \underbrace{\int_{0}^1 \int_0^s \big(U(t)-u(t)\big)^T
				\frac{\partial^2}{\partial v^2} f\big(t,u(t)+\tilde{s}(U(t)-u(t))\big)
				\big(U(t)-u(t)\big) \,\d \tilde{s} \, \d s}_{=:\,\rf(t)}.
		\end{align*}
		Here $f$ is interpreted as function of $(t,v)$ such that $\frac{\partial}{\partial v} f$
		for example is the derivative of $f$ with respect to the second (function) argument.
		Also note that we used for the last identity that
		\begin{equation*}
			\frac{\partial}{\partial v} f\big(t,u(t)+s(U(t)-u(t))\big)
				= \frac{\partial}{\partial v} f\big(t,u(t)\big) + \int_0^s \big(U(t)-u(t)\big)^T \frac{\partial^2}{\partial v^2}
					f\big(t,u(t)+\tilde{s}(U(t)-u(t))\big) \,\d \tilde{s}.
		\end{equation*}
		Shortly, we have
		\begin{gather*}
			e'(t) = \frac{\partial}{\partial v} f\big(t,u(t)\big) e(t) - \df(t) - \rf(t)
			\qquad \forall t \in I_n.
		\end{gather*}
		Therefore, the variation of constants formula, cf. e.g.~\cite[Theorem~I.11.2, p.~66]{HNW08},
		yields
		\begin{gather*}
			e(t) = R(t,t_{n-1}^+) e(t_{n-1}^+) - \int_{t_{n-1}}^t R(t,s)\big(\df(s) + \rf(s)\big) \,\d s
			\qquad \forall t \in I_n
		\end{gather*}
		where $R(t,s)$ is the resolvent of the homogeneous differential equation
		$y'(t)= \frac{\partial}{\partial v} f\big(t,u(t)\big) y(t)$ for initial
		values given at $s$.
		
		Using that $u$ is continuous as well as $U$ for $k \geq 1$ we conclude
		\begin{equation}
			\label{eq:superDecomposition}
			e(t_n^-) = R(t_n^-,t_{n-1}^+) e(t_{n-1}^-) - \delta_{0,k} R(t_n^-,t_{n-1}^+) \big[U \big]_{n-1}
				- \int_{I_{n}} R(t_n^-,s)\big(\df(s) + \rf(s)\big) \,\d s.
		\end{equation}
		After splitting the right-hand side in an appropriate way, we shall study
		the single terms separately.
		
		First, for the term including $e(t_{n-1}^-)$ we gain
		\begin{gather*}
			R(t_n^-,t_{n-1}^+) e(t_{n-1}^-)
				= \bigg(\underbrace{R(t_{n-1}^+,t_{n-1}^+)}_{=\Id}
					+ \int_{I_n} \frac{\partial}{\partial t} R(\tilde{s},t_{n-1}^+) \,\d \tilde{s} \bigg) e(t_{n-1}^-)
		\end{gather*}
		which implies
		\begin{equation}
			\label{ieq:superInitialError}
			\big\|R(t_n^-,t_{n-1}^+) e(t_{n-1}^-)\big\|
				\leq \bigg(1+\tau_n \sup_{\tilde{s} \in I_n} \bigg\| \frac{\partial}{\partial t} R(\tilde{s},t_{n-1}^+)\bigg\| \bigg)
					\big\|e(t_{n-1}^-)\big\|
				\leq \big(1+C\tau_n\big) \big\|e(t_{n-1}^-)\big\|.
		\end{equation}
		Furthermore, the term including the remainder term $\rf(\cdot)$ can be bounded as follows
		\begin{align}
			\bigg\|\int_{I_n} R(t_n^-,s)\rf(s) \,\d s \bigg\|
			& \leq \tau_n \sup_{s \in I_n} \big\| R(t_{n}^-,s)\big\| \sup_{s \in I_n} \big\| \rf(s)\big\| \nonumber \\
			& \leq \tau_n \sup_{s \in I_n} \big\| R(t_{n}^-,s)\big\|
					\frac{1}{2} \sup_{s \in I_n}\sup_{\tilde{s} \in [0,1]}\bigg\|\frac{\partial^2}{\partial v^2} f\big(s,u(s)-\tilde{s}e(s))\big)\bigg\|
					\sup_{s \in I_n} \big\|e(s)\big\|^2 \nonumber \\
			& \leq C \tau_n \sup_{s \in I_n} \big\|e(s)\big\|^2. \label{ieq:superSquareError}
		\end{align}
		Here, for the last step we also exploited that
		$u(s)-\tilde{s}e(s)$ is in a bounded neighborhood of $u(s)$ for all $s \in I_n$,
		$\tilde{s} \in [0,1]$, since by assumption $\big\|e(s)\big\| \leq C$.
		
		Finally, the remaining terms are considered. A Taylor series expansion
		of $R(t_n^-,s)$ with respect to $s$ in $t_{n-1}^+$ gives
		\begin{gather*}
			R(t_n^-,s) = \sum_{i=0}^{r-k} \frac{(s-t_{n-1})^i}{i!}
					\frac{\partial^i}{\partial s^i} R(t_n^-,t_{n-1}^+)
				+ \int_{t_{n-1}}^s \frac{(\tilde{s}-t_{n-1})^{r-k}}{(r-k)!}
					\frac{\partial^{r-k+1}}{\partial s^{r-k+1}} R(t_n^-,\tilde{s}) \,\d \tilde{s}.
		\end{gather*}
		This formula motivates the following decomposition
		\begin{align}
			& \delta_{0,k} R(t_n^-,t_{n-1}^+) \big[U \big]_{n-1}
				+ \int_{I_{n}} R(t_n^-,s)\df(s)\,\d s \label{eq:helpSuperSplit}\\
			& \quad =  \sum_{i=0}^{r-k} \frac{1}{i!} \frac{\partial^i}{\partial s^i} R(t_n^-,t_{n-1}^+)
				\underbrace{\left(\int_{I_{n}} \big(s-t_{n-1}\big)^i
					\big(U'(s)-f(s,U(s)) \big)\,\d s + \delta_{0,k} \delta_{0,i} \big[U \big]_{n-1} \right) }_{=\mathrm{(I)}} \nonumber \\
			& \quad \qquad + \underbrace{\int_{I_n}\int_{t_{n-1}}^s \frac{(\tilde{s}-t_{n-1})^{r-k}}{(r-k)!}
				\frac{\partial^{r-k+1}}{\partial s^{r-k+1}} R(t_n^-,\tilde{s}) \,\d \tilde{s} \, \big(U'(s)-f(s,U(s))\big)\,\d s}_{=\mathrm{(II)}}. \nonumber
		\end{align}

		We start with the second term and bound $\mathrm{(II)}$ as follows
		\begin{align*}
			\big\|\mathrm{(II)}\big\|
				& \leq \int_{I_n} \int_{I_n} \frac{(\tilde{s}-t_{n-1})^{r-k}}{(r-k)!}
					\left\| \frac{\partial^{r-k+1}}{\partial s^{r-k+1}} R(t_n^-,\tilde{s}) \right\| \,\d \tilde{s}
					\, \big\|U'(s)-f(s,U(s))\big\| \,\d s \\
				& \leq \tau_n \frac{\tau_n^{r-k+1}}{(r-k+1)!}
					\sup_{\tilde{s} \in I_n} \left\| \frac{\partial^{r-k+1}}{\partial s^{r-k+1}} R(t_n^-,\tilde{s}) \right\|
					\sup_{s \in I_n} \big\|U'(s)-f(s,U(s))\big\|.
		\end{align*}
		Because of~\eqref{eq:derUisApproxOfF} the last term on the right-hand side can be
		rewritten as an approximation error. Indeed, we have
		\begin{gather*}
			\sup_{s \in I_n} \big\|U'(s)-f(s,U(s))\big\|
				= \sup_{s \in I_n} \big\|\big(\Id - \PPn\big)f(s,U(s))\big\|
				\leq \KAT_{\eqref{ieq:estErrorPPn}}^{\min\{\bar{r},r,\rIfII+1\},0}\big[f(\cdot,U(\cdot))\big]
		\end{gather*}
		where we used the estimate of~\eqref{ieq:estErrorPPn}
		with $v(\cdot) = f(\cdot,U(\cdot))$.
		This results in
		\begin{align*}
			\big\|\mathrm{(II)}\big\|
				& \leq C \frac{\tau_n}{2} \tau_n^{\min\{\bar{r},r,\rIfII+1\}+r-k+1} \\
				& \qquad \sup_{\tilde{s} \in I_n} \left\| \frac{\partial^{r-k+1}}{\partial s^{r-k+1}} R(t_n^-,\tilde{s}) \right\|
					\sum_{j=0}^{\max\{\kJJ, \min\{\bar{r},r,\rIfII+1\}\}}
						\sup_{s\in I_n} \left\| \frac{\d^j}{\d s^j} f(s,U(s)) \right\|.
		\end{align*}
		
		In order to estimate $\mathrm{(I)}$ for $0 \leq i \leq r-k$, we split the term as
		\begin{align}
			\mathrm{(I)}
				& = \int_{I_n} (s-t_{n-1})^i \big(U'(s)-f(s,U(s))\big) \,\d s
					+ \delta_{0,k} \delta_{0,i} \big[U \big]_{n-1} \nonumber \\
				& = \left(\int_{I_n} (s-t_{n-1})^i \big(U'(s) - f(s,U(s))\big)\,\d s
					- \IIn{(s-t_{n-1})^i \big(U'(s) - f(s,U(s))\big)} \right) \nonumber \\
				& \qquad + \IIn{(s-t_{n-1})^i \big(U'(s)-f(s,U(s))\big)}
					+ \delta_{0,k} \delta_{0,i} \big[U \big]_{n-1} \nonumber \\
				& = \left(\int_{I_n} (s-t_{n-1})^i \big(U'(s) - f(s,U(s))\big)\,\d s
					- \IIn{(s-t_{n-1})^i \big(U'(s) - f(s,U(s))\big)} \right) \nonumber \\
				& \qquad + \IIn{(s-t_{n-1})^i \big(\If_n - \Id\big)f(s,U(s))} \label{help:Split(i)} % \\
%				& = \left(\int_{I_n} (s-t_{n-1})^i \big(U'(s) - f(s,U(s))\big)\,\d s
%					- \IIn{(s-t_{n-1})^i \big(U'(s) - f(s,U(s))\big)} \right) \nonumber \\
%				& \qquad + \left(\int_{I_n} (s-t_{n-1})^i \big(\Id - \If_n\big) f(s,U(s))\,\d s
%					- \IIn{(s-t_{n-1})^i \big(\Id - \If_n\big) f(s,U(s))} \right) \nonumber \\
%				& \qquad - \int_{I_n} (s-t_{n-1})^i \big(\Id - \If_n\big) f(s,U(s))\,\d s, \nonumber
		\end{align}
		where~\eqref{eq:locProbGenVar} was used for the last identity.
		The first given difference on the right-hand side can be bounded by~\eqref{ieq:IntErrorIIn}.
		We gain, additionally applying Leibniz' rule for the $j$th derivative,
		\begin{align*}
			& \left\|\int_{I_n} (s-t_{n-1})^i \big(U'(s) - f(s,U(s))\big)\,\d s
				- \IIn{(s-t_{n-1})^i \big(U'(s) - f(s,U(s))\big)} \right\| \\
			& \quad \leq C \frac{\tau_n}{2} \sum_{j=\min\{\bar{r},\rExII+1\}}^{\max\{\kII,\min\{\bar{r},\rExII+1\}\}}
				\left(\tfrac{\tau_n}{2}\right)^j
				\sup_{s \in I_n} \left\| \frac{\d^j}{\d s^j} (s-t_{n-1})^i \big(U'(s) - f(s,U(s))\big) \right\| \\
			& \quad \leq C \frac{\tau_n}{2} \sum_{j=\min\{\bar{r},\rExII+1\}}^{\max\{\kII,\min\{\bar{r},\rExII+1\}\}}
				\!\!\left(\tfrac{\tau_n}{2}\right)^j
				\!\!\!\sum_{m=\max\{0,j-i\}}^j \!\!\!\tbinom{j}{m}
					\sup_{s \in I_n} \left\| \frac{\d^m}{\d s^m} \big(U'(s) - f(s,U(s))\big) \right\| \\
			& \hspace{25em} \sup_{s \in I_n} \big\| \tfrac{i!}{(i-j+m)!} (s-t_{n-1})^{i-j+m} \big\| \\
			& \quad \leq C \frac{\tau_n}{2} \sum_{j=\min\{\bar{r},\rExII+1\}}^{\max\{\kII,\min\{\bar{r},\rExII+1\}\}}
				\sum_{m=\max\{0,j-i\}}^j \tau_n^{m+i} \sup_{s \in I_n} \left\| \frac{\d^m}{\d s^m} \big(U'(s) - f(s,U(s))\big) \right\|\!.
		\end{align*}
		Recalling~\eqref{eq:derUisApproxOfF}, the last supremum could be further
		estimated by~\eqref{ieq:estErrorPPn} with $v(\cdot) = f(\cdot,U(\cdot))$.
		More detailed, it holds
		\begin{align*}
			& \sup_{s \in I_n} \left\| \frac{\d^m}{\d s^m} \big(U'(s) - f(s,U(s))\big) \right\| \\
			& \quad \leq \begin{cases}
					\sup_{s \in I_n} \left\| \frac{\d^m}{\d s^m} \big(\PPn - \Id \big) f(s,U(s)) \right\|
					\leq \KAT_{\eqref{ieq:estErrorPPn}}^{\check{r},m}\big[f(\cdot,U(\cdot))\big],
						& \text{if } m \leq \check{r} = \min \{\bar{r},r,\rIfII+1\}, \\
					\sup_{s \in I_n} \left(\left\|U^{(m+1)}(s)\right\|
						+ \left\| \frac{\d^m}{\d s^m} f(s,U(s)) \right\| \right),
						& \text{otherwise.}
				\end{cases}
		\end{align*}
		Note that disregarding regularity aspects the choice $\check{r} = \min \{\bar{r},r,\rIfII+1\}$
		in Lemma~\ref{le:approxPPn} is allowed.
		Therefore, we have
		\begin{align*}
			& \left\|\int_{I_n} (s-t_{n-1})^i \big(U'(s) - f(s,U(s))\big)\,\d s
				- \IIn{(s-t_{n-1})^i \big(U'(s) - f(s,U(s))\big)} \right\| \\
			& \quad \leq C \frac{\tau_n}{2} \tau_n^{\max\{\min\{\bar{r},\rExII+1\},\,\min\{\bar{r},r,\rIfII+1\}+i\}} \\
			& \quad \qquad \sum_{j=0}^{\max\{\kJJ,\,\min\{\bar{r},r,\rIfII+1\},\,\min\{\bar{r},\rExII+1\}\}}
					\hspace{-1em}\sup_{s \in I_n} \!\left( \left\| U^{(j+1)}(s)\right\|
					+ \left\| \frac{\d^j}{\d s^j} f(s,U(s)) \right\| \right)\!.
		\end{align*}
		Factoring out $\tau_n^{i}$, the second term on the right-hand side of~\eqref{help:Split(i)} can be
		analyzed by Lemma~\ref{le:auxLemma} with $\varphi(t) = f(t,U(t))$ and
		$\psi_i(t) = \big(\frac{t-t_{n-1}}{\tau_n}\big)^i$. Then we obtain
		\begin{align*}
			& \left\|\IIn{(s-t_{n-1})^i \big(\Id - \If_n\big)f(s,U(s))}\right\|
				\leq C \frac{\tau_n}{2} \tau_n^{i}
				\sum_{j=j_{\mathrm{min},i,\bar{r}}^*}^{j_{\mathrm{max},i,\bar{r}}^*}
					\tau_n^{j} \sup_{s\in I_n} \left\| \frac{\d^j}{\d s^j} f(s,U(s)) \right\| \\
			& \quad \leq C \frac{\tau_n}{2} \tau_n^{\min\{\bar{r},\rIfIIi{i}+1\}+i}
				\sum_{j=0}^{j_{\mathrm{max},i,\bar{r}}^*-j_{\mathrm{min},i,\bar{r}}^*}
					\tau_n^{j} \sup_{s\in I_n}
						\left\| \frac{\d^{j_{\mathrm{min},i,\bar{r}}^*+j}}{\d s^{j_{\mathrm{min},i,\bar{r}}^*+j}} f(s,U(s)) \right\|
		\end{align*}
		with a constant $C$ independent of $\tau_n$ where
		\begin{gather*}
			j_{\mathrm{min},i,\bar{r}}^* = \min\{\bar{r},\rIfIIi{i}+1\},
			\qquad
			j_{\mathrm{max},i,\bar{r}}^* = \max\{\kII,\kIf,j_{\mathrm{min},i,\bar{r}}^*\}.
		\end{gather*}
		So, altogether we have
		\begin{align*}
			\left\|\mathrm{(I)}\right\|
				& \leq C \frac{\tau_n}{2} \tau_n^{r_{i,\bar{r}}^{\triangleright \triangleleft}}
				\sum_{j=0}^{j_{\mathrm{max},i,\bar{r}}^{\triangleright \triangleleft}}
					\left(\sup_{s \in I_n} \left\| U^{(j+1)}(s)\right\|
						+ \sup_{s \in I_n} \left\| \frac{\d^j}{\d s^j} f(s,U(s)) \right\| \right)
		\end{align*}
		where for abbreviation we set
		\begin{gather*}
			r_{i,\bar{r}}^{\triangleright \triangleleft}
				:= \min\{\bar{r}+i,\,\rIfIIi{i}+i+1,\,\max\{\min\{\bar{r},\rExII+1\},\,\min\{r,\rIfII+1\}+i\}\}, \\
			j_{\mathrm{max},i,\bar{r}}^{\triangleright \triangleleft}
				:= \max\{\kJJ, \min\{\bar{r},\rIfIIi{i}+1\}, \min\{\bar{r},r,\rIfII+1\}, \min\{\bar{r},\rExII+1\}\}.
		\end{gather*}
		
		Combining~\eqref{eq:helpSuperSplit} with the above estimates for $\mathrm{(I)}$
		and $\mathrm{(II)}$, we gain for $\bar{r}$ sufficiently large
		\begin{multline}
			\label{ieq:superRestTerms}
			\left\| \delta_{0,k} R(t_n^-,t_{n-1}^+) \big[U \big]_{n-1}
				+ \int_{I_{n}} R(t_n^-,s)\df(s)\,\d s \right\| \\
			\leq C \frac{\tau_n}{2}
				\tau_n^{\min\{2r-k+1,(\rIfII+1)+r-k+1,\max\{\rExII+1,\min\{r,\rIfII+1\}\},\rIfIIVar+1\}} G_n
		\end{multline}
		with $\rIfIIVar = \min_{0 \leq i \leq r-k}\{\rIfIIi{i}+i\}$ and
		\begin{gather*}
			G_n := \sum_{i=0}^{r-k+1}
					\sup_{\tilde{s} \in I_n} \left\| \frac{\partial^{i}}{\partial s^{i}} R(t_n^-,\tilde{s}) \right\|
				\sum_{j=0}^{j_{\mathrm{max},i,\bar{r}}^{\triangleright \triangleleft}}
					\left(\sup_{s \in I_n} \left\| U^{(j+1)}(s)\right\|
						+ \sup_{s \in I_n} \left\| \frac{\d^j}{\d s^j} f(s,U(s)) \right\| \right)\!.
		\end{gather*}
		Here note that because of $\rIfIIVar \leq \rIfIIi{r-k}+r-k =\rIfII +r-k$ the second
		term in the minimum on the right-hand side of~\eqref{ieq:superRestTerms} could be dropped.
		Therefore, incorporating~\eqref{ieq:superInitialError}, \eqref{ieq:superSquareError},
		and~\eqref{ieq:superRestTerms} in~\eqref{eq:superDecomposition} gives
		\begin{multline*}
			\left\|e(t_n^-)\right\|
				\leq \bigg(1+\tau_n \sup_{\tilde{s} \in I_n} \bigg\| \frac{\partial}{\partial t} R(\tilde{s},t_{n-1}^+)\bigg\| \bigg)
						\big\|e(t_{n-1}^-)\big\|
					+ C \tau_n \sup_{s \in I_n} \big\|e(s)\big\|^2 \\
					+ C \frac{\tau_n}{2}
						\tau_n^{\min\{2r-k+1,\max\{\rExII+1,\min\{r,\rIfII+1\}\},\rIfIIVar+1\}} G_n.
		\end{multline*}
		A variant of the discrete Gronwall lemma, see~\cite[Proposition~3.3]{Emm99}, applied with
		\begin{align*}
			a_n & = \left\|e(t_n^-)\right\|, \quad  \theta = 0, \quad
			\lambda_n = \sup_{\tilde{s} \in I_n} \big\| \tfrac{\partial}{\partial t} R(\tilde{s},t_{n-1}^+)\big\|, \\
			g_n & = C \sup_{s \in I_n} \big\|e(s)\big\|^2
			+ C/2 \,\tau_n^{\min\{2r-k+1,\max\{\rExII+1,\min\{r,\rIfII+1\}\},\rIfIIVar+1\}} G_n,
		\end{align*}
		then yields
		\begin{align*}
			& \left\|e(t_n^-)\right\|
				\leq \left\|e(t_0^-)\right\| \Bigg(\prod_{j=1}^n \left(1+\lambda_j \tau_j\right)\!\Bigg) \\
			& \quad + C \sum_{\nu=1}^n \tau_\nu \left(\sup_{s \in I_\nu}\big\|e(s)\big\|^2
					+ \frac{1}{2}
					\tau_\nu^{\min\{2r-k+1,\max\{\rExII+1,\min\{r,\rIfII+1\}\},\rIfIIVar+1\}} G_\nu \right)
					\!\Bigg(\prod_{j=\nu}^n \left(1+\lambda_j \tau_j\right) \!\Bigg).
		\end{align*}
		From the well known inequality $(1+x) \leq e^x$, it follows
		\begin{align*}
			\prod_{j=\nu}^n \left(1+\lambda_j \tau_j\right)
				\leq \prod_{j=1}^n \left(1+\lambda_j \tau_j \right)
				\leq \exp\left(\sum_{\nu=1}^n \lambda_\nu \tau_\nu \right)
				\leq \exp\left((t_n-t_0) \max_{\nu=1,\ldots,n} \lambda_\nu \right)\!.
		\end{align*}
		Hence, we conclude that
		\begin{align*}
			\left\|e(t_n^-)\right\|
				& \leq C \left(t_n-t_0\right) \exp\left((t_n-t_0) \max_{\nu=1,\ldots,n} \lambda_\nu \right) \\
				& \qquad \quad \max_{\nu=1,\ldots,n} \left(\sup_{s \in I_\nu}\big\|e(s)\big\|^2
					+ \tau_\nu^{\min\{2r-k+1,\max\{\rExII+1,\min\{r,\rIfII+1\}\},\rIfIIVar+1\}} G_\nu\right)\!
		\end{align*}
		where we also used $e(t_0^-) = 0$.
		
		It remains a small technical detail to verify that $G_\nu$ can be uniformly
		bounded independent of $\tau_\nu$. The term depends on partial derivatives
		of $f$, derivatives of $R$, and on the derivatives of the discrete solution
		$U$, thus, potentially also on the mesh parameter. However, $U$ can
		be uniformly bounded by assumption. So, we are done.
	\end{proof}
	
	\begin{remark}
		Using an alternative argument (inspired by the proof of~\cite[Theorem~II.7.9, pp.~212/213]{HNW08})
		that is based on the application of the nonlinear variation-of-constants
		formula~\cite[Corollary~I.14.6, p.~97]{HNW08}, it can be shown that
		for $1 \leq k \leq r$ the term $\sup_{t \in I} \big\|(u-U)(t)\big\|^2$
		in~\eqref{Eq:ErrSupTheo} is not necessary and can be dropped.
		
		However, for $k=0$ the alternative proof is much more complicated and
		in general only guarantees a worse superconvergence estimate than
		Theorem~\ref{th:superconvergence}. Moreover, for all $k$ the notation
		gets more involved.
	\end{remark}
	
	\begin{lemma}
		\label{le:boundU}
		Suppose that Assumption~\ref{assIIhat} holds along with an estimate
		similar to~\eqref{ieq:estErrorPPn} that at least guarantees approximation
		order $r-1$ for $\PPn$ (e.g. if $\rIfII \geq r-2$). In addition, let the solutions
		$u$ of~\eqref{initValueProb} and $U$ of~\eqref{eq:locProbGen} satisfy
		$\sup_{t \in I_n} \| (u - U)(t) \| \leq C$ 
		for some constant $C$ independent of the mesh parameter. Then we have
		\begin{gather*}
			\sup_{t \in I_n} \big\| U^{(l)}(t) \big\| \leq C(f,u)
		\end{gather*}
		for all $0 \leq l \leq r$.
	\end{lemma}
	\begin{proof}
		We first of all note that 
		the assumed estimate for the local error implies
		\begin{gather*}
			\sup_{t \in I_n} \big\| U(t) \big\|
				\leq \sup_{t \in I_n} \big\| u(t) \big\| + \sup_{t \in I_n} \big\| u(t) - U(t) \big\|
				\leq C(f,u).
		\end{gather*}
		Using this as basis we can prove the wanted estimates by induction, exploiting
		the identity~\eqref{eq:derUisApproxOfF}. For $0 \leq l \leq r-1$ we proceed as follows
		\begin{align*}
			\sup_{t \in I_n} \big\| U^{(l+1)}(t) \big\|
			& = \sup_{t \in I_n} \left\| \frac{\d^l}{\d t^l} \PPn f(t,U(t)) \right\| \\
			& \leq \sup_{t \in I_n} \left\| \frac{\d^l}{\d t^l} f(t,U(t)) \right\|
				+ \sup_{t \in I_n} \left\| \frac{\d^l}{\d t^l} \big(\Id - \PPn\big) f(t,U(t)) \right\|.
		\end{align*}
		The first term on the right-hand side contains derivatives of $U$ up to maximal
		order $l$, so can be bounded by $C(f,u)$ due to the induction hypothesis.
		It remains to study the second term.
		
		Using~\eqref{ieq:estErrorPPn} we obtain
		\begin{gather}
			\label{help:UniformUBound}
			\sup_{t \in I_n} \left\| \frac{\d^l}{\d t^l} \big(\Id - \PPn\big) f(t,U(t)) \right\|
				\leq C \sum_{j=j_{\mathrm{min}}^{\bullet}}^{j_{\mathrm{max}}^{\bullet}}
					\tau_n^{j-l} \sup_{t\in I_n} \left\| \frac{\d^j}{\d t^j} f(t,U(t)) \right\|
		\end{gather}
		with some non-negative integer constants
		$r-1 \leq j_{\mathrm{min}}^{\bullet} \leq j_{\mathrm{max}}^{\bullet}$.
		Now, according to a generalization of Fa\`{a} di Bruno's formula,
		see~\cite[Theorem~2.1]{M00} for details, we know that
		\begin{align*}
			\frac{\d^j}{\d t^j} f(t,U(t))
				= \sum_{Q} C_{Q} %\sum_0 \cdots \sum_j C_{0,\ldots,j}
					%\frac{j!}{\prod_{i=1}^j (i!)^{\sum_{m=0}^d q_{im}} \prod_{i=1}^j\prod_{m=0}^d q_{im}!}
					\frac{\partial^{\sum_{m = 0}^d p_m} f}{\partial t^{p_0} \partial u_1^{p_1} \cdots \partial u_d^{p_d}}(t,U(t))
					\prod_{i=1}^j \left(\big(\tfrac{\d^i}{\d t^i} t\big)^{q_{i0}} \prod_{m=1}^d \big(U_m^{(i)}(t)\big)^{q_{im}}\right)
		\end{align*}
		where the sum is over the non-negative integer solutions
		$Q = (q_{im})_{1 \leq i \leq j,\, 0 \leq m \leq d}$ of
		\begin{gather}
			\label{help:Diophantine}
			\sum_{i=1}^j i \left(\sum_{m=0}^d q_{im}\right) = j
		\end{gather}
		and
		\begin{gather*}
			C_Q := \frac{j!}{\prod_{i=1}^j (i!)^{\sum_{m=0}^d q_{im}} \prod_{i=1}^j\prod_{m=0}^d q_{im}!},\qquad
			p_m := \sum_{i=1}^j q_{im}.
		\end{gather*}
		Note that this formula
		especially implies that derivatives of $U$ appear up to maximal order $j$.
		
		Now, taking into account that $\tfrac{\d^i}{\d t^i} t = \delta_{1,i}$, $i \geq 1$, the above
		summands can be bounded by
		\begin{gather*}
			C \bigg\|\frac{\partial^{\sum_{m = 0}^d p_m} f}{\partial t^{p_0} \partial u_1^{p_1} \cdots \partial u_d^{p_d}}(t,U(t))\bigg\|
					\prod_{i=1}^j \prod_{m=1}^d \big| U_m^{(i)}(t)\big|^{q_{im}}
		\end{gather*}
		where the double product can be further estimated as
		\begin{gather*}
			\prod_{i=1}^j \prod_{m=1}^d \big| U_m^{(i)}(t)\big|^{q_{im}}
				\leq \prod_{i=1}^j \prod_{m=1}^d \big\| U^{(i)}(t)\big\|^{q_{im}}
				= \prod_{i=1}^j \big\| U^{(i)}(t)\big\|^{\sum_{m=1}^d q_{im}}.
		\end{gather*}
		So, applying an inverse inequality multiple times, more precisely at most $j$ times
		because of~\eqref{help:Diophantine}, we obtain
		\begin{multline*}
			\prod_{i=1}^j \big\| U^{(i)}(t)\big\|^{\sum_{m=1}^d q_{im}}
				\leq \prod_{i=1}^j \left(C_\mathrm{inv} \left(\tfrac{\tau_n}{2}\right)^{-i}
					\sup_{t \in I_n} \big\| U(t)\big\| \right)^{\sum_{m=1}^d q_{im}} \\
				= C \left(\tfrac{\tau_n}{2}\right)^{- \overbrace{\scriptstyle \sum_{i=1}^j i(\sum_{m=1}^d q_{im})}^{\leq j}}
					\sup_{t \in I_n} \big\| U(t)\big\|.
		\end{multline*}
		Note that this also implies that after applying an inverse inequality at most $j-l$ times,
		the right-hand side only depends on derivatives of $U$ of order less than or equal to $l$.
		
		Therefore the summands on the right-hand side of~\eqref{help:UniformUBound}
		can be estimated (independent of $\tau_n$) by terms that contain derivatives of $U$
		up to maximal order $l$. Hence, again by the induction hypotheses we gain the upper
		bound $C(f,u)$.
	\end{proof}
	
	Summarizing the above observations, our analysis guarantees the following
	estimates in the time mesh points.

	\begin{corollary}
		\label{Cor:ErrSup}
		Let $r,k \in \ZZ$, $0\leq k \leq r$. Suppose that
		Assumptions~\ref{assIIhat}, \ref{assSupII}, \ref{assSup}, and~\ref{assDiff}
		hold. Moreover, let Assumption~\ref{assPoint_a} or~\ref{assPoint_b} be
		satisfied.
		Let $u$ and $U$ denote the solutions of~\eqref{initValueProb}
		and~\eqref{eq:locProbGen}, respectively.
		Then, if $\rIfII \geq r-2$, we have for $1 \leq n \leq N$
		\begin{gather}
			\label{Eq:ErrSup}
			\big\| (u-U)(t_n^-) \big\|
			\leq C \Big(\tau^{\min\{2r-k+1,\rIfIIVar+1,\max\{\rExII+1,\min\{r,\rIfII+1\}\}\}}
				+ \delta_{0,k} \tau^{2\rIfII+4} \Big),
		\end{gather}
		with $\rIfIIVar := \min_{0 \leq i \leq r-k}\{\rIfIIi{i}+i\}$, $\rIfII = \rIfIIi{r-k}$,
		and $\rIfIIi{i}$ as defined in Definition~\ref{def:approxOrders}.
		
		If $\rIfII < r-2$, we in general cannot ensure the uniform boundedness
		of $U$ and its derivatives since Lemma~\ref{le:boundU} does not hold.
		Then we only have
		\begin{gather*}
			\big\| (u-U)(t_n^-) \big\| \leq \sup_{t \in I_n} \big\| (u-U)(t) \big\|
		\end{gather*}
		where we refer to Corollary~\ref{Cor:ErrEst} for bounds on the right-hand side term.
	\end{corollary}
	
%%%%%%%%%%%%%%%%%%%%%%%%%%%%%%%%%%%%%%%%%%%%%%%%%%%%%%%%%%%%%%%%%%%%%%%%%%%%%%%
%%%%%%%%%%%%%%%%%%%%%%%%%%%%%%%%%%%%%%%%%%%%%%%%%%%%%%%%%%%%%%%%%%%%%%%%%%%%%%%
	\section{Numerical experiments}
	\label{sec:numerics}
	
	We consider the initial value problem
	\begin{equation*}
		\begin{pmatrix}u_1'(t)\\u_2'(t)\end{pmatrix}
		= \begin{pmatrix} -u_1^2(t)-u_2(t)\\ u_1(t)-u_1(t)u_2(t)\end{pmatrix},
		\quad t\in(0,32),\qquad 
		u(0) = \begin{pmatrix} 1/2 \\ 0\end{pmatrix}\!,
	\end{equation*}
	of a system of nonlinear ordinary differential equations which has
	\begin{equation*}
	u_1(t) = \frac{\cos t}{2+\sin t},\qquad u_2(t) = \frac{\sin t}{2+\sin t}
	\end{equation*}
	as solution.
	
	The appearing nonlinear systems within each time step were solved by
	Newton's method where we applied a Taylor expansion of the inherited data
	from the previous time interval to calculate an initial guess for all
	unknowns on the current interval. If higher order derivatives were needed
	at initial time $t_0=0$, we apply
	\begin{equation*}
	\begin{alignedat}{2}
	u^{(0)}(t_0) & := u_0,&\qquad
	u^{(2)}(t_0) & := \partial_t f\big(t_0, u(t_0)\big) + \partial_u f\big(t_0,u(t_0)\big) u^{(1)}(t_0), \\
	                u^{(1)}(t_0) & := f\big(t_0,u(t_0)\big), &\qquad
	u^{(j)}(t_0) & := \frac{\d^{j-1}}{\d t^{j-1}} f\big(t,u(t)\big)
	\big|_{t=t_0}, \quad j \geq 3,
	\end{alignedat}
	\end{equation*}
	based on the ode system~\eqref{initValueProb} and its derivatives.
	
	By considering different choices for $\II$ and $\If_n$, we will show that
	our theory provides sharp bounds on the convergence order. Since $\II$ and $\If_n$ are
	obtained from $\IIhat$ and $\Ifhat$ via transformation, only the reference operators
	will be specified.
	
    Each integrator $\IIhat$ that has been used in our calculations is based on
    Lagrangian interpolation with respect to a specific node set $P_{\IIhat}$. Hence,
	we have $\kII=0$. The interpolation operator $\Ifhat$ is of Lagrange-type and
	uses the node set $P_{\Ifhat}$. This means that $\kIf=0$.
	Both node sets will be given for each of our test
	cases. Since often nodes of quadrature formulas are used, we will write for
	instance "left Gauss--Radau($k$)" to indicate that the nodes of the
	left-sided Gauss--Radau formula with $k$ points have been used. All
	upcoming settings fulfill Assumption~\ref{assIIhat}.
	
	For all test cases, the method $\vtd{6}{3}$, which is $\mathrm{cGP}\text{-}C^1(6)$,
	was applied as discretization. All calculations were carried out with the software
	Julia~\cite{julia} using the floating point data type \texttt{BigFloat} with 512
	bits. Errors will be measured in the norms
	\begin{gather*}
		\|\varphi\|_{L^\infty} := \sup_{t\in I} \|\varphi(t)\|,\qquad
		\|\varphi\|_{\ell^\infty} := \max_{1\le n\le N} \|\varphi(t_n^-)\|
	\end{gather*}
	where $\|\cdot\|$ denotes the Euclidean norm in $\RR^d$.
	
	\subsection{Case group 1}
	\begin{case} % case 1
	\label{case1}
	Choosing integration and interpolation according to
	\begin{equation*}
		\textstyle
		P_{\IIhat} = \left\{-\frac{3}{4}, -\frac{1}{4}, \frac{1}{4}, \frac{3}{4}\right\}\!,
		\qquad
		P_{\Ifhat} = \text{left Gauss--Radau($3$)}
	\end{equation*}
	leads to $\rExII=3$ and $\rIfII=2$. Hence, the condition
	$\max\{\rExII,\rIfII+1\} \geq r-1$ in Corollary~\ref{Cor:ErrEst} is violated and
	the expected convergence orders for both the $L^\infty$ norm and the
	$W^{1,\infty}$ seminorm are given by
	$\min\{r,\rIfII+1\}=3$, see~\eqref{Eq:ErrEst1}.
	\begin{table}[htb!]
	\begin{center}
	\caption{Errors and convergence orders in different norms for case~\ref{case1}.}
	\label{Tab:case1}
	\begin{tabular}{rclclcl}
	\toprule
	$N$ & $\|e\|_{L^\infty}$ & ord & $\|e'\|_{L^\infty}$ & ord & $\|e\|_{\ell^\infty}$ & ord\\
	\cmidrule(lr){1-1}
	\cmidrule(lr){2-3}
	\cmidrule(lr){4-5}
	\cmidrule(lr){6-7}
	   32 & 1.648-03 &      & 1.126-02 &      & 9.439-04 &      \\
	   64 & 1.413-04 & 3.54 & 1.423-03 & 2.98 & 1.046-04 & 3.17 \\
	  128 & 1.408-05 & 3.33 & 1.876-04 & 2.92 & 1.264-05 & 3.05 \\
	  256 & 1.615-06 & 3.12 & 2.369-05 & 2.99 & 1.559-06 & 3.02 \\
	  512 & 1.961-07 & 3.04 & 2.965-06 & 3.00 & 1.937-07 & 3.01 \\
	 1024 & 2.426-08 & 3.01 & 3.711-07 & 3.00 & 2.415-08 & 3.00 \\
	\cmidrule(lr){1-1}
	\cmidrule(lr){2-3}
	\cmidrule(lr){4-5}
	\cmidrule(lr){6-7}
	 theo &          & 3    &          & 3    &          & 3\\
	\bottomrule
	\end{tabular}
	\end{center}
	\end{table}
	It can be seen from Table~\ref{Tab:case1} that the theoretical predictions
	are met by the numerical experiments. Moreover, in accordance with
	Corollary~\ref{Cor:ErrSup} the $\ell^\infty$ convergence order is also 
	just $3$. This means that the uniform boundedness of $\sup_{t\in I}\|U^{(l)}(t)\|$
	required by Theorem~\ref{th:superconvergence}, which cannot be guaranteed
	because of $\rIf+1 < r-1$, is violated since otherwise~\eqref{Eq:ErrSupTheo}
	would yield order $4$.
	\end{case}
	
	The condition $\max\{\rExII,\rIfII+1\} \geq r-1$ of Corollary~\ref{Cor:ErrEst}
	will be fulfilled for all coming cases. Hence, the computations should show
	the convergence order given by~\eqref{Eq:ErrEst2} for the
	$L^\infty$ norm.
	
	\subsection{Case group 2}
	This group of cases provides choices for $P_{\IIhat}$ and $P_{\Ifhat}$ such
	that the $L^\infty$ convergence order is limited by the maximum expression
	inside the outer minimum in~\eqref{Eq:ErrEst2}. In addition, the presented
	cases will show that each of the three terms occuring in the maximum term
	can limit the convergence order. We will indicate in the following the
	limiting term in boldface.
	
	\begin{subtheorem}{case} % cases 2a-2c
	\begin{case} % case 2a
	\label{case2a}
	The choices 
	\begin{equation*}
		\textstyle
		P_{\IIhat} = \left\{-\frac{3}{4}, -\frac{1}{4}, \frac{1}{4}, \frac{3}{4}\right\}\!,
		\qquad
		P_{\Ifhat} = \text{Gauss($5$)}
	\end{equation*}
	provide the $L^\infty$ convergence order
	$\min\{7,6,6,\max\{4,\min\{6,\boldsymbol{5}\}\}\}=5$
	where the convergence order is limited by the second term inside the inner
	minimum.
	\begin{table}[htb!]
	\begin{center}
	\caption{Errors and convergence orders in the $L^\infty$ norm for the cases of group~2.}
	\label{Tab:case2}
	\begin{tabular}{rclclclcl}
	\toprule
	 & \multicolumn{2}{c}{case 2a} &
	   \multicolumn{2}{c}{case 2a*} &
	   \multicolumn{2}{c}{case 2b} &
	   \multicolumn{2}{c}{case 2c} \\
	$N$ & $\|e\|_{L^\infty}$ & ord & $\|e\|_{L^\infty}$
	& ord & $\|e\|_{L^\infty}$ & ord & $\|e\|_{L^\infty}$ & ord\\
	\cmidrule(lr){1-1}
	\cmidrule(lr){2-3}
	\cmidrule(lr){4-5}
	\cmidrule(lr){6-7}
	\cmidrule(lr){8-9}
	   32 & 8.482-06 &      & 7.745e-06 &      & 8.910-07 &      & 1.996-06 &     \\
	   64 & 1.477-07 & 5.84 & 1.500e-07 & 5.69 & 9.394-09 & 6.57 & 2.792-08 & 6.16\\
	  128 & 2.398-09 & 5.95 & 3.498e-09 & 5.42 & 1.081-10 & 6.44 & 4.251-10 & 6.04\\
	  256 & 3.759-11 & 6.00 & 9.412e-11 & 5.22 & 1.465-12 & 6.21 & 6.604-12 & 6.01\\
	  512 & 5.862-13 & 6.00 & 2.775e-12 & 5.08 & 2.175-14 & 6.07 & 1.034-13 & 6.00\\
	 1024 & 9.160-15 & 6.00 & 8.508e-14 & 5.03 & 3.344-16 & 6.02 & 1.617-15 & 6.00\\
	\cmidrule(lr){1-1}
	\cmidrule(lr){2-3}
	\cmidrule(lr){4-5}
	\cmidrule(lr){6-7}
	\cmidrule(lr){8-9}
	 theo &          & 5    &          & 5    &          & 6    &          & 6\\
	\bottomrule
	\end{tabular}
	\end{center}
	\end{table}
	We see from Table~\ref{Tab:case2} that the experimental order of convergence
	is $6$, i.e., one order higher than expected. This behavior can be
	explained by a closer look to Lemma~\ref{le:errorJJnPoint}. The proof there
	guarantees in this case that $(v - \JJn v)(t_n^-)=0$ for all $v \in P_5(I_n)$. However,
	due to symmetry reasons, it holds $\int_{I_n} (v-\JJn v)'(t)\,\d t = 0$ for
	all $v \in P_6(I_n)$ which implies that $(v - \JJn v)(t_n^-)=0$ even for
	$v \in P_6(I_n)$. Thus, the convergence order of the limiting term is actually
	better than predicted.
	
	Taking the same setting for $P_{\IIhat}$ but using
	\begin{equation*}
		\textstyle P_{\Ifhat} = \left\{ -\frac{5}{6}, -\frac{13}{23}, \frac{1}{10}, \frac{12}{17}, \frac{4}{5} \right\}\!,
	\end{equation*}
	the convergence order predicted by~\eqref{Eq:ErrEst2} is again
	$\min\{7,6,6,\max\{4,\min\{6,\boldsymbol{5}\}\}\}=5$. The limitation is here also
	caused by the second argument of the inner minimum.
	Table~\ref{Tab:case2} shows under case~\ref{case2a}* that this convergence
	order is obtained in the numerical experiments.
	Note that here the interpolation points just were chosen such that still
	$\rIfIIi{0}=5 > 4 = \rIfII$, i.e., especially
	$\IIhat\big[(\hat{v}-\JJhat \hat{v})'\big] = \IIhat\big[\hat{v}'- \Ifhat( \hat{v}')\big] =0$
	for $\hat{v}(\,\hat{t}\,)=\hat{t}\,^6$, but
	$\int_{-1}^1 (\hat{v}-\JJhat \hat{v})'(\,\hat{t}\,) \,\d\hat{t}\neq 0$ for
	$\hat{v}(\,\hat{t}\,)=\hat{t}\,^6$ which ensures that the estimate of
	Lemma~\ref{le:errorJJnPoint} is sharp.
	\end{case}

	\begin{case} % case 2b
	\label{case2b}
	If we set
	\begin{equation*}
		\textstyle
		P_{\IIhat} = \left\{-\frac{3}{4}, -\frac{1}{4}, \frac{1}{4}, \frac{3}{4}\right\}\!,
		\qquad
		P_{\Ifhat} = \left\{-\frac{3}{4}, -\frac{1}{4}, \frac{1}{4}, \frac{3}{4}\right\}\!,
	\end{equation*}
	the expected convergence order is
	$\min\{7,\infty,\infty, \max\{4,\min\{\boldsymbol{6},\infty\}\}\} = 6$ where the
	limitation comes from the first argument of the inner minimum. The numerical
	results given in Table~\ref{Tab:case2} clearly show this convergence order.
	\end{case}
	
	\begin{case} % case 2c
	\label{case2c}
	Choosing
	\begin{equation*}
		\textstyle
		P_{\IIhat} = \left\{ -1, -\frac{3}{5}, -\frac{1}{5}, \frac{1}{5}, \frac{3}{5}, 1 \right\}\!,
		\qquad
		\textstyle P_{\Ifhat} = \left\{ -1, -\frac{3}{5}, -\frac{1}{5}, \frac{1}{5}, \frac{3}{5}, 1 \right\}\!,
	\end{equation*}
	results in the convergence order
	$\min\{7,\infty,\infty, \max\{\boldsymbol{6},\min\{\boldsymbol{6},\infty\}\}\}=6$.
	The first argument of the maximum acts as limitation. The numerical results
	in Table~\ref{Tab:case2} show that the expected convergence order is
	obtained. Note that it is not possible that $\rExII+1$ is the only
	limiting term since the structure of~\eqref{Eq:ErrEst2} implies that
	$\min\{r+1,\,\rIfII+2\} \geq \rExII +1 \geq \min\{r,\rIfII+1\}$
	if $\rExII+1$ is limiting. Hence, the integer $\rExII+1$ 
	coincides either with $\min\{r+1,\,\rIfII+2\}$ or $\min\{r,\rIfII+1\}$.
	\end{case}
	\end{subtheorem}
	
	\subsection{Case group 3}
	This group of cases studies the convergence orders in the $L^\infty$ norm
	and the $W^{1,\infty}$ seminorm. The presented choices will show that each
	of the first three expressions in the outer minimum in~\eqref{Eq:ErrEst2}
	can bound the $L^\infty$ convergence order. Moreover, the cases will
	demonstrate that the convergence order in the $W^{1,\infty}$ seminorm can
	be limited by both occuring terms in~\eqref{Eq:ErrEst1}. Again, the
	limiting numbers will be given in boldface.
	
	\begin{subtheorem}{case} % cases 3a-3c
	\begin{case} % 3a
	\label{case3a}
	The choice
	\begin{equation*}
		\textstyle
		P_{\IIhat} = \text{Gauss($6$)},
		\qquad
		P_{\Ifhat} = \left\{ -1, -\frac{1}{2}, \frac{1}{4}, \frac{3}{4}, 1 \right\}
	\end{equation*}
	results in
	\begin{equation*}
		L^\infty\text{~order} = \min\{7,6,\boldsymbol{5},\max\{12,\min\{6,5\}\}\} = 5,
		\qquad W^{1,\infty}\text{~order} = \min\{6,\boldsymbol{5}\} = 5.
	\end{equation*}
	Hence, the third argument in the outer minimum determines the convergence
	order for the $L^\infty$ norm while the second argument of the minimum
	limits the convergence order of the $W^{1,\infty}$ seminorm. 
	\begin{table}[htb!]
	\begin{center}
	\caption{Errors and convergence orders in different norms for
	the cases~\ref{case3a} and~\ref{case3b}.}
	\label{Tab:case3ab}
	\begin{tabular}{rclclclcl}
	\toprule
	 & \multicolumn{4}{c}{case 3a} &
	   \multicolumn{4}{c}{case 3b} \\
	\cmidrule(lr){2-5}
	\cmidrule(lr){6-9}
	$N$ & $\|e\|_{L^\infty}$ & ord & $\|e'\|_{L^\infty}$ & ord
	& $\|e\|_{L^\infty}$ & ord & $\|e'\|_{L^\infty}$ & ord\\
	\cmidrule(lr){1-1}
	\cmidrule(lr){2-3}
	\cmidrule(lr){4-5}
	\cmidrule(lr){6-7}
	\cmidrule(lr){8-9}
	   32 & 2.306-05 &      & 3.193-04 &      & 1.242-03 &      & 1.717-02 &     \\
	   64 & 3.786-07 & 5.93 & 1.044-05 & 4.94 & 8.223-05 & 3.92 & 2.169-03 & 2.98\\
	  128 & 7.266-09 & 5.70 & 3.411-07 & 4.94 & 5.037-06 & 4.03 & 2.842-04 & 2.93\\
	  256 & 1.716-10 & 5.40 & 1.072-08 & 4.99 & 3.069-07 & 4.04 & 3.581-05 & 2.99\\
	  512 & 4.688-12 & 5.19 & 3.353-10 & 5.00 & 1.886-08 & 4.02 & 4.479-06 & 3.00\\
	 1024 & 1.400-13 & 5.07 & 1.048-11 & 5.00 & 1.168-09 & 4.01 & 5.604-07 & 3.00\\
	\cmidrule(lr){1-1}
	\cmidrule(lr){2-3}
	\cmidrule(lr){4-5}
	\cmidrule(lr){6-7}
	\cmidrule(lr){8-9}
	 theo &          & 5    &          & 5    &          & 4    &          & 3\\
	\bottomrule
	\end{tabular}
	\end{center}
	\end{table}
	The numerical results in Table~\ref{Tab:case3ab} provide the predicted
	convergence orders.
	\end{case}
	
	\begin{case} % 3b
	\label{case3b}
	Setting
	\begin{equation*}
		\textstyle
		P_{\IIhat} = \text{Gauss($6$)},
		\qquad
		P_{\Ifhat} = \text{left Gauss--Radau($3$)}
	\end{equation*}
	gives
	\begin{equation*}
		L^\infty\text{~order} =
		\min\{7,\boldsymbol{4},5,\max\{12,\min\{6,3\}\}\} = 4,
		\qquad
		W^{1,\infty}\text{~order} = \min\{6,\boldsymbol{3}\} = 3.
	\end{equation*}
	The convergence order in the $W^{1,\infty}$ seminorm is again determinded
	by the second argument of the corresponding minimum. The limitation of the
	convergence order of the $L^\infty$ norm is caused by the second term. We
	clearly see from Table~\ref{Tab:case3ab} that the expected convergence
	orders are obtained by the numerical simulations.
	\end{case}
	
	\begin{case} % 3c
	\label{case3c}
	If we take
	\begin{equation*}
		P_{\IIhat} = \text{Gauss($6$)},
		\qquad
		P_{\Ifhat} = \text{Gauss($5$)},
	\end{equation*}
	we get
	\begin{align*}
		L^\infty\text{~order} &= \min\{\boldsymbol{7},8,10,\max\{12,\min\{6,7\}\}\} = 7,
		\qquad 
		W^{1,\infty}\text{~order} = \min\{\boldsymbol{6},7\} = 6,\\
		\ell^\infty\text{~order} &=
		\min\{\boldsymbol{10},\boldsymbol{10},\max\{12,\min\{6,7\}\}\} = 10.
	\end{align*}
	The convergence orders of the $L^\infty$ norm and the $W^{1,\infty}$
	seminorm are limited by the first argument in the corresponding minimum
	expressions. The numerical results in Table~\ref{Tab:case3c} indicate that
	the predicted orders are achieved.
	\begin{table}[htb!]
	\begin{center}
	\caption{Errors and convergence orders in different norms for case~\ref{case3c}.}
	\label{Tab:case3c}
	\begin{tabular}{rclclcl}
	\toprule
	$N$ & $\|e\|_{L^\infty}$ & ord & $\|e'\|_{L^\infty}$ & ord & $\|e\|_{\ell^\infty}$ & ord\\
	\cmidrule(lr){1-1}
	\cmidrule(lr){2-3}
	\cmidrule(lr){4-5}
	\cmidrule(lr){6-7}
	   32 & 7.603-07 &      & 1.669-05 &      & 7.587-10 &      \\
	   64 & 6.524-09 & 6.86 & 2.870-07 & 5.86 & 7.087-13 & 10.06\\
	  128 & 5.181-11 & 6.98 & 4.556-09 & 5.98 & 6.862-16 & 10.01\\
	  256 & 4.068-13 & 6.99 & 7.155-11 & 5.99 & 6.689-19 & 10.00\\
	  512 & 3.181-15 & 7.00 & 1.119-12 & 6.00 & 6.529-22 & 10.00\\
	 1024 & 2.486-17 & 7.00 & 1.749-14 & 6.00 & 6.375-25 & 10.00\\
	\cmidrule(lr){1-1}
	\cmidrule(lr){2-3}
	\cmidrule(lr){4-5}
	\cmidrule(lr){6-7}
	 theo &          & 7    &          & 6    &          & 10\\
	\bottomrule
	\end{tabular}
	\end{center}
	\end{table}
	The additionally presented results in
	the $\ell^\infty$ show also the predicted behavior. Note that all three
	error expressions show the optimal convergence orders which are 
	also obtained if exact integration is used and $\If$ is the identity operator.
	\end{case}
	\end{subtheorem}
	
	\subsection{Case group 4}
	This group of cases studies the superconvergence. Hence, we will restrict
	ourselves to cases where the convergence order in the $\ell^\infty$ norm
	suggested by~\eqref{Eq:ErrSup} is stricly greater than the convergence order
	in the $L^\infty$ norm given by~\eqref{Eq:ErrEst2}. We will show for this situation
	that the first two arguments in the minimum in~\eqref{Eq:ErrSup} and the first
	argument inside the maximum there can limit the $\ell^\infty$ convergence order.
	We remind that the limiting term will be written in boldface.
	\begin{subtheorem}{case} % cases 4a-4d
	\begin{case} % 4a
	\label{case4a}
	The choice
	\[
	P_{\IIhat} = \text{Gauss($6$)},
	\qquad
	P_{\Ifhat} = \text{Gauss--Lobatto($5$)}
	\]
	leads to
	\begin{align*}
	L^\infty\text{~order} & = \min\{7,\boldsymbol{6},8,\max\{12,\min\{6,5\}\}\} = 6,\\
	\ell^\infty\text{~order} & = \min\{10,\boldsymbol{8},\max\{12,\min\{6,5\}\}\} = 8,
	\end{align*}
	see~\eqref{Eq:ErrEst2} and~\eqref{Eq:ErrSup}. The convergence order in
	$\ell^\infty$ is bounded by the second argument of the minimum expression.
	\begin{table}[htb!]
	\begin{center}
	\caption{Errors and convergence orders in different norms for the cases of
	group 4.}
	\label{Tab:case4}
	\begin{tabular}{rclclclcl}
	\toprule
	 & \multicolumn{4}{c}{case 4a} &
	   \multicolumn{4}{c}{case 4b} \\
	\cmidrule(lr){2-5}
	\cmidrule(lr){6-9}
	$N$ & $\|e\|_{L^\infty}$ & ord & $\|e\|_{\ell^\infty}$ & ord & $\|e\|_{L^\infty}$ & ord & $\|e\|_{\ell^\infty}$ & ord\\
	\cmidrule(lr){1-1}
	\cmidrule(lr){2-3}
	\cmidrule(lr){4-5}
	\cmidrule(lr){6-7}
	\cmidrule(lr){8-9}
	   32 & 1.058-05 &      & 8.552-08 &      & 7.560-07 &      & 5.899-10 &      \\
	   64 & 1.717-07 & 5.94 & 3.348-10 & 8.00 & 6.529-09 & 6.86 & 5.610-13 & 10.04\\
	  128 & 2.835-09 & 5.92 & 1.310-12 & 8.00 & 5.183-11 & 6.98 & 5.452-16 & 10.01\\
	  256 & 4.464-11 & 5.99 & 5.120-15 & 8.00 & 4.068-13 & 6.99 & 5.318-19 & 10.00\\
	  512 & 6.981-13 & 6.00 & 2.005-17 & 8.00 & 3.181-15 & 7.00 & 5.192-22 & 10.00\\
	 1024 & 1.092-14 & 6.00 & 7.833-20 & 8.00 & 2.486-17 & 7.00 & 5.070-25 & 10.00\\
	\cmidrule(lr){1-1}
	\cmidrule(lr){2-3}
	\cmidrule(lr){4-5}
	\cmidrule(lr){6-7}
	\cmidrule(lr){8-9}
	 theo &          & 6    &          & 8    &          & 7    &          & 10\\
	\\[-1ex]
	 & \multicolumn{4}{c}{case 4c} &
	   \multicolumn{4}{c}{case 4d} \\
	\cmidrule(lr){2-5}
	\cmidrule(lr){6-9}
	$N$ & $\|e\|_{L^\infty}$ & ord & $\|e\|_{\ell^\infty}$ & ord & $\|e\|_{L^\infty}$ & ord & $\|e\|_{\ell^\infty}$ & ord\\
	\cmidrule(lr){1-1}
	\cmidrule(lr){2-3}
	\cmidrule(lr){4-5}
	\cmidrule(lr){6-7}
	\cmidrule(lr){8-9}
	   32 & 1.617-06 &      & 4.037-08 &      & 7.603-04 &      & 1.867-04 &     \\
	   64 & 1.387-08 & 6.86 & 1.654-10 & 7.93 & 4.955-05 & 3.94 & 1.132-05 & 4.04\\
	  128 & 1.101-10 & 6.98 & 6.522-13 & 7.99 & 3.219-06 & 3.94 & 7.017-07 & 4.01\\
	  256 & 8.654-13 & 6.99 & 2.565-15 & 7.99 & 2.028-07 & 3.99 & 4.377-08 & 4.00\\
	  512 & 6.759-15 & 7.00 & 1.002-17 & 8.00 & 1.268-08 & 4.00 & 2.735-09 & 4.00\\
	 1024 & 5.283-17 & 7.00 & 3.916-20 & 8.00 & 7.929-10 & 4.00 & 1.710-10 & 4.00\\
	\cmidrule(lr){1-1}
	\cmidrule(lr){2-3}
	\cmidrule(lr){4-5}
	\cmidrule(lr){6-7}
	\cmidrule(lr){8-9}
	 theo &          & 7    &          & 8    &          & 4    &          & 4\\
	\bottomrule
	\end{tabular}
	\end{center}
	\end{table}
	As viewable in Table~\ref{Tab:case4}, the expected convergence orders are
	obtained.
	\end{case}
	
	\begin{case} % 4b
	\label{case4b}
	Setting
	\[
	P_{\IIhat} = \text{Gauss($6$)},
	\qquad
	P_{\Ifhat} = \text{Gauss($6$)},
	\]
	the convergence orders
	\begin{align*}
	L^\infty\text{~order} &=
	\min\{\boldsymbol{7},\infty,\infty,\max\{12,\min\{6,\infty\}\}\} = 7,\\
	\ell^\infty\text{~order} &= \min\{\boldsymbol{10},\infty,\max\{12,\min\{6,\infty\}\}\} = 10 
	\end{align*}
	are expected by our theory. Hence, the convergence order in the
	$\ell^\infty$ norm is limited by the first term inside the minimum
	in~\eqref{Eq:ErrSup}. The numerical results coincide with our predictions.
	\end{case}
	
	\begin{case} % 4c
	\label{case4c}
	Taking 
	\[
	P_{\IIhat} = \text{Gauss($4$)},
	\qquad
	P_{\Ifhat} = \text{Gauss($4$)}
	\]
	provides
	\begin{align*}
	L^\infty\text{~order} &=
	\min\{\boldsymbol{7},\infty,\infty,\max\{8,\min\{6,\infty\}\}\} = 7,\\
	\ell^\infty\text{~order} &= \min\{10,\infty,\max\{\boldsymbol{8},\min\{6,\infty\}\}\} = 8,
	\end{align*}
	compare~\eqref{Eq:ErrEst2} and~\eqref{Eq:ErrSup}. The limitation of the
	$\ell^\infty$ convergence is caused by the first argument of the maximum
	in~\eqref{Eq:ErrSup}. The results in Table~\ref{Tab:case4} show clearly the
	superconvergence since the convergence order in the $\ell^\infty$ norm
	is one higher than the convergence order in the $L^\infty$ norm.
	\end{case}
	
	\begin{case} % 4d
	\label{case4d}
	Choosing
	\[
	P_{\IIhat} = \text{Gauss($6$)},
	\qquad
	P_{\Ifhat} = \text{Gauss($3$)},
	\]
	the estimates~\eqref{Eq:ErrEst2} and~\eqref{Eq:ErrSup} suggest
	\begin{align*}
	L^\infty\text{~order} &= \min\{7,\boldsymbol{4},6,\max\{12,\min\{6,3\}\}\} = 4,\\
	\ell^\infty\text{~order} &= \min\{10,\boldsymbol{6},\max\{12,\min\{6,3\}\}\} = 6.
	\end{align*}
	\end{case}
	However, since the uniform boundedness
	of $\sup_{t\in I}\|U^{(l)}(t)\|$ assumed in Theorem~\ref{th:superconvergence}
	cannot be ensured by Lemma~\ref{le:boundU} due to $\rIf+1 = 3 < 5 = r-1$,
	we actually do not expect any
	superconvergence. These expectations are confirmed by the numerical results
	given in Table~\ref{Tab:case4}. They show that for both the $L^\infty$ and
	the $\ell^\infty$ norm convergence order $4$ is obtained.
	\end{subtheorem}
	
	\subsection{Summary}
	The experimentally obtained and theoretically predicted
	convergence orders for all cases and all considered norms are collected
	in Table~\ref{Tab:allcases}. The experimental orders of
	convergence were calculated using the results obtained for $256$ and
	$512$ time steps.
	\begin{table}[htb!]
	\begin{center}
	\caption{Errors and convergence orders in different norms for all cases.}
	\label{Tab:allcases}
	\begin{tabular}{lrrrrrr}
	\toprule
	 & \multicolumn{2}{c}{$\|e\|_{L^\infty}$} &
	   \multicolumn{2}{c}{$\|e'\|_{L^\infty}$} &
	   \multicolumn{2}{c}{$\|e\|_{\ell^\infty}$} \\
	\cmidrule(lr){2-3}
	\cmidrule(lr){4-5}
	\cmidrule(lr){6-7}
	\multicolumn{1}{c}{case} & eoc & theo & eoc & theo & eoc & theo\\
	\cmidrule(lr){1-1}
	\cmidrule(lr){2-3}
	\cmidrule(lr){4-5}
	\cmidrule(lr){6-7}
	\ref{case1} & 3.04 & 3 & 3.00 & 3 & 3.01 & 3 \\
	\ref{case2a} & 6.00 & 5 & 5.00 & 5 & 6.00 & 5 \\
	\ref{case2a}* & 5.08 & 5 & 5.00 & 5 & 5.00 & 5 \\
	\ref{case2b} & 6.07 & 6 & 6.00 & 6 & 6.00 & 6 \\
	\ref{case2c} & 6.00 & 6 & 5.99 & 6 & 6.00 & 6 \\
	\ref{case3a} & 5.19 & 5 & 5.00 & 5 & 5.02 & 5 \\
	\ref{case3b} & 4.02 & 4 & 3.00 & 3 & 4.12 & 4 \\
	\ref{case3c} & 7.00 & 7 & 6.00 & 6 & 10.00 & 10 \\
	\ref{case4a} & 6.00 & 6 & 5.00 & 5 & 8.00 & 8 \\
	\ref{case4b} & 7.00 & 7 & 6.00 & 6 & 10.00 & 10 \\
	\ref{case4c} & 7.00 & 7 & 6.00 & 6 & 8.00 & 8 \\
	\ref{case4d} & 4.00 & 4 & 3.00 & 3 & 4.00 & 4 \\
	\bottomrule
	\end{tabular}
	\end{center}
	\end{table}

%%%%%%%%%%%%%%%%%%%%%%%%%%%%%%%%%%%%%%%%%%%%%%%%%%%%%%%%%%%%%%%%%%%%%%%%%%%%%%%
%%%%%%%%%%%%%%%%%%%%%%%%%%%%%%%%%%%%%%%%%%%%%%%%%%%%%%%%%%%%%%%%%%%%%%%%%%%%%%%

	\bibliographystyle{plain}
	\bibliography{vtd_ode_seuBiblio}

\begin{thebibliography}{10}

\bibitem{AMN11}
G.~Akrivis, Ch. Makridakis, and R.~H. Nochetto.
\newblock Galerkin and {R}unge-{K}utta methods: unified formulation, a
  posteriori error estimates and nodal superconvergence.
\newblock {\em Numer. Math.}, 118:429--456, 2011.

\bibitem{BM21}
S.~Becher and G.~Matthies.
\newblock Variational time discretizations of higher order and higher
  regularity.
\newblock {\em BIT Numer. Math.}, 2021.
\newblock \url{https://doi.org/10.1007/s10543-021-00851-6}.

\bibitem{BMW19}
S.~Becher, G.~Matthies, and D.~Wenzel.
\newblock {Variational Methods for Stable Time Discretization of First-Order
  Differential Equations}.
\newblock In K.~Georgiev, M.~Todorov, and Ivan G., editors, {\em {Advanced
  Computing in Industrial Mathematics: BGSIAM 2017}}, volume 793 of {\em
  {Studies in Computational Intelligence}}, pages 63--75, Cham, 2019. Springer
  International Publishing.

\bibitem{julia}
J.~Bezanson, A.~Edelman, S.~Karpinski, and V.~B. Shah.
\newblock Julia: a fresh approach to numerical computing.
\newblock {\em SIAM Rev.}, 59(1):65--98, 2017.

\bibitem{DHT81}
M.~Delfour, W.~Hager, and F.~Trochu.
\newblock Discontinuous {G}alerkin methods for ordinary differential equations.
\newblock {\em Math. Comp.}, 36(154):455--473, 1981.

\bibitem{DD86}
M.~C. Delfour and F.~Dubeau.
\newblock Discontinuous polynomial approximations in the theory of one-step,
  hybrid and multistep methods for nonlinear ordinary differential equations.
\newblock {\em Math. Comp.}, 47(175):169--189, 1986.

\bibitem{Emm99}
E.~Emmrich.
\newblock Discrete versions of {G}ronwall's lemma and their application to the
  numerical analysis of parabolic problems.
\newblock Preprint 637-1999, Preprint series of the Institute of Mathematics,
  Technische Universit\"{a}t Berlin, 1999.

\bibitem{ES02}
D.~Estep and A.~Stuart.
\newblock The dynamical behavior of the discontinuous {G}alerkin method and
  related difference schemes.
\newblock {\em Math. Comp}, 71(239):1075--1103, 2002.

\bibitem{HLW04}
E.~Hairer, Ch. Lubich, and G.~Wanner.
\newblock {\em Geometric {N}umerical {I}ntegration}.
\newblock Springer-Verlag, 2002.
\newblock Corrected 2nd printing 2004.

\bibitem{HNW08}
E.~Hairer, S.~P. N{\o}rsett, and G.~Wanner.
\newblock {\em Solving {O}rdinary {D}ifferential {E}quations {I}}.
\newblock Springer-Verlag, 2nd edition, 1993.
\newblock Corrected 3rd printing 2008.

\bibitem{Hen62}
P.~Henrici.
\newblock {\em Discrete variable methods in ordinary differential equations}.
\newblock Wiley, 1962.

\bibitem{Hul72b}
B.~L. Hulme.
\newblock Discrete {G}alerkin and related one-step methods for ordinary
  differential equations.
\newblock {\em Math. Comp.}, 26(120):881--891, 1972.

\bibitem{Hul72a}
B.~L. Hulme.
\newblock One-step piecewise polynomial {G}alerkin methods for initial value
  problems.
\newblock {\em Math. Comp.}, 26(118):415--426, 1972.

\bibitem{M00}
R.~L. Mishkov.
\newblock Generalization of the formula of {F}a\`{a} di {B}runo for a composite
  function with a vector argument.
\newblock {\em Internat. J. Math. \& Math. Sci.}, 24(7):481--491, 2000.

\bibitem{QV08}
A.~Quarteroni and A.~Valli.
\newblock {\em Numerical {A}pproximation of {P}artial {D}ifferential
  {E}quations}.
\newblock Springer Series in Computatioal Mathematics. Springer, Berlin, 2008.

\end{thebibliography}
	
\end{document}